\documentclass[10pt,letterpaper]{amsart}	
\usepackage[english]{babel} 		
\usepackage{fullpage}
\usepackage[
backend=biber,
style=alphabetic,
maxbibnames=99,
maxalphanames=5,
]{biblatex}
	
\usepackage{amsthm,amsfonts,amssymb,amsmath}

\usepackage{graphicx}				
\usepackage{hyperref}				
\usepackage{listings}
\usepackage[utf8]{inputenc}
\usepackage{algorithm}
\usepackage[noend]{algpseudocode}
\usepackage{csquotes}
\usepackage{ytableau}
\usepackage{textcmds}
\usepackage{enumitem}
\usepackage{tikz}
\usetikzlibrary{tikzmark,decorations.pathreplacing}
\usetikzlibrary{matrix}
\usepackage{comment}

\newtheorem{theorem}{Theorem}[section]
\newtheorem{lemma}[theorem]{Lemma}
\newtheorem{corollary}[theorem]{Corollary}
\newtheorem{conjecture}[theorem]{Conjecture}
\newtheorem{proposition}[theorem]{Proposition}
\theoremstyle{definition}
\newtheorem{definition}[theorem]{Definition}
\newtheorem{example}[theorem]{Example}

\newtheorem{remark}[theorem]{Remark}

\DeclareMathOperator{\im}{Im}
\DeclareMathOperator{\sign}{sign}
\DeclareMathOperator{\PT}{PT}
\DeclareMathOperator{\inc}{inc}
\DeclareMathOperator{\sh}{sh}
\DeclareMathOperator{\iv}{inv}
\DeclareMathOperator{\heigt}{ht}

\newcommand{\bN}{\mathbb{N}}

\newcommand{\bP}{\mathbb{Z}_+}

\newcommand{\bZ}{\mathbb{Z}}
\newcommand\numberthis{\addtocounter{equation}{1}\tag{\theequation}}
\def\inv{^{-1}}

\newlength{\cellsize}
\newcount\tableauRow
\newcount\tableauCol

\newenvironment{Tableau}[1]{%
  \tikzpicture[scale=.6,draw/.append style={thick,black},
                      baseline=(current bounding box.center)]
    \tableauRow=-1.5
    \foreach \Row in {#1} {
       \tableauCol=0.5
       \foreach\k in \Row {
         \draw[thin](\the\tableauCol,\the\tableauRow)rectangle++(1,1);
         \draw[thin](\the\tableauCol,\the\tableauRow)+(0.5,0.5)node{$\k$};
         \global\advance\tableauCol by 1
       }
       \global\advance\tableauRow by -1
    }
}{\endtikzpicture}

\begin{document}
\author[Wang]{Shiyun Wang}
\address[S.~Wang]{Department of Mathematics, University of Minnesota Twin Cities, MN, 55455}
\email{wang8406@umn.edu}
\title{$e$-positivity of chromatic quasisymmetric functions of bounce number three via the inverse Kostka matrix}
\date{\today}

\maketitle

\begin{abstract}
Let $G$ be the incomparability graph of a natural unit interval order $P$ on $[n]$ with the associated Dyck path $\mathbf{d}$. Let $X_G(\mathbf{x},q)$ be its chromatic quasisymmetric function. Shareshian and Wachs conjectured that $X_G(\mathbf{x},q)$ is $e$-positive. At $q=1$, this is the Stanley--Stembridge conjecture, which was proved by Hikita in 2024. Using the combinatorial interpretation of the inverse Kostka matrix by E\u{g}ecio\u{g}lu and Remmel, we expand $X_G(\mathbf{x},q)$ in the elementary basis when $\mathbf{d}$ has bounce number three. In our expansion, every coefficient is an alternating sum of at most seven generating functions of $P$-tableaux in terms of the number of inversions. We then construct sign reversing injections between sets of $P$-tableaux preserving inversions. We prove that the coefficient of $e_\lambda$ lies in $\bN[q]$ for every partition $\lambda=(\lambda_1,\lambda_2,\lambda_3)$ with $\lambda_2-\lambda_3\leq1$, and that $X_G(\mathbf{x},q)$ is $e$-positive whenever $P$ contains an antichain with $n-3$ elements, which extends the hook shape case of Cho and Huh.
\end{abstract}

\section{Introduction}
Chromatic symmetric functions were first introduced by Stanley \cite{STANLEY1995166} and have grown to be a fruitful area of research in relation to enumerative combinatorics, algebraic geometry, and representation theory. They can form numerous bases for the algebra of symmetric functions, $\Lambda= \oplus_n \Lambda^n$, and have a close relationship with Hessenberg varieties. Chromatic symmetric functions, as an algebraic tool to study graphs, are defined as follows:

Let $G$ be a finite graph with $V(G)=\{v_1, v_2, \cdots, v_n\}$ and $\mathbb{P}=\{1,2,3,\cdots\}$. The chromatic symmetric function $X_G(\mathbf{x})$ associated with $G$ is \[X_G(\mathbf{x})=X_G(x_1,x_2,\cdots)=\sum_{\kappa}x_{\kappa(v_1)}x_{\kappa(v_2)}\cdots x_{\kappa(v_n)},\] where the sum ranges over all proper colorings $\kappa: V \rightarrow \mathbb{P} $. 

Let $\lambda$ be the partition of the number of vertices of $G$, denoted by $\lambda\vdash |V|$. Let $\ell(\lambda)$ be the length of $\lambda$. Stanley \cite{STANLEY1995166} considered the expansion of $X_G(\mathbf{x})$ into various symmetric function bases. For example, the expansion of $X_G(\mathbf{x})$ in terms of the monomial symmetric functions $m_\lambda$ can be considered as a generating function for stable partitions into independent sets of the vertices of $G$. The coefficients of $X_G(\mathbf{x})$ in the expansion of the power sum symmetric functions $p_\lambda$ are related to the M\"obius function of the lattice of contractions of $G$. If the chromatic symmetric function is expanded into the elementary symmetric function $e_\lambda$ as $X_G(\mathbf{x})=\sum_\lambda c_\lambda e_\lambda$, then the summation of the coefficients $\sum_{\lambda\vdash |V|, \ell(\lambda)=j}c_\lambda$ is the number of acyclic orientations of $G$ with $j$ sinks. Further, if $X_G(\mathbf{x})$ can be written as a nonnegative linear combination of a basis $b$, then we say $X_G(\mathbf{x})$ is $b$-positive. Since $X_G(\mathbf{x})$ is $m$-positive and the image of $X_G(\mathbf{x})$ is $p$-positive under the automorphism $\omega$ in symmetric functions, it is natural to explore the combinatorial interpretation of the coefficients of $e_\lambda$. 

Let $P$ be a finite poset and $\inc(P)$ be its incomparability graph, that is, the graph $G$ with vertices the elements of $P$ where two vertices are connected by an edge if and only if the elements are incomparable in $P$. We say $P$ is $(\bf{a}+\bf{b})$-free if $P$ does not contain an induced subposet isomorphic to a disjoint union of an $\bf{a}$-chain and a $\bf{b}$-chain. The famous Stanley--Stembridge conjecture \cite{STANLEY1995166}\cite{STANLEY1993261}, proved by Hikita, is as follows.

\begin{theorem}\cite{Hikita2024}\label{3+1epositive}
If $P$ is $(\bf{3}+\bf{1})$-free, then $X_{\inc(P)}$ is $e$-positive.
\end{theorem}

Hikita's proof gave an explicit $e$-expansion of $X_{\inc(P)}(\mathbf{x},q)$ with coefficients the sums of rational functions of $q$, which are nonnegative for every real $q>0$. Shortly after that, another proof was given by Griffin, Mellit, Romero, Weigl and Wen \cite{GMRWW2025}. Before these groundbreaking results, some critical results related to this Stanley--Stembridge conjecture were discovered. Gasharov \cite{GASHAROV1996193} proved the incomparability graphs of $(\bf{3}+\bf{1})$-free posets are Schur-positive. In particular, $X_G(\mathbf{x})$ can be expanded as $X_G(\mathbf{x})=\sum_\lambda c_\lambda s_\lambda$, where $c_\lambda$ is the number of $P$-tableaux of shape $\lambda$ and $s_\lambda$ represents the Schur basis of symmetric functions. This result is important because it provides a new combinatorial object, which may help to find a combinatorial interpretation of the coefficients in the $e$-expansion. Guay-Paquet \cite{guaypaquet2013modular} reduced the conditions in Theorem \ref{3+1epositive} so that it is enough to prove the $e$-positivity of the incomparability graphs of natural unit interval orders, that is, posets which are both $(\bf{3}+\bf{1})$-free and $(\bf{2}+\bf{2})$-free.

Shareshian-Wachs \cite{SHARESHIAN2016497} introduced a quasisymmetric refinement of Stanley's chromatic symmetric function with a variable $q$ counting ascents in the colorings. For $G$ the incomparability graph of a natural unit interval order, this quasisymmetric function $X_G(\mathbf{x},q)$ turned out to be symmetric. Shareshian-Wachs generalized the expansion of $X_G(\mathbf{x},q)$ in terms of various symmetric function bases. For example, the coefficients $c_\lambda(q)$ of $X_G(\mathbf{x},q)$ in the expansion in terms of the $e_\lambda$ were generalized as $\sum_{\lambda\vdash |V|, \ell(\lambda)=j}c_\lambda(q)=\sum_{o\in O(G,j)}q^{\text{asc}(o)}$, where $O(G,j)$ is the set of acyclic orientations of $G$ with $j$ sinks and $\text{asc}(o)$ is the number of directed edges $(a,b)$ of $o$ such that $a<b$. The Schur-basis expansion also possessed its quasisymmetric version in \cite{SHARESHIAN2016497}, which can be reduced to the special case in \cite{GASHAROV1996193} with $q=1$. Furthermore, if $X_G(\mathbf{x},q)$ can be expanded into a basis $b$ such that all coefficients are in $\mathbb{N}[q]$, then we say $X_G(\mathbf{x},q)$ is $b$-positive. Shareshian-Wachs obtained the following version of the Stanley-Stembridge conjecture.

\begin{conjecture}\cite{SHARESHIAN2016497}\label{SWepositive}
Let $G$ be the incomparability graph of a natural unit interval order, then $X_G(\mathbf{x},q)$ is $e$-positive.
\end{conjecture}

Conjecture \ref{SWepositive} remains open, and so does the problem of finding a combinatorial interpretation of the coefficients. Before \cite{Hikita2024} and \cite{GMRWW2025}, both conjectures had been intensively studied for many special families of graphs, such as cycle graphs \cite{ALEXANDERSSON20183453}, complete graphs \cite{SHARESHIAN2016497}, lollipop graphs \cite{doi:10.1137/17M1144805}, and melting lollipop graphs \cite{HUH2020111728}. Since the natural unit interval orders on $[n]$ are in bijection with the Dyck paths on an $n\times n$ board (Proposition \ref{UIO char}), many proved cases are often described in terms of Dyck paths. A natural measure of their complexity is the \emph{bounce number $r$} of the Dyck path. This gives the longest chain in the poset. Cho and Hong \cite{Cho2022PositivityOC} proved Theorem \ref{3+1epositive} for bounce number at most three. Conjecture \ref{SWepositive} was proved for the Dyck paths of bounce number two (also called the \emph{abelian} case) in \cite{Cho2019OnEA, ALCO_2019__2_6_1059_0, AbreuNigro2021}, as well as for a hook-shape case with bounce number three in \cite{Cho2019OnEA}. Many results also focus on certain coefficients for an arbitrary bounce number $r$. The coefficient of $e_n$ lies in $\bN[q]$ was proved in \cite{SHARESHIAN2016497} and \cite{brucesagantalk}. Hwang \cite{Hwang2022} proved the cases of hook shapes and two-column shapes using heaps and noncommutative $P$-symmetric functions. Blasiak, Eriksson, Pylyavskyy and Siegl \cite{BEPS2025} obtained these coefficients again from noncommutative Schur functions for posets. Siegl \cite{Siegl2025} proved the same for the shapes $\lambda=(k,2^a)$. At $q=1$, Abreu and Nigro \cite{AbreuNigro2023} proved the nonnegativity of the coefficients of $e_\lambda$ for all $\lambda$ with length two.

In this paper, we focus on Conjecture \ref{SWepositive} for Dyck paths of bounce number three. We start with the Schur expansion of Shareshian and Wachs, that is, $X_G(\mathbf{x},q)=\sum_\mu B_\mu(q)s_\mu$, where $B_\mu(q)$ is the generating function of Gasharov's $P$-tableaux of shape $\mu$ by the number of inversions (Theorem \ref{quasi shur exp}). Then we further expand each $s_\mu$ in the elementary symmetric basis using the inverse Kostka numbers. E\u{g}ecio\u{g}lu and Remmel \cite{doi:10.1080/03081089008817966} gave a combinatorial interpretation of the inverse Kostka numbers in terms of special rim hook tabloids. Thus we have \[X_G(\mathbf{x},q)=\sum_{\mu}\sum_{\lambda}B_{\mu}(q)K\inv_{\lambda,\mu'}e_{\lambda}.\]
Enumerating the special rim hook tabloids with at most three rows, we show that the coefficient of $e_\lambda$ is an alternating sum of at most seven polynomials $B(\ell',j')$ (Proposition \ref{ecoeff}). To prove these coefficients are in $\bN[q]$, we construct sign reversing injections, which are injective maps from the $P$-tableaux of the shapes with a negative sign into the $P$-tableaux of the shapes with a positive sign preserving the number of inversions. Denote by $[e_\lambda]X_G(\mathbf{x},q)$ the coefficient of $e_\lambda$ in the $e$-expansion of $X_G(\mathbf{x},q)$. Our main results are the following.
 
\begin{enumerate}
\item (Theorems \ref{e_{(n-2l,l,l)}} and \ref{e_{(n-2l-1,l+1,l)}}) Let $\mathbf{d}$ be a Dyck path of bounce number three and $G=\inc(P(\mathbf{d}))$. Then $[e_\lambda]X_G(\mathbf{x},q)\in\bN[q]$ for every partition $\lambda=(\lambda_1,\lambda_2,\lambda_3)$ with $\lambda_2-\lambda_3\leq1$.

\item (Theorem \ref{epos1}) Let $\mathbf{d}$ be a Dyck path of bounce number three such that $P(\mathbf{d})$ contains an antichain with $n-3$ elements. Then $X_G(\mathbf{x},q)$ is $e$-positive.
\end{enumerate}

The partitions in Theorems \ref{e_{(n-2l,l,l)}} and \ref{e_{(n-2l-1,l+1,l)}} are $\lambda=(n-2\ell,\ell,\ell)$ and $\lambda=(n-2\ell-1,\ell+1,\ell)$ with $\ell\geq0$. The cases $\ell\leq2$ of Theorem \ref{e_{(n-2l,l,l)}} and $\ell=0$ of Theorem \ref{e_{(n-2l-1,l+1,l)}} are covered by the results of \cite{SHARESHIAN2016497,Hwang2022,BEPS2025,Siegl2025}. All the other cases are new (Remark \ref{knowncases}). Theorem \ref{epos1} extends the hook shape case of Cho and Huh \cite{Cho2019OnEA}. Among the six coefficients in the expansion of Section \ref{section4}, the coefficient of $e_{(n-3,3)}$ is the one whose positivity does not follow from the results above (Remark \ref{knowncases4}). Finally, we have a combinatorial interpretation of the coefficients in our positive expansion. Namely, each coefficient is the inversion generating function of an explicit set $\mathcal{T}_\lambda$ of $P$-tableaux of shape $\lambda$.

The paper is organized as follows. Section \ref{section2} gives the background on $P$-tableaux, Dyck paths and the inverse Kostka numbers in terms of special rim hook tabloids. In Section \ref{section3}, we derive the expansion of $X_G(\mathbf{x},q)$ for bounce number three in Proposition \ref{ecoeff} and prove Theorems \ref{e_{(n-2l,l,l)}} and \ref{e_{(n-2l-1,l+1,l)}}. In Section \ref{section4}, we reduce the expansion of $X_G(\mathbf{x},q)$ under the context of Theorem \ref{epos1}, and the rest of this section is devoted to the proof. Section \ref{section5} discusses on remaining coefficients for bounce number three and how our results relate to recent developments.

\section{Background} \label{section2}
This section introduces definitions and background on the chromatic quasisymmetric function and the inverse Kostka matrix needed for our results.

Following \cite{SHARESHIAN2016497}, for a finite graph $G=(V,E)$ with $V=\{v_1,\cdots,v_n\}$, the {\em chromatic quasisymmetric function} is defined as \[X_G(\mathbf{x}, q)=\sum_\kappa q^{\text{asc}(\kappa)}\mathbf{x}_\kappa,\]
where the sum ranges over all proper colorings $\kappa: V \rightarrow \bP, \mathbf{x}_\kappa:=\prod_{v\in V}x_{\kappa(v)},$ and $\text{asc}(\kappa):=|\{(v_i,v_j)\in E:i<j\,\, \text{and}\,\, \kappa(v_i)<\kappa(v_j)\}|$.

Note that $X_G(\mathbf{x}, q)$ is quasisymmetric in general. Shareshian-Wachs \cite{SHARESHIAN2016497} proved that if $G=\inc(P)$ for a natural unit interval order $P,$ then $X_G(\mathbf{x}, q)$ becomes symmetric. We only consider this symmetric case for $X_G(\mathbf{x}, q)$ in this paper without further comment.

\subsection{$P$-tableaux and the Schur expansion}

In \cite{GASHAROV1996193}, Gasharov defined the $P$-tableau to expand $X_G(\mathbf{x})$ in terms of the Schur basis $s_\lambda$. Let $a_{i,j}$ be the entry of a Young diagram at $i^{th}$ row and $j^{th}$ column. A {\em $P$-tableau of shape $\lambda$} is a filling of a Young diagram of shape $\lambda$ with elements of the poset $P=(P, \prec)$ satisfying the following.

\begin{enumerate}
    \item Each element in $P$ appears once in the diagram.
    \item $a_{i,j}\prec a_{i,j+1}$ for all $i,j$.
    \item $a_{i+1,j}\nprec a_{i,j}$ for all $i,j$.
\end{enumerate}

\begin{example}\label{ptabex}
Let $P$ be the poset on $[5]$ with relations
\[1\prec3,\quad 1\prec4,\quad 1\prec5,\quad 2\prec4,\quad 2\prec5,\quad 3\prec5.\]
The filling $\ytableausetup{mathmode, boxsize=1.6em}
T=\ \begin{ytableau} 1 & 3 & 5\\ 2 & 4 \end{ytableau}$ is a $P$-tableau of shape $(3,2)$. Its rows $1\prec3\prec5$ and $2\prec4$ are chains, and in the columns $2\nprec1$ and $4\nprec3$. Since the longest chain of $P$ has three elements, no $P$-tableau has a row of length greater than three. one checks that $P$ has exactly one $P$-tableau of shape $(3,2)$.
\end{example}

\begin{theorem}\cite{GASHAROV1996193}
 Let $P$ be a $(\bf{3}+\bf{1})$-free poset, and $G=\inc(P)$. If $X_G(\mathbf{x})=\sum_{\lambda}c_\lambda s_\lambda$, then $c_\lambda$ is the number of $P$-tableaux of shape $\lambda$.
\end{theorem}

This expansion has a version for $X_G(\mathbf{x},q)$. Let $\PT(G)$ be the set of all $P$-tableaux obtained from $G=\inc(P)$. Then for any $T\in \PT(G)$, an {\em inversion} $(i,j)$ is a pair of entries $i$ and $j$ of $T$ such that $i>j$ in the natural order on $[n]$, $i$ and $j$ are incomparable in $P$, and $i$ is in a strictly higher row of $T$ than $j$. Let $\iv(T)$ denote the number of inversions of $T$ and $\sh(T)$ be the shape of $T$. 

\begin{theorem}\cite{SHARESHIAN2016497} \label{quasi shur exp}
Let $P$ be a natural unit interval order and $G=\inc(P)$. Then \[X_G(\mathbf{x},q)=\sum_{T\in \PT(G)}q^{\iv(T)}s_{\sh(T)}.\]
\end{theorem}

Equivalently,

\begin{equation} \label{schur exp}
X_G(\mathbf{x},q)=\sum_{\lambda}B_{\lambda}(q)s_{\lambda},\,\,\, \text{where}\,\, B_{\lambda}(q):=\displaystyle\sum_{\substack{T\in \PT(G)\\ \sh(T)=\lambda}}q^{\iv(T)}.
\end{equation}

In particular $B_\lambda(q)\in\bN[q]$, and $B_\lambda(q)=0$ unless there exists a $P$-tableau of shape $\lambda$ in $\PT(G)$.

\begin{example}
    For the $P$-tableau $T$ of Example \ref{ptabex}, the inversions are $(3,2)$ and $(5,4)$. Hence $\iv(T)=2$. Computing the inversions of all $P$-tableaux for Example \ref{ptabex}, for the path graph $G=1-2-3-4-5$,
\begin{align*}
X_G(\mathbf{x},q)={}&q^2s_{(3,2)}+q^2s_{(3,1,1)}+(2q+5q^2+2q^3)\,s_{(2,2,1)}\\
&+(3q+6q^2+3q^3)\,s_{(2,1,1,1)}+(1+q)^4\,s_{(1^5)}.
\end{align*}
\end{example}

\subsection{Dyck paths and bounce paths}
Recall that we introduced the natural unit interval orders as posets both $(\bf{3}+\bf{1})$-free and $(\bf{2}+\bf{2})$-free. Shareshian-Wachs \cite{SHARESHIAN2016497} provided a helpful characterization of this family of posets.

Let $\mathbf{d}=(d_1,d_2,\cdots,d_n)$ be a sequence of positive integers satisfying $d_1\leq d_2\leq\cdots\leq d_n\leq n$ and $d_i\geq i$ for all $i$. The last two conditions force $d_n=n.$ $\mathbf{d}$ is also called a {\em Dyck path} on an $n\times n$ board, taking a staircase walk from the bottom left corner $(0,0)$ to the top right corner $(n,n)$ that lies above the diagonal. It is connecting \[(0,0),(0,d_1),(1,d_1),(1,d_2),(2,d_2),(2,d_3),\cdots,(n-2,d_{n-1}),(n-1,d_{n-1}),(n-1,d_n),(n,d_n)\] with either horizontal or vertical steps.

For any Dyck path $\mathbf{d}$, we can obtain a poset $P(\mathbf{d})$ associated with $\mathbf{d}$. The poset relations in $P(\mathbf{d})$ are given by $i\prec j$ if $i<n$ and $j\in \{d_i+1,d_i+2,\cdots,n\}$. Equivalently, the poset relations $i\prec j$ in $P(\mathbf{d})$ are given by all cells with coordinates $(i,j)$ that lie above $\mathbf{d}$. Denote the Young diagram of all cells above $\mathbf{d}$ by $\tau$. Read by rows from the top and
\[\tau_i=\#\{\,j\in[n]\;:\;d_j<n+1-i\,\},\qquad i\in[n],\]
so that $\tau$ has exactly $n-d_1$ nonzero parts.

A bijection exists between the natural unit interval orders and the Dyck paths.

\begin{proposition}\cite{SHARESHIAN2016497}\label{UIO char}
Let $P$ be a poset on $[n]$, then $P$ is a natural unit interval order iff $P=P(\mathbf{d})$ for some $\mathbf{d}=(d_1,d_2,\cdots,d_n)$.
\end{proposition}

Thus, given a natural unit interval order $P$ on $[n]$, we can obtain its corresponding Dyck path $\mathbf{d}=(d_1,d_2,\cdots,d_n)$ on an $n\times n$ board, the partition $\tau$ determined by $\mathbf{d}$, and $G=\inc(P)$. Set a sequence $b$ recursively such that $b_0:=0$ and \[b_i:=d_{b_{i-1}+1},\qquad i\geq 1.\] Since $d_j\geq j$ for all $j$, we have $b_i\geq b_{i-1}+1$, so the sequence $b$ is strictly increasing to $n$ since it is bounded by $n$ with $d_n=n$. Let $r$ be the least index with $b_r=n$. Abusing notation slightly, we write $b=(b_1,b_2,\dots,b_r)$.

Let $b(\mathbf{d})$ be the {\em bounce path} connecting
\[(0,0),(0,b_1),(b_1,b_1),(b_1,b_2), (b_2,b_2), \cdots,(b_{r-1},b_r),(b_r,b_r)=(n,n).\] We call $r$ the \emph{bounce number}, that is, the number of times the bounce path hits the diagonal including the endpoint $(n,n)$.

For $i\in[r]$, let
\[S_i:=\{b_{i-1}+1,\,b_{i-1}+2,\,\dots,\,b_i\},\qquad n_i:=|S_i|=b_i-b_{i-1}\] Note that $S_1,\dots,S_r$ partition $[n]$ and $n=n_1+\cdots+n_r$. For notational convenience, we let \[k:=n_2+n_3+\cdots+n_r=n-b_1.\] Here $k$ is exactly the number of nonzero parts of $\tau$ described above.

Figure \ref{b(d)=3} illustrates an example of a natural unit interval order $P=P(\mathbf{d})$ for Dyck path $\mathbf{d}=(4,6,6,6,6,7,8,8)$.

\begin{figure}[h]
    \centering
    \vspace{3mm}
\begin{tikzpicture}[every node/.style={minimum size=.4cm-\pgflinewidth, outer sep=0pt}]
\draw[step=0.4cm] (-1.6,-1.6) grid (1.6,1.6);
\draw (-1.6,-1.6) -- (1.6,1.6);
\draw [red,very thick] (-1.6,-1.6)--(-1.6,0)--(-1.2,0)--(-1.2,+.8)--(+0.4,+.8)--(0.4,1.2)--(.8,1.2)--(.8,1.6)--(1.6,1.6);
\draw [green, ultra thick, dashed] (-1.6,-1.6)--(-1.6,0)--(0,0)--(0,.8)--(.8,.8)--(.8,1.6)--(1.6,1.6);
    \node[fill=yellow!30] at (-1.4,1.4) {};
    \node[fill=yellow!30] at (-1,1.4) {};
    \node[fill=yellow!30] at (-.6,1.4) {};
    \node[fill=yellow!30] at (-.2,1.4) {};
    \node[fill=yellow!30] at (.2,1.4) {};
    \node[fill=yellow!30] at (.6,1.4) {};
    \node[fill=yellow!30] at (-1.4,1) {};
    \node[fill=yellow!30] at (-1,1) {};
    \node[fill=yellow!30] at (-.6,1) {};
    \node[fill=yellow!30] at (-.2,1) {};
    \node[fill=yellow!30] at (.2,1) {};
    \node[fill=yellow!30] at (-1.4,.6) {};
    \node[fill=yellow!30] at (-1.4,.2) {};
    \draw[decorate,decoration={brace,mirror},thick,blue] 
 (-1.8,1.6)--(-1.8,.8) node[midway,left]{$n_3$};
    \draw[decorate,decoration={brace,mirror},thick,blue]
    (-1.8,.8)--(-1.8,0) node[midway,left]{$n_2$};
    \draw[decorate,decoration={brace,mirror},thick,blue]
    (-1.6,-1.8)--(0,-1.8) node[midway,below]{$n_1$};
\end{tikzpicture}
\qquad
\begin{tikzpicture}
    \node[label=above:$G:$] (1) at (0,0) {1};
    \node (2) at (1,0) {2};
    \node (3) at (2,0) {3};
    \node (4) at (3,0) {4};
    \node (5) at (4,0) {5};
    \node (6) at (5,0) {6};
    \node (7) at (6,0) {7};
    \node (8) at (7,0) {8};
    \path (1) edge node {} (2);
    \path (1) edge[bend right] node {} (3);
    \path (1) edge[bend right] node {} (4);
    \path (2) edge node {} (3);
    \path (2) edge[bend right] node {} (4);
    \path (2) edge[bend right] node {} (5);
    \path (2) edge[bend right] node {} (6);
    \path (3) edge node {} (4); 
    \path (3) edge[bend right] node {} (5);
    \path (3) edge[bend right] node {} (6);
    \path (4) edge node {} (5); 
    \path (4) edge[bend right] node {} (6);
    \path (5) edge node {} (6); 
    \path (6) edge node {} (7); 
    \path (7) edge node {} (8); 
\end{tikzpicture}
\qquad
\caption{Left: $P=P(\mathbf{d})$ for $\mathbf{d}=(4,6,6,6,6,7,8,8)$ with the corresponding partition $\tau=(6511)$. $b(\mathbf{d})=(0,0),(0,4),(4,4),(4,6),(6,6),(6,8),(8,8)$. $b=(4,6,8)$. $S_1=\{1,2,3,4\}$; $S_2=\{5,6\}$; $S_3=\{7,8\}$. $(n_1,n_2,n_3)=(4,2,2)$ and $k=4=n-b_1$. Right: $G=\inc(P)$.}
    \label{b(d)=3}
\end{figure}

\subsection{Special rim hook tabloids}
Following \cite{doi:10.1080/03081089008817966}, a {\em rim hook} of size $t$ in a partition $\mu$ is a connected skew shape $\mu/\nu$ consisting of $t$ cells and containing no $2\times 2$ blocks of cells. Its \emph{height} $\heigt(H)$ is the number of rows it meets.

A {\em special rim hook} $H$ is a rim hook with at least one cell in the first column. Define the sign of $H$ as $\sign(H)=(-1)^{\heigt(H)-1}$.

A {\em special rim hook tabloid} $U$ of shape $\mu$ and type $\lambda=(\lambda_1,\cdots,\lambda_p)$ is a tiling of the Young diagram of shape $\mu$  with special rim hooks of sizes $\lambda_1,\cdots,\lambda_p$, such that the hooks can be removed one at a time from the outside so that a Young diagram remains at each stage. Define the sign of $U$ as $\sign(U)=\prod_{H}\sign(H)$.

\begin{example} \label{special rim hook}
Let $U$ be a special rim hook tabloid of shape $\mu=(4333)$ and type $\lambda=(5431)$. Then we can obtain two such possible $U$ as shown in Figure \ref{SRHT}.
\end{example}

\begin{figure}[h]
    \centering
    \vspace{5mm}
  \begin{Tableau} 
  {{ ,,,},{,,},{,,},{,,}}
      \draw[very thick,blue] (0.5,-3.5)--(2.5,-3.5);
      \draw[very thick,blue]
      (0.5,-2.5)--(2.5,-2.5)--(2.5,-1.5);
      \draw[very thick,blue]
      (0.5,-1.5)--(1.5,-1.5)--(1.5,-0.5)--(3.5,-0.5);
      \fill[blue] (0.5,-0.5) circle(.1);
  \end{Tableau}
\qquad
  \begin{Tableau}
   {{ ,,,},{,,},{,,},{,,}}
      \draw[very thick,blue] (0.5,-3.5)--(2.5,-3.5)--(2.5,-1.5);
      \draw[very thick,blue]
      (0.5,-2.5)--(1.5,-2.5)--(1.5,-1.5);
      \fill[blue] (0.5,-1.5) circle(.1);
      \draw[very thick,blue]
      (0.5,-0.5)--(3.5,-0.5);
  \end{Tableau}
  \caption{Special rim hook tabloids $U_1$ (left) and $U_2$ (right).}
    \label{SRHT}
\end{figure}

Hence $\sign(U_1)=1\cdot(-1)\cdot(-1)\cdot1=1$, and $\sign(U_2)=1\cdot1\cdot(-1)\cdot1=-1$.

\subsection{The inverse Kostka matrix}
Let $K_{\lambda,\mu}$ be the Kostka number enumerating {\em semistandard Young tableaux} (SSYT) of shape $\lambda$ and type $\mu$. Also let $m_\lambda$ and $h_\lambda$ denote the monomial symmetric functions and the complete homogeneous symmetric functions respectively. We have $s_{\lambda}=\sum_{\nu}K_{\lambda,\nu}m_{\nu}$. Since $\langle h_{\mu}, m_{\nu}\rangle =\delta_{\mu\nu}$ for the symmetric functions scalar product $\langle \, , \,\rangle$, we can obtain that \[h_{\mu}=\sum_{\lambda}K_{\lambda,\mu}s_{\lambda}.\] Since both $s_\lambda$ and $h_\lambda$ form bases for the symmetric functions, we can invert the Kostka matrix to obtain

\begin{equation}\label{eqnstoh}
s_{\mu}=\sum_{\lambda}K\inv_{\lambda,\mu}h_{\lambda},
\end{equation}
where $K\inv$ are the inverse Kostka number. Recall that in symmetric functions, there is an involution $\omega: \Lambda \rightarrow \Lambda$ such that $\omega(e_\lambda)=h_{\lambda}$ and $\omega(s_\mu)=s_{\mu'}$, where $\mu'$ denotes the conjugate partition. Applying $\omega$ to Equation \ref{eqnstoh} gives $s_{\mu'}=\sum_{\lambda}K\inv_{\lambda,\mu}e_{\lambda},$ and replacing $\mu$ by $\mu'$ yields

\begin{equation} \label{inv K exp}
s_{\mu}=\sum_{\lambda}K\inv_{\lambda,\mu'}e_{\lambda},
\end{equation}

E\u{g}ecio\u{g}lu-Remmel \cite{doi:10.1080/03081089008817966} introduced a combinatorial interpretation of the inverse of the Kostka matrix so that we can compute the inverse Kostka number as follows.

\begin{theorem}\cite{doi:10.1080/03081089008817966} \label{inv Kostka}
Let $\lambda, \mu\vdash n.$ Then
\[K\inv_{\lambda, \mu}=\sum_{U}\sign(U),\] where the summation is over all special rim hook tabloids of type $\lambda$ and shape $\mu$.
\end{theorem}

\begin{example}
In Example \ref{special rim hook}, $K\inv_{(5431), (4333)}=\sign(U_1)+\sign(U_2)=0$.
\end{example}

\begin{remark}\label{ER convention}
In Theorem \ref{inv Kostka} the tabloid has shape $\mu$ (the second index) and type $\lambda$ (the first index). This is the transpose of the convention used in \cite{doi:10.1080/03081089008817966}, and it is the one that makes Equation \ref{inv K exp} correct.
\end{remark}

\section{The $e$-expansion for bounce number three}\label{section3}

\subsection{The expansion via inverse Kostka numbers}

 In an $n\times n$ board, let $P$ be a natural unit interval order and $G=\inc(P)$. Then $P=P(\mathbf{d})$ for $\mathbf{d}$ a Dyck path of bounce number $r$. By Theorem \ref{quasi shur exp} and Equation \eqref{inv K exp}, we can write the associated chromatic quasisymmetric function as

\begin{equation} \label{inv kostka expand}
X_G(\textbf{x},q)=\sum_{\mu}\sum_{\lambda}B_{\mu}(q)K\inv_{\lambda,\mu'}e_{\lambda}.
\end{equation}

The next lemma bounds which shapes $\mu$ can occur in Equation (\ref{inv kostka expand}).

\begin{lemma} \label{remark 1}
Each $S_i$ is an antichain in $P$. Consequently, every chain in $P$ has at most $r$ elements. Further, the coefficient $B_\mu(q)$ in Equation (\ref{inv kostka expand}) is nonzero only if $\mu_1\leq r$.
\end{lemma}

\begin{proof}
Let $x,y\in S_i$ with $x<y$, so that $b_{i-1}+1\leq x<y\leq b_i$. Since $\mathbf{d}$ is weakly increasing, $d_x\geq d_{b_{i-1}+1}=b_i\geq y$, so $x\nprec y$. On the other hand $y\nprec x$, since $y\prec x$ would force $x>d_y\geq y$, contradicting $x<y.$ Thus $S_i$ is an antichain. Therefore, a chain in $P$ contains at most one element from each of $S_1,\dots,S_r$, and it has at most $r$ elements. Finally, if $T$ is a $P$-tableau of shape $\mu$, then the first row of $T$ is a chain with $\mu_1$ elements, so $T$ exists only if $\mu_1\leq r$, and $B_\mu(q)=0$ unless $\mu_1\leq r.$
\end{proof}

From now on and throughout the paper, we consider $X_G(\textbf{x},q)$ when $r=3$, and we keep the notations $S_1,S_2,S_3,n_1,n_2,n_3,k$ of Section \ref{section2}.

\begin{lemma}\label{P exist}
Let $P$ be a natural unit interval order on $[n]$ such that $P\simeq P(\mathbf{d})$ for a Dyck path $\mathbf{d}$ with $r=3$. For any $\ell$ and $j$, if a $P$-tableau $T$ with $\sh(T)=3^\ell2^j1^{n-3\ell-2j}$ exists, i.e., $B_{3^\ell2^j1^{n-3\ell-2j}}(q)\neq 0$, then $2\ell+j\leq k$ and $\ell\leq \min\{n_1,n_2,n_3\}$.
\end{lemma}

\begin{proof}
Using Lemma \ref{remark 1}, it is straightforward to see that every entry of $T$ in the second or third column lies in $S_2\cup S_3.$ These entries are distinct and there are $2\ell+j$ of them, so $2\ell+j\leq n_2+n_3=k.$ For the second bound, each of the $\ell$ rows of length three is a chain with three distinct elements. since each $S_i$ is an antichain, so a chain of length three contains exactly one element from each of $S_1,S_2,S_3.$ Thus $\ell\leq \min\{n_1,n_2,n_3\}$.
\end{proof}

For $\ell, j\in \bZ$, we write $B(\ell,j):=B_{3^\ell2^j1^{n-3\ell-2j}}(q),$ with the convention that $B(\ell,j)=0$ when $\ell<0,$ or $j<0,$ or $3\ell+2j>n.$ By Lemma \ref{P exist}, $B(\ell,j)=0$ unless $0\leq \ell\leq \min\{n_1,n_2,n_3\}$ and $0\leq j\leq k-2\ell.$ Using this fact, we can reduce Equation (\ref{inv kostka expand}) by finding all nonzero terms. The conjugate of $\mu$ is $\mu'=(n-2\ell-j, \ell+j, \ell)$, which is a partition with at most three parts. Thus only special rim hook tabloids whose shape has at most three rows occur in Equation (\ref{inv kostka expand}). Now we study these tabloids in terms of their shape.

\begin{lemma}\label{six types}
Let $\nu=(\nu_1,\nu_2,\nu_3)$ be a partition of $n$ with at most three parts. Every special rim hook tabloid of shape $\nu$ has one of the following six types, each occurring at most once. The superscripts represent their sign.
\[
\begin{array}{lll}
1:\ (\nu_1, \nu_2, \nu_3)^{+} &\quad 2:\ (\nu_1+2, \nu_2-1, \nu_3-1)^{+} &\quad 3:\ (\nu_1+1, \nu_2+1, \nu_3-2)^{+}\\[3pt]
4:\ (\nu_1, \nu_2+1, \nu_3-1)^{-} &\quad 5:\ (\nu_1+1, \nu_2-1, \nu_3)^{-} &\quad 6:\ (\nu_1+2, \nu_2, \nu_3-2)^{-}
\end{array}
\]
Therefore, for a partition $\lambda\vdash n$, $K\inv_{\lambda,\nu}$ is the sum of the signs of those of the six triples whose entries rearrange to $\lambda$.
\end{lemma}
 
\begin{proof}
The first column of $\nu$ has at most three cells and each special rim hook contains exactly one of them, so the tabloid is determined by how far the hook travels along the rim. The bottom hook could end in row $3$, $2$ or $1$. Similarly, the second hook ends in its own row or in the row above. Enumerating all possible cases gives the six configurations shown in Figure \ref{6tabloids}, and summing $\heigt(H)-1$ over the hooks of each type gives the corresponding signs. The computation of $K\inv_{\lambda,\nu}$ is due to Theorem \ref{inv Kostka}.
\end{proof}

\begin{remark}\label{sixtypesremark} We address two points here for Lemma \ref{six types}.
\begin{itemize}
    \item[(1)] The six triples in Lemma \ref{six types} are compositions, and a triple with a negative entry does not occur. A triple with a zero entry is read as a partition with fewer parts after rearrangement.
    \item[(2)] Two of the six triples may rearrange to the same partition, and their signs add. One can directly compare the six triples, and see that only when $\nu_1=\nu_2=\nu_3$, the two triples $(\nu_1,\nu_2+1,\nu_3-1)$ and $(\nu_1+1,\nu_2-1,\nu_3)$ both rearrange to $(\nu_1+1,\nu_1,\nu_1-1)$. First half of Example \ref{worked} illustrates this situation.
\end{itemize}
\end{remark}

Following our discussion on $B(\ell,j)=B_{3^\ell2^j1^{n-3\ell-2j}}(q),$ the six tabloids in Figure \ref{6tabloids} are the six configurations of Lemma \ref{six types}, drawn for $\nu=\mu'=(n-2\ell-j,\ell+j,\ell)$. In addition, specializing to $\ell=0$ recovers the two-row case, where only type $1$ and type $5$ have no negative entry, yielding $K\inv_{(n-j,j),(n-j,j)}=1$ and $K\inv_{(n-j+1,j-1),(n-j,j)}=-1$. Further specializing to $j=0$ gives type $1$ left, and $K\inv_{(n),(n)}=1.$

\begin{figure}[h]
 \centering
\begin{tikzpicture}[x=0.55pt,y=0.55pt,yscale=-1,xscale=1]
 
 
\draw   (100,100) -- (240,100) -- (240,110) -- (100,110) -- cycle ;
 
\draw   (200,100) -- (200,120) -- (100,120) -- (100,110);
 
\draw   (160,100) -- (160,130) -- (100,130) -- (100,120);
\draw[blue, very thick] (105,125) -- (155,125);
\draw[blue, very thick] (105,105) -- (235,105);
\draw[blue, very thick] (105,115) -- (195,115);
\draw[decorate,decoration={brace},orange]
    (100,95)--(159,95) node[midway,above]{$\ell$};
\draw[decorate,decoration={brace},orange]
    (161,95)--(200,95) node[midway,above]{$j$};
\end{tikzpicture}
\qquad
\hspace{2cm}
\begin{tikzpicture}[x=0.55pt,y=0.55pt,yscale=-1,xscale=1]
 
 
\draw   (100,100) -- (240,100) -- (240,110) -- (100,110) -- cycle ;
 
\draw   (200,100) -- (200,120) -- (100,120) -- (100,110);
 
\draw   (160,100) -- (160,130) -- (100,130) -- (100,120);
\draw[blue, very thick] (105,125) -- (155,125)-- (155,115)--(195,115)--(195,105)--(235, 105);
\draw[blue, very thick] (105,105) -- (185,105);
\draw[blue, very thick] (105,115) -- (145,115);
\draw[decorate,decoration={brace},orange]
    (100,95)--(159,95) node[midway,above]{$\ell$};
\draw[decorate,decoration={brace},orange]
    (161,95)--(200,95) node[midway,above]{$j$};
\end{tikzpicture} 
\qquad
\hspace{2cm}
\begin{tikzpicture}[x=0.55pt,y=0.55pt,yscale=-1,xscale=1]
 
 
\draw   (100,100) -- (240,100) -- (240,110) -- (100,110) -- cycle ;
 
\draw   (200,100) -- (200,120) -- (100,120) -- (100,110);
 
\draw   (160,100) -- (160,130) -- (100,130) -- (100,120);
\draw[blue, very thick] (105,125) -- (155,125)-- (155,115)--(195,115);
\draw[blue, very thick] (105,105) -- (135,105);
\draw[blue, very thick] (105,115) -- (145,115)--(145,105)--(235,105);
\draw[decorate,decoration={brace},orange]
    (100,95)--(159,95) node[midway,above]{$\ell$};
\draw[decorate,decoration={brace},orange]
    (161,95)--(200,95) node[midway,above]{$j$};
\end{tikzpicture} 
 
{\footnotesize $1:\ (n{-}2\ell{-}j,\,\ell{+}j,\,\ell)^{+}$};
\qquad
{\footnotesize $2:\ (n{-}2\ell{-}j{+}2,\,\ell{+}j{-}1,\,\ell{-}1)^{+}$};
\qquad
{\footnotesize $3:\ (n{-}2\ell{-}j{+}1,\,\ell{+}j{+}1,\,\ell{-}2)^{+}$};
 
\begin{tikzpicture}[x=0.55pt,y=0.55pt,yscale=-1,xscale=1]
 
 
\draw   (100,100) -- (240,100) -- (240,110) -- (100,110) -- cycle ;
 
\draw   (200,100) -- (200,120) -- (100,120) -- (100,110);
 
\draw   (160,100) -- (160,130) -- (100,130) -- (100,120);
\draw[blue, very thick] (105,125) -- (155,125)-- (155,115)--(195,115);
\draw[blue, very thick] (105,105) -- (235,105);
\draw[blue, very thick] (105,115) -- (145,115);
\draw[decorate,decoration={brace},orange]
    (100,95)--(159,95) node[midway,above]{$\ell$};
\draw[decorate,decoration={brace},orange]
    (161,95)--(200,95) node[midway,above]{$j$};
\end{tikzpicture} 
\qquad
\hspace{2cm}
\begin{tikzpicture}[x=0.55pt,y=0.55pt,yscale=-1,xscale=1]
 
 
\draw   (100,100) -- (240,100) -- (240,110) -- (100,110) -- cycle ;
 
\draw   (200,100) -- (200,120) -- (100,120) -- (100,110);
 
\draw   (160,100) -- (160,130) -- (100,130) -- (100,120);
\draw[blue, very thick] (105,125) -- (155,125);
\draw[blue, very thick] (105,105) -- (185,105);
\draw[blue, very thick] (105,115) -- (195,115)--(195,105)--(235,105);
\draw[decorate,decoration={brace},orange]
    (100,95)--(159,95) node[midway,above]{$\ell$};
\draw[decorate,decoration={brace},orange]
    (161,95)--(200,95) node[midway,above]{$j$};
\end{tikzpicture} 
\qquad
\hspace{2cm}
\begin{tikzpicture}[x=0.55pt,y=0.55pt,yscale=-1,xscale=1]
 
 
\draw   (100,100) -- (240,100) -- (240,110) -- (100,110) -- cycle ;
 
\draw   (200,100) -- (200,120) -- (100,120) -- (100,110);
 
\draw   (160,100) -- (160,130) -- (100,130) -- (100,120);
\draw[blue, very thick] (105,125) -- (155,125)-- (155,115)--(195,115)--(195,105)--(235, 105);
\draw[blue, very thick] (105,105) -- (135,105);
\draw[blue, very thick] (105,115) -- (145,115)--(145,105)--(185,105);
\draw[decorate,decoration={brace},orange]
    (100,95)--(159,95) node[midway,above]{$\ell$};
\draw[decorate,decoration={brace},orange]
    (161,95)--(200,95) node[midway,above]{$j$};
\end{tikzpicture} 
 
{\footnotesize $4:\ (n{-}2\ell{-}j,\,\ell{+}j{+}1,\,\ell{-}1)^{-}$};
\qquad
{\footnotesize $5:\ (n{-}2\ell{-}j{+}1,\,\ell{+}j{-}1,\,\ell)^{-}$};
\qquad
{\footnotesize $6:\ (n{-}2\ell{-}j{+}2,\,\ell{+}j,\,\ell{-}2)^{-}$}.
 
\caption{The six configurations of Lemma \ref{six types} for the shape $\mu'=(n-2\ell-j,\ell+j,\ell)$. Each type is labeled by its number in Lemma \ref{six types}. The superscripts represent their sign.}
\label{6tabloids}
\end{figure}

\begin{proposition}\label{expansionprop}
Let $P$ be a natural unit interval order on $[n]$ with $r=3, k=n_2+n_3$ and $G=\inc(P)$. Then $X_G(\mathbf{x},q)$ expands as
\begin{align*}
\numberthis \label{expansion}
X_G(\mathbf{x},q)=\sum_{\ell=0}^{\min\{n_1,n_2,n_3\}}\Bigg[
&\sum_{j=2}^{k-2\ell}\Big(B(\ell,j)+B(\ell+1,j)+B(\ell+2,j-3)-B(\ell+1,j-2)\\
&\qquad\qquad -B(\ell,j+1)-B(\ell+2,j-2)\Big)\,e_{(n-2\ell-j,\,\ell+j,\,\ell)}\\
&+\Big(B(\ell,0)+B(\ell+1,0)-B(\ell,1)\Big)\,e_{(n-2\ell,\,\ell,\,\ell)}\\
&+\Big(B(\ell,1)+B(\ell+1,1)-B(\ell,2)-B(\ell+1,0)\Big)\,e_{(n-2\ell-1,\,\ell+1,\,\ell)}\Bigg].
\end{align*}
\end{proposition}
 
\begin{proof}
Using Equation (\ref{schur exp}),
\[X_G(\mathbf{x},q)=\sum_{\ell=0}^{\min\{n_1,n_2,n_3\}}\sum_{j=0}^{k-2\ell}B(\ell,j)\,s_{3^\ell2^j1^{n-3\ell-2j}}
=\sum_{\ell}\sum_{j}B(\ell,j)\sum_{\lambda}K\inv_{\lambda,(n-2\ell-j,\ell+j,\ell)}e_{\lambda}.\]
The value of $j$ is the number of rows of length two in $P$-tableaux, which is up to $k-2\ell$. This is because the entries of a $P$-tableau in the second and third column are chosen from $S_2$ and $S_3$ with a total size of $k,$ and $2\ell$ is the size of all entries in rows of length three except the first column.
 
Next we collect the shapes that contribute to each index. We take a closer look at the six types of Lemma \ref{six types} evaluated at $\nu=\mu'=(n-2\ell'-j',\ell'+j',\ell')$. Each type is actually of the form $(n-2\ell-j,\ell+j,\ell)$ for a unique pair $(\ell,j)$ as follows.

\[
\begin{array}{c|cccccc}
\text{type} & 1 & 2 & 3 & 4 & 5 & 6\\\hline
(\ell',\,j') & (\ell,\,j) & (\ell+1,\,j) & (\ell+2,\,j-3) & (\ell+1,\,j-2) & (\ell,\,j+1) & (\ell+2,\,j-2)\\
\text{sign} & + & + & + & - & - & -
\end{array}
\]
For example, we let $\ell'=\ell+1, j'=j$ for type 2, and we find $(n-2\ell'-j'+2, \ell'+j'-1, \ell'-1)^+$ becomes exactly the form $(n-2\ell-j, \ell+j, \ell).$ This means every type of $\mu'=(n-2\ell'-j',\ell'+j',\ell')$ can be written with the index $\ell,j$ of another shape. Note that the pairs with index $\ell, j$ in the table provide exactly the first coefficient of Proposition \ref{expansionprop}.

For $j\geq 3$, all six terms are present in the first coefficient of Proposition \ref{expansionprop}. For $j=2, B(\ell+2,j-3)=0$ by convention and the rest five are all present. For $j=1,$ the corresponding shapes $B(\ell,1)-B(\ell,2)+B(\ell+1,1)$ is counted in the third coefficient. For $j=0,$ the corresponding shapes $B(\ell,0)-B(\ell,1)+B(\ell+1,0)$ is counted in the second coefficient.

it remains to explain the last shape $-B(\ell+1,0)$ in the third coefficient. In fact, for type $5$ in the table, we can collect one more shape $(n-2\ell'+1, \ell'-1, \ell')$ when $j'=0$ and $j=-1.$ It rearranges to $(n-2\ell'+1, \ell', \ell'-1)=(n-2\ell-1, \ell+1, \ell)$ with $\ell=\ell'-1.$ Thus we collect the term $-B(\ell+1,0)e_{(n-2\ell-1,\ell+1,\ell)}$ in the last line of the expansion.
\end{proof}

Example \ref{worked} shows the process to expand the chromatic quasisymmetric function using $P$-tableaux and inverse Kostka numbers, with an illustration of Remark \ref{sixtypesremark} (2).

\begin{example}\label{worked}
Let $n=6$ and $\mathbf{d}=(2,3,4,4,6,6)$ with \[1\prec3,4,5,6,\qquad 2\prec4,5,6, \qquad 3,4\prec5,6.\] The bounce sequence is $b=(2,4,6)$, so $r=3$ with
\[S_1=\{1,2\},\qquad S_2=\{3,4\},\qquad S_3=\{5,6\},\qquad (n_1,n_2,n_3)=(2,2,2),\qquad k=4.\]
Then $G=\inc(P(\mathbf{d}))$ are the paths $1-2-3-4$ and $5-6$. The diagram above $\mathbf{d}$ is $\tau=(4,4,2,1)$, with $\ell(\tau)=4=k$. See Figure \ref{workedfig}.
 
By Lemma \ref{P exist}, only the $P$-tableaux $B(\ell,j)$ with shape $3^\ell2^j1^{6-3\ell-2j}$ will occur with $0\leq \ell\leq2$ and $0\leq j\leq 4-2\ell$. Now counting $P$-tableaux by inversions gives the following nonzero terms:
\[
\begin{array}{lll}
B(0,0)=(1+q)^4, & B(0,1)=(1+q)^2(1+4q+q^2), & B(0,2)=(1+q)^2(1+5q+q^2),\\[2pt]
B(0,3)=2q(1+q)^2, & B(1,0)=2q(1+q)^2, & B(1,1)=3q(1+q)^2,\\[2pt]
B(2,0)=q(1+q)^2.
\end{array}
\]
 
The shape $\mu=(3^2)$ gives $\mu'=(2,2,2)$, which is exactly the case described in Remark \ref{sixtypesremark} (2) when two triples rearrange to the same partition. Specifically, the corresponding six triples are

\[(2,2,2)^+,\, (4,1,1)^+,\, (3,3,0)^+,\,
(2,3,1)^-,\, (3,1,2)^-,\, (4,2,0)^-.\]
Note that $(2,3,1)^-$ and $(3,1,2)^-$ both rearrange to $(3,2,1)$, and
\[s_{(3,3)}=e_{(2,2,2)}+e_{(4,1,1)}+e_{(3,3)}-2\,e_{(3,2,1)}-e_{(4,2)},\qquad K\inv_{(3,2,1),(2,2,2)}=-2 .\]
 
Now to find $X_G(\mathbf{x},q)$, we collect the following coefficients using Proposition \ref{expansionprop}.

\[
\begin{array}{cc|l|l|l}
\ell & j & \text{coefficient} & \text{value} & \text{type}\\\hline
0&0& B(0,0)-B(0,1)+B(1,0) & 0 & (6,0,0)\\
0&1& B(0,1)-B(0,2)+B(1,1)-B(1,0) & 0 & (5,1,0)\\
0&2& B(0,2)-B(0,3)-B(1,0)-B(2,0) & (1+q)^2(1+q^2) & (4,2,0)\\
0&3& B(0,3)-B(1,1)+B(2,0) & 0 & (3,3,0)\\
0&4& 0 & 0 & (2,4,0)\\
1&0& B(1,0)-B(1,1)+B(2,0) & 0 & (4,1,1)\\
1&1& B(1,1)-B(2,0) & 2q(1+q)^2 & (3,2,1)\\
1&2& -B(2,0) & -q(1+q)^2 & (2,3,1)\\
2&0& B(2,0) & q(1+q)^2 & (2,2,2)\\
\end{array}
\]

In the row $(\ell,j)=(0,4)$, every term is zero since no such $P$-tableau exists. In the rows $(\ell,j)=(0,0)$, $(0,1)$, $(0,3)$ and $(1,0)$, the values of the coefficients are $0$ due to cancellations of $q$ polynomials.

The types $(2,3,1)$ and $(3,2,1)$ give the same $e_\lambda$, and the combined $-2B(2,0)$ is precisely the inverse number $K\inv_{(3,2,1),(2,2,2)}=-2$ above. Then we have the $e$-positive expansion
\[X_G(\mathbf{x},q)=(1+q)^2(1+q^2)\,e_{(4,2)}+q(1+q)^2\,e_{(3,2,1)}+q(1+q)^2\,e_{(2,2,2)}.\]
\end{example}
 
\begin{figure}[h]
    \centering
    \vspace{3mm}
\begin{tikzpicture}[every node/.style={minimum size=.4cm-\pgflinewidth, outer sep=0pt}]
\draw[step=0.4cm] (-1.2,-1.2) grid (1.2,1.2);
\draw (-1.2,-1.2) -- (1.2,1.2);
\node[fill=yellow!30] at (-1.0,-0.2) {}; \node[fill=yellow!30] at (-1.0,0.2) {};
\node[fill=yellow!30] at (-1.0,0.6) {};  \node[fill=yellow!30] at (-1.0,1.0) {};
\node[fill=yellow!30] at (-0.6,0.2) {};  \node[fill=yellow!30] at (-0.6,0.6) {};
\node[fill=yellow!30] at (-0.6,1.0) {};
\node[fill=yellow!30] at (-0.2,0.6) {};  \node[fill=yellow!30] at (-0.2,1.0) {};
\node[fill=yellow!30] at (0.2,0.6) {};   \node[fill=yellow!30] at (0.2,1.0) {};
\draw[step=0.4cm] (-1.2,-1.2) grid (1.2,1.2);
\draw (-1.2,-1.2) -- (1.2,1.2);
\draw [red,very thick] (-1.2,-1.2)--(-1.2,-0.4)--(-0.8,-0.4)--(-0.8,0)--(-0.4,0)--(-0.4,0.4)--(0.4,0.4)--(0.4,1.2)--(1.2,1.2);
\draw [green, ultra thick, dashed] (-1.2,-1.2)--(-1.2,-0.4)--(-0.4,-0.4)--(-0.4,0.4)--(0.4,0.4)--(0.4,1.2)--(1.2,1.2);
\draw[decorate,decoration={brace,mirror},thick,blue]
    (-1.4,1.2)--(-1.4,0.4) node[midway,left]{$n_3$};
\draw[decorate,decoration={brace,mirror},thick,blue]
    (-1.4,0.4)--(-1.4,-0.4) node[midway,left]{$n_2$};
\draw[decorate,decoration={brace,mirror},thick,blue]
    (-1.2,-1.4)--(-0.4,-1.4) node[midway,below]{$n_1$};
\end{tikzpicture}
\qquad\qquad
\begin{tikzpicture}
    \node[label=above:$G:$] (1) at (0,0) {1};
    \node (2) at (1,0) {2};
    \node (3) at (2,0) {3};
    \node (4) at (3,0) {4};
    \node (5) at (4,0) {5};
    \node (6) at (5,0) {6};
    \path (1) edge node {} (2);
    \path (2) edge node {} (3);
    \path (3) edge node {} (4);
    \path (5) edge node {} (6);
\end{tikzpicture}
\caption{The Dyck path $\mathbf{d}=(2,3,4,4,6,6)$ of Example \ref{worked}, with bounce path $b(\mathbf{d})$ dashed, $\tau=(4,4,2,1)$ shaded, and $G=\inc(P(\mathbf{d}))$ on the right.}
\label{workedfig}
\end{figure}
 
\begin{proposition}\label{ecoeff}
Let $\lambda=(\lambda_1,\lambda_2,\lambda_3)$ be a partition of $n$ with at most three parts with
\[\ell:=\lambda_3,\qquad j:=\lambda_2-\lambda_3.\]
So $\lambda=(n-2\ell-j,\,\ell+j,\,\ell)$ and $(\ell,j)\mapsto\lambda$ is a bijection onto the partitions of $n$ with at most three parts. Let $m(\lambda)\in\{0,1,2\}$ be the number of the two equalities
\[\lambda_1=\lambda_2+1,\qquad\qquad \lambda_2=\lambda_3+1\]
that hold. Hence, we have
\begin{equation}\label{ecoeffeq}
\begin{split}
[e_\lambda]X_G(\mathbf{x},q)=\ &B(\ell,j)+B(\ell+1,j)+B(\ell+2,j-3)-B(\ell,j+1)\\
&-m(\lambda)\,B(\ell+1,j-1)-B(\ell+1,j-2)-B(\ell+2,j-2),
\end{split}
\end{equation}
and $[e_\lambda]X_G(\mathbf{x},q)=0$ if $\lambda$ has more than three parts.
\end{proposition}
 
\begin{proof}
Following Lemma \ref{six types}, it is convenient to substitute $\nu=\mu'$ and write Equation \eqref{inv kostka expand} as
\[X_G(\mathbf{x},q)=\sum_{\substack{\mu\vdash n\\ \mu_1\leq3}}B_{\mu}(q)\sum_{\lambda}K\inv_{\lambda,\mu'}\,e_\lambda
\ \ =\ \
\sum_{\substack{\nu\vdash n\\ \ell(\nu)\leq3}}B_{\nu'}(q)\sum_{\lambda}K\inv_{\lambda,\nu}\,e_\lambda.\]
For $\nu=(\nu_1,\nu_2,\nu_3)$, we have $\nu'=3^{\nu_3}2^{\nu_2-\nu_3}1^{\nu_1-\nu_2}$ and $B_{\nu'}=B(\nu_3,\,\nu_2-\nu_3)$. The six triples associated to $\nu$ (six types in Lemma \ref{six types}) have at most three parts, so $K\inv_{\lambda,\nu}=[e_\lambda]X_G(\mathbf{x},q)=0$ if $\lambda$ has more than three parts.

Now fix $\lambda,$ it remains to determine the partition $\nu$ whose associated types rearrange to $\lambda.$

Since $\nu$ is a partition ($\nu_1\geq \nu_2\geq \nu_3$), the types $1,2,3$ and $6$ are weakly decreasing so that these types equal to $\lambda$ automatically. In particular, for $\lambda=(\lambda_1,\lambda_2,\lambda_3),$ we have \[\nu=\lambda,\qquad \nu=\lambda+(-2,1,1),\qquad \nu=\lambda+(-1,-1,2),\qquad \nu=\lambda+(-2,0,2)\]
respectively. Therefore each $\nu$ gives a corresponding term (with signs) in Equation (\ref{ecoeffeq}) as \[+B(\ell,j),\,\,\, +B(\ell+1,j),\,\,\, +B(\ell+2,j-3),\,\,\, -B(\ell+2,j-2).\]
Note that $(\ell,j)\mapsto\lambda$ is a bijection onto the partitions of $n$ with at most three parts by construction. For example, $\nu=\lambda+(-2,1,1)$ is a partition exactly when $\lambda_1-2\geq\lambda_2+1$. That is, $3(\ell+1)+2j\leq n$, which is the cell count for $B(\ell+1,j)$.

Additionally, type $4, (\nu_1,\nu_2+1,\nu_3-1)$, is weakly decreasing exactly when $\nu_1>\nu_2$, in which case $\nu=\lambda+(0,-1,1)$. Instead if $\nu_1=\nu_2$, this type rearranges to $(\nu_2+1,\nu_1,\nu_3-1)$ and $\nu=\lambda+(-1,0,1)$, in which case $\lambda_1=\lambda_2+1$. Similarly, type $5,  (\nu_1+1,\nu_2-1,\nu_3)$, is weakly decreasing exactly when $\nu_2>\nu_3$, in which case $\nu=\lambda+(-1,1,0)$. Instead if $\nu_2=\nu_3$, this type rearranges to $(\nu_1+1,\nu_3,\nu_2-1)$, and $\nu=\lambda+(-1,0,1)$, in which case $\lambda_2=\lambda_3+1$. Consequently, the partition $\nu=\lambda+(-1,0,1)$ appears $m(\lambda)$ times with a negative sign. Hence we obtain the term $-m(\lambda)B(\ell+1,j-1)$ in Equation (\ref{ecoeffeq}) with \[B_{\nu'}=B(\nu_3,\,\nu_2-\nu_3)=B(\lambda_3+1,\lambda_2-\lambda_3-1)=B(\ell+1,j-1).\]
The types $4$ and $5$ in their automatically weakly decreasing cases give the rest two terms in Equation (\ref{ecoeffeq}) \[-B(\ell+1,j-2),\,\,\, -B(\ell,j+1).\]
\end{proof}

\subsection{The coefficients of $e_\lambda$ with $\lambda_2-\lambda_3\leq 1$}

We show that two families of coefficients in Equation (\ref{ecoeffeq}) lie in $\bN[q]$. Specifically, we construct a \emph{sign reversing injection} for each of them, so that every $P$-tableau with a negative sign is injectively mapped to a $P$-tableau with a positive sign, and the number of inversions is unchanged. Consequently, what remains is a polynomial in $q$ with nonnegative coefficients.

The next lemma gives a bound on the chains contained in a column of a $P$-tableau. It is critical for the insertion algorithms below being reversible. 

\begin{lemma}\label{lm1}
Let $P$ be a natural unit interval order on $[n]$ with $r=3$. Let $G=\inc(P)$ and $T\in\PT(G).$ Let $a_{i,j}$ be the entry of $T$ at position $(i,j)$. Fix a column $j$ of $T$, for every $i$ there are at most two indices $s>i$ such that \[a_{i,j},a_{i+1,j}\cdots,a_{s-1,j}\prec a_{s,j}.\]
\end{lemma}

\begin{proof}
Suppose that there exist two indices $s_1, s_2$ satisfying the chain relations. Without loss of generality, we assume $s_1<s_2$. The chain relation for $s_1$ gives $a_{i,j}\prec a_{s_1,j},$ and the chain relation for $s_2$ gives $a_{s_1,j}\prec a_{s_2,j}$. Hence $a_{i,j}\prec a_{s_1,j}\prec a_{s_2,j}$ is a chain of length three. If a third index $s_3$ exists such that $s_3>s_2$, then similarly we have a chain of length four $a_{i,j}\prec a_{s_1,j}\prec a_{s_2,j}\prec a_{s_3,j},$ contradicting Lemma \ref{remark 1}.
\end{proof}

\begin{theorem} \label{e_{(n-2l,l,l)}}
Let $P$ be a natural unit interval order on $[n]$ such that $P\simeq P(\mathbf{d})$ for a Dyck path $\mathbf{d}$ with $r=3$, and $G=\inc(P)$. Then for $\ell\in[0,\min\{n_1,n_2,n_3\}],$ \[[e_{(n-2\ell,\ell,\ell)}]X_G(\mathbf{x},q)=B(\ell,0)+B(\ell+1,0)-B(\ell,1)\in \bN[q].\]
\end{theorem}

\begin{proof}
Apply $j=0$ so that $\lambda=(n-2\ell,\ell,\ell)$ to Proposition \ref{ecoeff}, and we get $[e_{(n-2\ell,\ell,\ell)}]X_G(\mathbf{x},q)=B(\ell,0)+B(\ell+1,0)-B(\ell,1)$. Note that all other terms are $0$ since $B(\ell,j)=0$ for any negative $\ell, j$.

Let $D$ be the set of $P$-tableaux of shape $3^\ell2^11^{n-3\ell-2}$ counted by $B(\ell,1)$, and $R$ be the set of $P$-tableaux of shape $3^\ell1^{n-3\ell}$ or $3^{\ell+1}1^{n-3\ell-3}$ counted by $B(\ell,0)+B(\ell+1,0)$. It suffices to construct an injection $\alpha:D\to R$ with $\iv(\alpha(T))=\iv(T)$.
 
A tableau $T\in D$ has $\ell+1+(n-3\ell-2)=n-2\ell-1$ rows. Its first column is $a_{1,1},\dots,a_{n-2\ell-1,1}$. Define $\alpha$ by the following three cases.
 
\begin{enumerate}
\item[(1)] If there exists some index $s\in[\ell+2,n-2\ell-1]$ such that 
\begin{equation}\label{conditioneqalpha}
  a_{\ell+1,1},\,a_{\ell+2,1},\,\cdots,\,a_{s-1,1}\ \prec\ a_{s,1}\ \prec\ a_{\ell+1,2},  
\end{equation}
then move $a_{s,1}$ into position $(\ell+1,2)$ and move $a_{\ell+1,2}$ into position $(\ell+1,3)$. See Figure \ref{coeff1} (top). The shape of $T$ becomes $3^{\ell+1}1^{n-3\ell-3}$. Note that such index $s$ is unique. Otherwise, if $s_1, s_2$ are two such indices with $s_1<s_2$, then $a_{\ell+1,1}\prec a_{s_1,1}\prec a_{s_2,1}\prec a_{\ell+1,2}$ is a chain of length four by Lemma \ref{lm1}, contradicting Lemma \ref{remark 1}.

\item[(2)] If there does not exist such index $s$, then let $m$ be the smallest index in $[\ell+2,n-2\ell-1]$ so that $a_{m,1}\nprec a_{\ell+1,2}$. Insert $a_{\ell+1,2}$ into the first column immediately above $a_{m,1}$. See Figure \ref{coeff1} (bottom). The shape of $T$ becomes $3^{\ell}1^{n-3\ell}$.

\item[(3)] If there does not exist such index $m$, that is, $a_{\ell+2,1},\dots,a_{n-2\ell-1,1}\prec a_{\ell+1,2}$, then place $a_{\ell+1,2}$ at the bottom of the first column. The shape of $T$ becomes $3^{\ell}1^{n-3\ell}$.
\end{enumerate}

\begin{figure}[h]
\centering
\begin{tikzpicture}[x=0.95cm,y=0.72cm,node distance=3pt,
  cell/.style={draw,thick,minimum width=0.95cm,minimum height=0.72cm,inner sep=0pt,anchor=north west,font=\scriptsize},
  s1/.style={fill=orange!30},
  s2/.style={fill=green!25},
  s3/.style={fill=cyan!25},
  s23/.style={fill=gray!20},
  s12/.style={fill=yellow!35},
  ghost/.style={draw,thick,gray!60,minimum width=0.95cm,minimum height=0.72cm,inner sep=0pt,anchor=north west,font=\scriptsize,text=gray},
  lab/.style={font=\scriptsize},
  note/.style={font=\scriptsize,anchor=north west,align=left,text width=3.6cm,inner sep=1pt}]
\node[lab,anchor=south,align=center] at (1.5,0.45) {$T\in D$, case (1):\\ $a_{\ell+1,1},\dots,a_{s-1,1}\prec a_{s,1}\prec a_{\ell+1,2}$};
\node at (0.5,0.12) {$\vdots$};
\node[ghost] at (0,0) {$a_{\ell,1}$};
\node[ghost] at (1,0) {$a_{\ell,2}$};
\node[ghost] at (2,0) {$a_{\ell,3}$};
\node[cell,s1] at (0,-1) {$a_{\ell+1,1}$};
\node[cell,s3] at (1,-1) {$a_{\ell+1,2}$};
\node[cell,s1] at (0,-2) {$a_{\ell+2,1}$};
\node at (0.5,-3.45) {$\vdots$};
\node[cell,s1] at (0,-4) {$a_{s-1,1}$};
\node[cell,s2] at (0,-5) {$a_{s,1}$};
\node[cell] at (0,-6) {$a_{s+1,1}$};
\node at (0.5,-7.45) {$\vdots$};
\node[lab,anchor=east] at (-0.12,-1.45) {row $\ell{+}1$};
\draw[decorate,decoration={brace,amplitude=4pt},thick] (-0.12,-4.95) -- (-0.12,-1.05) node[midway,left=5pt,lab,align=right] {all $\prec a_{s,1}$};
\node[note] (n22224) at (1.12,-5.05) {$s$ is unique (two would give a chain of length four)};
\node at (5.35,-4) {$\xrightarrow{\ \ \alpha\ \ }$};
\node[lab,anchor=south,align=center] at (7.7,0.45) {$\alpha(T)$};
\node at (6.7,0.12) {$\vdots$};
\node[ghost] at (6.2,0) {$a_{\ell,1}$};
\node[ghost] at (7.2,0) {$a_{\ell,2}$};
\node[ghost] at (8.2,0) {$a_{\ell,3}$};
\node[cell,s1] at (6.2,-1) {$a_{\ell+1,1}$};
\node[cell,s2] at (7.2,-1) {$a_{s,1}$};
\node[cell,s3] at (8.2,-1) {$a_{\ell+1,2}$};
\node[cell,s1] at (6.2,-2) {$a_{\ell+2,1}$};
\node at (6.7,-3.45) {$\vdots$};
\node[cell,s1] at (6.2,-4) {$a_{s-1,1}$};
\node[cell] at (6.2,-5) {$a_{s+1,1}$};
\node at (6.7,-6.45) {$\vdots$};
\node[note,text=blue!70!black] (n13360) at (7.32,-2.05) {all pairs are comparable, $\iv$ unchanged};

\node[lab,anchor=south,align=center] at (1.5,-9.35) {$T\in D$, cases (2),(3): $s$ not exist};
\node at (0.5,-9.68) {$\vdots$};
\node[ghost] at (0,-9.8) {$a_{\ell,1}$};
\node[ghost] at (1,-9.8) {$a_{\ell,2}$};
\node[ghost] at (2,-9.8) {$a_{\ell,3}$};
\node[cell,s12] at (0,-10.8) {$a_{\ell+1,1}$};
\node[cell,s23] at (1,-10.8) {$a_{\ell+1,2}$};
\node[cell,s12] at (0,-11.8) {$a_{\ell+2,1}$};
\node at (0.5,-13.25) {$\vdots$};
\node[cell,s12] at (0,-13.8) {$a_{m-1,1}$};
\node[cell] at (0,-14.8) {$a_{m,1}$};
\node at (0.5,-16.25) {$\vdots$};
\node[lab,anchor=east] at (-0.12,-11.25) {row $\ell{+}1$};
\draw[decorate,decoration={brace,amplitude=4pt},thick] (-0.12,-14.75) -- (-0.12,-10.85) node[midway,left=5pt,lab,align=right] {all $\prec a_{\ell+1,2}$};
\node[note] (n13552) at (1.12,-14.45) {$m$ = smallest index with $a_{m,1}\nprec a_{\ell+1,2}$.};
\node at (5.35,-13.8) {$\xrightarrow{\ \ \alpha\ \ }$};
\node[lab,anchor=south,align=center] at (7.7,-9.35) {$\alpha(T)$};
\node at (6.7,-9.68) {$\vdots$};
\node[ghost] at (6.2,-9.8) {$a_{\ell,1}$};
\node[ghost] at (7.2,-9.8) {$a_{\ell,2}$};
\node[ghost] at (8.2,-9.8) {$a_{\ell,3}$};
\node[cell,s12] at (6.2,-10.8) {$a_{\ell+1,1}$};
\node[cell,s12] at (6.2,-11.8) {$a_{\ell+2,1}$};
\node at (6.7,-13.25) {$\vdots$};
\node[cell,s12] at (6.2,-13.8) {$a_{m-1,1}$};
\node[cell,s23] at (6.2,-14.8) {$a_{\ell+1,2}$};
\node[cell] at (6.2,-15.8) {$a_{m,1}$};
\node at (6.7,-17.25) {$\vdots$};
\node[note,text=blue!70!black] (n14912) at (7.32,-12.45) {all $\prec a_{\ell+1,2}$, \\$\iv$ unchanged};
\node[note] (n24912) at (7.32,-14.85) {in case (3) $a_{\ell+1,2}$ goes to the bottom};
\begin{scope}[yshift=-13.456cm,xshift=1.2cm]
\node[cell,s1,minimum width=0.45cm,minimum height=0.35cm] at (0,0) {}; \node[lab,anchor=west] at (0.45,-0.25) {$S_1$};
\node[cell,s2,minimum width=0.45cm,minimum height=0.35cm] at (1.0,0) {}; \node[lab,anchor=west] at (1.45,-0.25) {$S_2$};
\node[cell,s3,minimum width=0.45cm,minimum height=0.35cm] at (2.0,0) {}; \node[lab,anchor=west] at (2.45,-0.25) {$S_3$};
\node[cell,s12,minimum width=0.45cm,minimum height=0.35cm] at (3.0,0) {}; \node[lab,anchor=west] at (3.45,-0.25) {$S_1\cup S_2$};
\node[cell,s23,minimum width=0.45cm,minimum height=0.35cm] at (4.8,0) {}; \node[lab,anchor=west] at (5.25,-0.25) {$S_2\cup S_3$};
\end{scope}
\end{tikzpicture}
\caption{The injective map $\alpha$.}
\label{coeff1}
\end{figure}

First, under the map $\alpha$, $\alpha(T)$ is a $P$-tableau. Since $S_1,S_2,S_3$ are antichains by Lemma \ref{remark 1}, any chain of length three must have its entries lie one in each of $S_1,S_2$ and $S_3$. Thus, for chains $a_{\ell+1,1}\prec a_{s,1}\prec a_{\ell+1,2}$ and $a_{\ell,1}\prec a_{\ell,2}\prec a_{\ell,3}$, we have $a_{\ell,1}, a_{\ell+1,1}\in S_1$, $a_{\ell,2}, a_{s,1}\in S_2$ and $a_{\ell,3}, a_{\ell+1,2}\in S_3$. Therefore in case $(1)$, $a_{s,1}\nprec a_{\ell,2}$ and $a_{\ell+1,2}\nprec a_{\ell,3}$, which are the two new column conditions in $\alpha(T).$ In the first column, the only new adjacency is $a_{s-1,1}$ directly above $a_{s+1,1}$. We know $a_{s+1,1}\nprec a_{s-1,1}$. Otherwise $a_{s+1,1}\prec a_{s-1,1}\prec a_{s,1}$, contradicting $a_{s+1,1}\nprec a_{s,1}$ in $T$. In case $(2)$, the new adjacencies are $a_{m-1,1}$ above $a_{\ell+1,2}$ above $a_{m,1}$. The choice of $m$ guarantees that $a_{m-1,1}\prec a_{\ell+1,2}$ and $a_{m,1}\nprec a_{\ell+1,2}$. Case $(3)$ is the same with $a_{n-2\ell-1,1}=a_{m-1,1}$.
 
Second, inversions are preserved. In each case exactly one entry changes its row and passes only entries comparable to it, so no pair of incomparable entries changes its relative position. In particular, the entry $a_{s,1}$ of case $(1)$ moves up from row $s$ to row $\ell+1$, passing $a_{\ell+2,1},\dots,a_{s-1,1}$, all of which $\prec a_{s,1}$. In cases $(2)$ and $(3)$, the entry $a_{\ell+1,2}$ moves down passing $a_{\ell+1,1}, a_{\ell+2,1},\dots, a_{m-1,1}$, all of which $\prec a_{\ell+1,2}$. Hence $\iv(\alpha(T))=\iv(T)$.

Third, $\alpha$ is injective. We will find the inverse map under each case. In case $(1)$, let $I=\alpha(T)$ of shape $3^{\ell+1}1^{n-3\ell-3}$ and write its row $\ell+1$ as $(a_{\ell+1,1},x,y)$. Back to $T$, the entry $x$ lay in the first column, and all entries between rows $\ell+1$ and the row of $x$ are $\prec x$. Hence $T$ is recovered from $I$ by returning $y$ to position $(\ell+1,2)$ and inserting $x$ into the first column immediately above the topmost row $t\geq \ell+2$ of $I$ so that $I_{t,1}\nprec x$. If no such row $t$ exists, place $x$ at the bottom of the first column. Now in case $(2)$ or $(3)$, let $U=\alpha(T)$ of shape $3^\ell1^{n-3\ell}$ and let $p$ be the smallest index greater than $\ell+1$ with
\begin{equation}\label{starstar}
U_{\ell+1,1},\,\cdots,\,U_{p-1,1}\ \prec\ U_{p,1}.
\end{equation}
The moved entry $a_{\ell+1,2}$ in Figure \ref{coeff1} (bottom) satisfies Condition (\ref{starstar}) by the choice of $m$ in case $(2)$ or $(3)$. By Lemma \ref{lm1} at most two indices exist for Condition (\ref{starstar}). Suppose there exist exactly two with $p<p'$, and the moved entry $a_{\ell+1,2}$ is $p'.$ Moving $p'$ leaves the rows above it untouched, so in $T$ the entry $U_{p,1}$ still lies at row $p$ with $a_{\ell+1,1},\dots,a_{p-1,1}\prec U_{p,1}$ and $U_{p,1}\prec U_{p',1}=a_{\ell+1,2}.$ However, this $p=s$ satisfies the condition described in case $(1)$ of $T$, giving a contradiction. Thus the moved entry $a_{\ell+1,2}$ is $p,$ and $T$ can be recovered by moving $U_{p,1}$ into position $(\ell+1,2)$. Therefore $\alpha$ is injective.
\end{proof}

Next, we want to prove the other family of coefficients in Proposition \ref{ecoeff}. That is when $\lambda=(n-2\ell-1,\ell+1,\ell)$ with $j=1$, $[e_{(n-2\ell-1,\ell+1,\ell)}]X_G(\mathbf{x},q)=B(\ell,1)+B(\ell+1,1)-B(\ell,2)-m(\lambda)B(\ell+1,0)\in \bN[q]$. For $j=1,$ one has $\lambda_2=\lambda_3+1$ so that $m(\lambda)\geq 1.$

\begin{theorem} \label{e_{(n-2l-1,l+1,l)}}
Let $P$ be a natural unit interval order on $[n]$ such that $P=P(\mathbf{d})$ for a Dyck path $\mathbf{d}$ with $r=3$, and $G=\inc(P)$. Then for $\ell\in [0,\min\{n_1,n_2,n_3\}]$, \[[e_{(n-2\ell-1,\ell+1,\ell)}]X_G(\mathbf{x},q)=B(\ell,1)+B(\ell+1,1)-B(\ell,2)-m(\lambda)B(\ell+1,0)\in \bN[q].\]
\end{theorem}

\begin{proof}
We first prove the case when $m(\lambda)=1$ so that $\lambda_1\neq \lambda_2+1$. The case $m(\lambda)=2$ is treated at the end of the proof. With $j=1,$ the terms $B(\ell+2,j-3), B(\ell+1, j-2)$ and $B(\ell+2,j-2)$ are all $0,$ so it remains to show that $[e_{(n-2\ell-1,\ell+1,\ell)}]X_G(\mathbf{x},q)=B(\ell,1)+B(\ell+1,1)-B(\ell,2)-B(\ell+1,0)\in \bN[q].$ 

Let $D_1,D_2,D_3,D_4$ be the sets of $P$-tableaux in $\PT(G)$ of the four shapes
\[3^\ell2^11^{n-3\ell-2},\qquad 3^{\ell+1}2^11^{n-3\ell-5},\qquad 3^\ell2^21^{n-3\ell-4},\qquad 3^{\ell+1}1^{n-3\ell-3},\]
which are counted by $B(\ell,1), B(\ell+1,1), B(\ell,2)$, and $B(\ell+1,0)$ respectively. Then we need to show \[\sum_{T\in D_1}q^{\iv(T)}+\sum_{T\in D_2}q^{\iv(T)}-\sum_{T\in D_3}q^{\iv(T)}-\sum_{T\in D_4}q^{\iv(T)}\in \bN[q].\]

We will prove it in the following steps:
\begin{enumerate}
    \item Construct an inversion-preserving map $f:D_4\to D_1$ and show $f$ is injective.
    \item Construct an inversion-preserving map $g:D_2\to D_3$.
    \item Set $\mathcal{A}=D_1\setminus\im(f),$ and $\mathcal{B}=D_3\setminus\im(g).$ Thus $\sum_{D_1}q^{\iv}-\sum_{D_4}q^{\iv}=\sum_{\mathcal{A}}q^{\iv}$. Also $\sum_{D_2}q^{\iv}\geq \sum_{\im(g)}q^{\iv}$ since $g$ needs not be injective.
    \item Construct an inversion-preserving injection $\phi:\mathcal{B}\to \mathcal{A}$, so that \[\sum_{T\in \mathcal{A}}q^{\iv(T)}-\sum_{T\in \mathcal{B}}q^{\iv(T)}=\sum_{T\in \mathcal{A}\setminus\phi(\mathcal{B})}q^{\iv(T)}\ \in\ \bN[q].\]
\end{enumerate}

We will repeatedly use the following two facts:
\begin{enumerate}
    \item[(F1)] If $x\prec z$ and $y\nprec z,$ then $d_x<z\leq d_y$, hence $x<y$ in the natural order. Similarly, if $x\prec y$ and $x\nprec z$, then $y>z$ in the natual order. 
    \item[(F2)] Every element in $S_1$ is minimal and every element in $S_3$ is maximal in $P$. So every element in $S_1$ is $\prec$ every element in $S_3$ since $d_x\leq d_{b_1}\leq b_2$ for any $x\in S_1.$ Further, a chain $x\prec y\prec z$ has $x\in S_1, y\in S_2$ and $z\in S_3.$
\end{enumerate}

(Step 1): We now define the map $f:D_4\to D_1$ with $\iv(f(T))=\iv(T)$ as follows: Let $T_1\in D_4$, we find the smallest index $s$ for $s\in[\ell+2,n-2\ell-2]$ such that $a_{s,1}\nprec a_{\ell+1,2}$, then insert $a_{\ell+1,2}$ above $a_{s,1}$, and move $a_{\ell+1,3}$ left by one cell. See Figure \ref{f}. If there does not exist such $a_{s,1}$, then we simply move $a_{\ell+1,2}$ to the bottom of the first column.

First, the resulting tableau $f(T_1)$ is a $P$-tableau. Row $\ell+1$ becomes the chain of length two $a_{\ell+1,1}\prec a_{\ell+1,3}$. $a_{\ell+1,3}\nprec a_{\ell,2}$ since $a_{\ell+1,3}\in S_3$ and $a_{\ell,2}\in S_2$. In the first column, $a_{s-1,1}\prec a_{\ell+1,2}$ and $a_{s,1}\nprec a_{\ell+1,2}.$ Second, inversions are preserved. $a_{\ell+1,2}$ moves down passing only entries comparable to it, and no pair of incomparable entries changes its relative position. Third, $f$ is injective. $T_1$ is recovered from $f(T_1)$ by $\alpha(f(T_1)$, where $\alpha$ is the map of Theorem \ref{e_{(n-2l,l,l)}} (See Lemma \ref{charA} below).

\begin{figure}[h]
\centering
\begin{tikzpicture}[x=0.95cm,y=0.72cm,node distance=3pt,
  cell/.style={draw,thick,minimum width=0.95cm,minimum height=0.72cm,inner sep=0pt,anchor=north west,font=\scriptsize},
  s1/.style={fill=orange!30},
  s2/.style={fill=green!25},
  s3/.style={fill=cyan!25},
  s23/.style={fill=gray!20},
  s12/.style={fill=yellow!35},
  ghost/.style={draw,thick,gray!60,minimum width=0.95cm,minimum height=0.72cm,inner sep=0pt,anchor=north west,font=\scriptsize,text=gray},
  lab/.style={font=\scriptsize},
  note/.style={font=\scriptsize,anchor=north west,align=left,text width=3.6cm,inner sep=1pt}]
\node[lab,anchor=south,align=center] at (1.5,0.45) {$T_1\in D_4$};
\node at (0.5,0.12) {$\vdots$};
\node[ghost] at (0,0) {$a_{\ell,1}$};
\node[ghost] at (1,0) {$a_{\ell,2}$};
\node[ghost] at (2,0) {$a_{\ell,3}$};
\node[cell,s1] at (0,-1) {$a_{\ell+1,1}$};
\node[cell,s2] at (1,-1) {$a_{\ell+1,2}$};
\node[cell,s3] at (2,-1) {$a_{\ell+1,3}$};
\node[cell] at (0,-2) {$a_{\ell+2,1}$};
\node at (0.5,-3.45) {$\vdots$};
\node[cell] at (0,-4) {$a_{s-1,1}$};
\node[cell] at (0,-5) {$a_{s,1}$};
\node at (0.5,-6.45) {$\vdots$};
\node[lab,anchor=east] at (-0.12,-1.45) {row $\ell{+}1$};
\draw[decorate,decoration={brace,amplitude=4pt},thick] (-0.12,-4.95) -- (-0.12,-2.05) node[midway,left=5pt,lab,align=right] {all $\prec a_{\ell+1,2}$};

\node[note] (n24912) at (1.12,-5.05) {$s$ = smallest index with $a_{s,1}\nprec a_{\ell+1,2}$. if there is none, $a_{\ell+1,2}$ goes to the bottom};
\node at (5.35,-4) {$\xrightarrow{\ \ f\ \ }$};
\node[lab,anchor=south,align=center] at (7.7,0.45) {$f(T_1)\in D_1$};
\node at (6.7,0.12) {$\vdots$};
\node[ghost] at (6.2,0) {$a_{\ell,1}$};
\node[ghost] at (7.2,0) {$a_{\ell,2}$};
\node[ghost] at (8.2,0) {$a_{\ell,3}$};
\node[cell,s1] at (6.2,-1) {$a_{\ell+1,1}$};
\node[cell,s3] at (7.2,-1) {$a_{\ell+1,3}$};
\node[cell] at (6.2,-2) {$a_{\ell+2,1}$};
\node at (6.7,-3.45) {$\vdots$};
\node[cell] at (6.2,-4) {$a_{s-1,1}$};
\node[cell,s2] at (6.2,-5) {$a_{\ell+1,2}$};
\node[cell] at (6.2,-6) {$a_{s,1}$};
\node at (6.7,-7.45) {$\vdots$};
\node[note,text=blue!70!black] (n22224) at (7.32,-2.45) {all comparable to $a_{\ell+1,2}$,\\ $\iv$ unchanged};

\begin{scope}[yshift=-6.6cm,xshift=1.2cm]
\node[cell,s1,minimum width=0.45cm,minimum height=0.35cm] at (0,0) {}; \node[lab,anchor=west] at (0.45,-0.25) {$S_1$};
\node[cell,s2,minimum width=0.45cm,minimum height=0.35cm] at (1.0,0) {}; \node[lab,anchor=west] at (1.45,-0.25) {$S_2$};
\node[cell,s3,minimum width=0.45cm,minimum height=0.35cm] at (2.0,0) {}; \node[lab,anchor=west] at (2.45,-0.25) {$S_3$};
\node[cell,s12,minimum width=0.45cm,minimum height=0.35cm] at (3.0,0) {}; \node[lab,anchor=west] at (3.45,-0.25) {$S_1\cup S_2$};
\node[cell,s23,minimum width=0.45cm,minimum height=0.35cm] at (4.8,0) {}; \node[lab,anchor=west] at (5.25,-0.25) {$S_2\cup S_3$};
\end{scope}
\end{tikzpicture}
\caption{The injective map $f$.}
\label{f}
\end{figure}

(Step 2): To define map $g:D_2\to D_3$ with $\iv(g(T))=\iv(T)$, let $T_2\in D_2,$ and we discuss two main cases.
\begin{itemize}
    \item Case 1: $a_{\ell+2,2}\prec a_{\ell+1,3}$.

By transitivity $a_{\ell+2,1}\prec a_{\ell+2,2}\prec a_{\ell+1,3}$. We move $a_{\ell+1,3}$ into the position $(\ell+2,2)$, and find the smallest index $s$ for $s\in[\ell+3,n-2\ell-3]$ such that $a_{s,1}\nprec a_{\ell+2,2}$, then insert $a_{\ell+2,2}$ above $a_{s,1}$. See Figure \ref{g} (left). If there does not exist such $a_{s,1}$, then we simply move $a_{\ell+2,2}$ to the bottom of the first column.

First, the result is a $P$-tableau. Row $\ell+2$ becomes the chain $a_{\ell+2,1}\prec a_{\ell+1,3}$ and the first column is handled the same as for $f$. Second, inversions are preserved. $a_{\ell+1,3}$ and $a_{\ell+2,2}$ both move down passing only entries comparable to them.
     
     \item Case 2: $a_{\ell+2,2}\nprec a_{\ell+1,3}$. We consider three sub-cases. 

\textbf{Case 2a:} If $a_{\ell+1,2}\nprec a_{\ell+2,2}$ and $a_{\ell+2,1}\nprec a_{\ell+1,2}$, find the smallest index $s$ for $s\in[\ell+3,n-2\ell-3]$ such that $a_{s,1}\nprec a_{\ell+1,2}$ and insert $a_{\ell+1,2}$ above $a_{s,1}$ (at the bottom of the first column if there is no such $s$). Move $a_{\ell+1,3}$ left by one cell. See Figure \ref{g} (right). 

First, the result is a $P$-tableau. Row $\ell+1$ becomes the chain of length two $a_{\ell+1,1}\prec a_{\ell+1,3}$. $a_{\ell+2,2}\nprec a_{\ell+1,3}$ by hypothesis. In the first column, $a_{\ell+1,2}\in S_2$ is not $\prec a_{\ell+2,1}\in S_1\cup S_2,$ and the entries $a_{\ell+3,1},\cdots,a_{s-1,1}\prec a_{\ell+1,2}$ while $a_{s,1}\nprec a_{\ell+1,2}$. Second, inversions are preserved. Since $a_{\ell+1,2}\prec a_{\ell+1,3}$ while $a_{\ell+2,2}\nprec a_{\ell+1,3}$, then (F1) gives $a_{\ell+1,2}<a_{\ell+2,2}$. Similarly, since $a_{\ell+2,1}\prec a_{\ell+2,2}$ while $a_{\ell+1,2}\nprec a_{\ell+2,2}$, (F1) gives $a_{\ell+2,1}<a_{\ell+1,2}$. Both pairs are incomparable. By inserting $a_{\ell+1,2}$ in a position below both $a_{\ell+2,1}$ and $a_{\ell+2,2}$, the previous inversion $(a_{\ell+1,2},a_{\ell+2,1})$ is gone while a new inversion $(a_{\ell+2,2},a_{\ell+1,2})$ is created. The other entries passed by $a_{\ell+1,2}$ are all comparable to it. Thus the resulting $P$-tableau $g(T_2)$ has the same number of inversions as $T_2.$
        
\textbf{Case 2b:} If $a_{\ell+1,2}\nprec a_{\ell+2,2}$ and $a_{\ell+2,1}\prec a_{\ell+1,2}$, we conduct the same insertion algorithm as described in Case 2a, and in addition, we switch $a_{\ell+1,3}$ and $a_{\ell+2,2}$, and the result is still a $P$-tableau. See Figure \ref{g2} (left). 

Before switching $a_{\ell+1,3}$ and $a_{\ell+2,2}$, the map is the same as Case 2a. However, now $a_{\ell+2,1}\prec a_{\ell+1,2}$ gives no inversion anymore. Since $a_{\ell+1,2}\prec a_{\ell+1,3}$ while $a_{\ell+1,2}\nprec a_{\ell+2,2}$, (F1) gives $a_{\ell+2,2}<a_{\ell+1,3}$. Thus switching $a_{\ell+1,3}$ and $a_{\ell+2,2}$ removes the inversion $(a_{\ell+1,3}, a_{\ell+2,2}$), offsetting the new inversion $(a_{\ell+2,2},a_{\ell+1,2})$.

\textbf{Case 2c:} If $a_{\ell+1,2}\prec a_{\ell+2,2}$, then find the smallest index $s$ for $s\in[\ell+2,n-2\ell-3]$ such that $a_{s,1}\nprec a_{\ell+1,2}$ and insert $a_{\ell+1,2}$ above $a_{s,1}$ in the first column (at the bottom of the first column if there is no such $s$). Move $a_{\ell+1,3}$ left by one cell. This is the same as Figure \ref{g} (right). For $s\geq \ell+3,$ the result is a $P$-tableau by the same checks as in Case 2a (now $a_{\ell+2,1}\prec a_{\ell+1,2}$), and all inversions remain unchanged since every entry passed by $a_{\ell+1,2}$ including $a_{\ell+2,2}$ is comparable to it. Note that the search for $s$ starts at $s=\ell+2$ in this case. If $a_{\ell+2,1}\nprec a_{\ell+1,2}$, so that $s=\ell+2$, we still insert $a_{\ell+1,2}$ above $a_{\ell+2,1}$, which means that $a_{\ell+1,2}$ occupies the position $(\ell+2,1)$, and $a_{\ell+2,1}$ and all other remaining entries of the first column move down by one cell. See Figure \ref{g2} (right). It is straightforward to compare that inversions are again unchanged.
    \end{itemize}

\begin{figure}[h]
\centering
\begin{tikzpicture}[x=0.95cm,y=0.72cm,node distance=3pt,
  cell/.style={draw,thick,minimum width=0.95cm,minimum height=0.72cm,inner sep=0pt,anchor=north west,font=\scriptsize},
  s1/.style={fill=orange!30},
  s2/.style={fill=green!25},
  s3/.style={fill=cyan!25},
  s23/.style={fill=gray!20},
  s12/.style={fill=yellow!35},
  ghost/.style={draw,thick,gray!60,minimum width=0.95cm,minimum height=0.72cm,inner sep=0pt,anchor=north west,font=\scriptsize,text=gray},
  lab/.style={font=\scriptsize},
  note/.style={font=\scriptsize,anchor=north west,align=left,text width=3.6cm,inner sep=1pt}]
\node[lab,anchor=south,align=center] at (1.5,0.45) {$T_2\in D_2$, Case 1: $a_{\ell+2,2}\prec a_{\ell+1,3}$};
\node at (0.5,0.12) {$\vdots$};
\node[ghost] at (0,0) {$a_{\ell,1}$};
\node[ghost] at (1,0) {$a_{\ell,2}$};
\node[ghost] at (2,0) {$a_{\ell,3}$};
\node[cell,s1] at (0,-1) {$a_{\ell+1,1}$};
\node[cell,s2] at (1,-1) {$a_{\ell+1,2}$};
\node[cell,s3] at (2,-1) {$a_{\ell+1,3}$};
\node[cell,s1] at (0,-2) {$a_{\ell+2,1}$};
\node[cell,s2] at (1,-2) {$a_{\ell+2,2}$};
\node[cell] at (0,-3) {$a_{\ell+3,1}$};
\node at (0.5,-4.45) {$\vdots$};
\node[cell] at (0,-5) {$a_{s-1,1}$};
\node[cell] at (0,-6) {$a_{s,1}$};
\node at (0.5,-7.45) {$\vdots$};
\node[lab,anchor=east] at (-0.12,-1.45) {row $\ell{+}1$};
\node[lab,anchor=east] at (-0.12,-2.45) {row $\ell{+}2$};
\draw[decorate,decoration={brace,amplitude=4pt},thick] (-0.12,-5.95) -- (-0.12,-3.05) node[midway,left=5pt,lab,align=right] {all $\prec a_{\ell+2,2}$};
\node[note] (n22224) at (1.12,-6.05) {$s$ = smallest index with $a_{s,1}\nprec a_{\ell+2,2}$};
\node at (5.35,-4.2) {$\xrightarrow{\ \ g\ \ }$};
\node[lab,anchor=south,align=center] at (7.7,0.45) {$g(T_2)\in D_3$};
\node at (6.7,0.12) {$\vdots$};
\node[ghost] at (6.2,0) {$a_{\ell,1}$};
\node[ghost] at (7.2,0) {$a_{\ell,2}$};
\node[ghost] at (8.2,0) {$a_{\ell,3}$};
\node[cell,s1] at (6.2,-1) {$a_{\ell+1,1}$};
\node[cell,s2] at (7.2,-1) {$a_{\ell+1,2}$};
\node[cell,s1] at (6.2,-2) {$a_{\ell+2,1}$};
\node[cell,s3] at (7.2,-2) {$a_{\ell+1,3}$};
\node[cell] at (6.2,-3) {$a_{\ell+3,1}$};
\node at (6.7,-4.45) {$\vdots$};
\node[cell] at (6.2,-5) {$a_{s-1,1}$};
\node[cell,s2] at (6.2,-6) {$a_{\ell+2,2}$};
\node[cell] at (6.2,-7) {$a_{s,1}$};
\node at (6.7,-8.45) {$\vdots$};
\node[note,text=blue!70!black] (n33360) at (7.32,-5.45) {$a_{\ell+2,2}$ passes only entries $\prec a_{\ell+2,2}$, $\iv$ unchanged};
\node[lab,anchor=south,align=center] at (1.5,-10.15) {$T_2\in D_2$, Case 2a: $a_{\ell+2,2}\nprec a_{\ell+1,3}$,\\ $a_{\ell+1,2}\nprec a_{\ell+2,2}$, $a_{\ell+2,1}\nprec a_{\ell+1,2}$};
\node at (0.5,-10.48) {$\vdots$};
\node[ghost] at (0,-10.6) {$a_{\ell,1}$};
\node[ghost] at (1,-10.6) {$a_{\ell,2}$};
\node[ghost] at (2,-10.6) {$a_{\ell,3}$};
\node[cell,s1] at (0,-11.6) {$a_{\ell+1,1}$};
\node[cell,s2] at (1,-11.6) {$a_{\ell+1,2}$};
\node[cell,s3] at (2,-11.6) {$a_{\ell+1,3}$};
\node[cell,s12] at (0,-12.6) {$a_{\ell+2,1}$};
\node[cell,s23] at (1,-12.6) {$a_{\ell+2,2}$};
\node[cell] at (0,-13.6) {$a_{\ell+3,1}$};
\node at (0.5,-15.05) {$\vdots$};
\node[cell] at (0,-15.6) {$a_{s-1,1}$};
\node[cell] at (0,-16.6) {$a_{s,1}$};
\node at (0.5,-18.05) {$\vdots$};
\node[lab,anchor=east] at (-0.12,-12.05) {row $\ell{+}1$};
\node[lab,anchor=east] at (-0.12,-13.05) {row $\ell{+}2$};
\draw[decorate,decoration={brace,amplitude=4pt},thick] (-0.12,-16.55) -- (-0.12,-13.65) node[midway,left=5pt,lab,align=right] {all $\prec a_{\ell+1,2}$};
\node[note] (n10096) at (1.12,-13.65) {$a_{\ell+2,1}<a_{\ell+1,2}<a_{\ell+2,2}$, two incomparable pairs};
\node[note] (n20096) at (1.12,-16.65) {$s$ = smallest index with $a_{s,1}\nprec a_{\ell+1,2}$};
\node at (5.35,-14.8) {$\xrightarrow{\ \ g\ \ }$};
\node[lab,anchor=south,align=center] at (7.7,-10.15) {$g(T_2)\in D_3$};
\node at (6.7,-10.48) {$\vdots$};
\node[ghost] at (6.2,-10.6) {$a_{\ell,1}$};
\node[ghost] at (7.2,-10.6) {$a_{\ell,2}$};
\node[ghost] at (8.2,-10.6) {$a_{\ell,3}$};
\node[cell,s1] at (6.2,-11.6) {$a_{\ell+1,1}$};
\node[cell,s3] at (7.2,-11.6) {$a_{\ell+1,3}$};
\node[cell,s12] at (6.2,-12.6) {$a_{\ell+2,1}$};
\node[cell,s23] at (7.2,-12.6) {$a_{\ell+2,2}$};
\node[cell] at (6.2,-13.6) {$a_{\ell+3,1}$};
\node at (6.7,-15.05) {$\vdots$};
\node[cell] at (6.2,-15.6) {$a_{s-1,1}$};
\node[cell,s2] at (6.2,-16.6) {$a_{\ell+1,2}$};
\node[cell] at (6.2,-17.6) {$a_{s,1}$};
\node at (6.7,-19.05) {$\vdots$};
\node[note,text=blue!70!black] (n22224) at (7.32,-13.85) {$\iv -1: (a_{\ell+1,2},a_{\ell+2,1})$\\ $\iv +1: (a_{\ell+2,2},a_{\ell+1,2})$ \\ net change $0$};
\begin{scope}[yshift=-15.032cm,xshift=1.2cm]
\node[cell,s1,minimum width=0.45cm,minimum height=0.35cm] at (0,0) {}; \node[lab,anchor=west] at (0.45,-0.25) {$S_1$};
\node[cell,s2,minimum width=0.45cm,minimum height=0.35cm] at (1.0,0) {}; \node[lab,anchor=west] at (1.45,-0.25) {$S_2$};
\node[cell,s3,minimum width=0.45cm,minimum height=0.35cm] at (2.0,0) {}; \node[lab,anchor=west] at (2.45,-0.25) {$S_3$};
\node[cell,s12,minimum width=0.45cm,minimum height=0.35cm] at (3.0,0) {}; \node[lab,anchor=west] at (3.45,-0.25) {$S_1\cup S_2$};
\node[cell,s23,minimum width=0.45cm,minimum height=0.35cm] at (4.8,0) {}; \node[lab,anchor=west] at (5.25,-0.25) {$S_2\cup S_3$};
\end{scope}
\end{tikzpicture}
\caption{The map $g$ for Case 1 and Case $2(a)$.}
\label{g}
\end{figure}

\begin{figure}[h]
\centering
\begin{tikzpicture}[x=0.95cm,y=0.72cm,node distance=3pt,
  cell/.style={draw,thick,minimum width=0.95cm,minimum height=0.72cm,inner sep=0pt,anchor=north west,font=\scriptsize},
  s1/.style={fill=orange!30},
  s2/.style={fill=green!25},
  s3/.style={fill=cyan!25},
  s23/.style={fill=gray!20},
  s12/.style={fill=yellow!35},
  ghost/.style={draw,thick,gray!60,minimum width=0.95cm,minimum height=0.72cm,inner sep=0pt,anchor=north west,font=\scriptsize,text=gray},
  lab/.style={font=\scriptsize},
  note/.style={font=\scriptsize,anchor=north west,align=left,text width=3.6cm,inner sep=1pt}]
\node[lab,anchor=south,align=center] at (1.5,0.45) {$T_2\in D_2$, Case 2b: $a_{\ell+2,2}\nprec a_{\ell+1,3}$,\\ $a_{\ell+1,2}\nprec a_{\ell+2,2}$, $a_{\ell+2,1}\prec a_{\ell+1,2}$};
\node at (0.5,0.12) {$\vdots$};
\node[ghost] at (0,0) {$a_{\ell,1}$};
\node[ghost] at (1,0) {$a_{\ell,2}$};
\node[ghost] at (2,0) {$a_{\ell,3}$};
\node[cell,s1] at (0,-1) {$a_{\ell+1,1}$};
\node[cell,s2] at (1,-1) {$a_{\ell+1,2}$};
\node[cell,s3] at (2,-1) {$a_{\ell+1,3}$};
\node[cell,s1] at (0,-2) {$a_{\ell+2,1}$};
\node[cell,s23] at (1,-2) {$a_{\ell+2,2}$};
\node[cell] at (0,-3) {$a_{\ell+3,1}$};
\node at (0.5,-4.45) {$\vdots$};
\node[cell] at (0,-5) {$a_{s-1,1}$};
\node[cell] at (0,-6) {$a_{s,1}$};
\node at (0.5,-7.45) {$\vdots$};
\node[lab,anchor=east] at (-0.12,-1.45) {row $\ell{+}1$};
\node[lab,anchor=east] at (-0.12,-2.45) {row $\ell{+}2$};
\draw[decorate,decoration={brace,amplitude=4pt},thick] (-0.12,-5.95) -- (-0.12,-3.05) node[midway,left=5pt,lab,align=right] {all $\prec a_{\ell+1,2}$};
\node[note] (n12224) at (1.12,-3.05) {incomparable pair $a_{\ell+2,2}<a_{\ell+1,3}$};
\node[note] (n22224) at (1.12,-6.05) {$s$ = smallest index with $a_{s,1}\nprec a_{\ell+1,2}$};
\node at (5.35,-4.2) {$\xrightarrow{\ \ g\ \ }$};
\node[lab,anchor=south,align=center] at (7.7,0.45) {$g(T_2)\in D_3$};
\node at (6.7,0.12) {$\vdots$};
\node[ghost] at (6.2,0) {$a_{\ell,1}$};
\node[ghost] at (7.2,0) {$a_{\ell,2}$};
\node[ghost] at (8.2,0) {$a_{\ell,3}$};
\node[cell,s1] at (6.2,-1) {$a_{\ell+1,1}$};
\node[cell,s23] at (7.2,-1) {$a_{\ell+2,2}$};
\node[cell,s1] at (6.2,-2) {$a_{\ell+2,1}$};
\node[cell,s3] at (7.2,-2) {$a_{\ell+1,3}$};
\node[cell] at (6.2,-3) {$a_{\ell+3,1}$};
\node at (6.7,-4.45) {$\vdots$};
\node[cell] at (6.2,-5) {$a_{s-1,1}$};
\node[cell,s2] at (6.2,-6) {$a_{\ell+1,2}$};
\node[cell] at (6.2,-7) {$a_{s,1}$};
\node at (6.7,-8.45) {$\vdots$};
\node[note,text=blue!80!black] (n14912) at (8.32,-1.05) {switching $a_{\ell+1,3}$ and $a_{\ell+2,2}$ \\ $\iv -1: (a_{\ell+1,3},a_{\ell+2,2})$ \\ $\iv +1: (a_{\ell+2,2},a_{\ell+1,2})$ \\ net change $0$};

\node[lab,anchor=south,align=center] at (1.5,-10.15) {$T_2\in D_2$, Case 2c with $s=\ell+2$:\\ $a_{\ell+1,2}\prec a_{\ell+2,2}$, $a_{\ell+2,1}\nprec a_{\ell+1,2}$};
\node at (0.5,-10.48) {$\vdots$};
\node[ghost] at (0,-10.6) {$a_{\ell,1}$};
\node[ghost] at (1,-10.6) {$a_{\ell,2}$};
\node[ghost] at (2,-10.6) {$a_{\ell,3}$};
\node[cell,s1] at (0,-11.6) {$a_{\ell+1,1}$};
\node[cell,s2] at (1,-11.6) {$a_{\ell+1,2}$};
\node[cell,s3] at (2,-11.6) {$a_{\ell+1,3}$};
\node[cell,s12] at (0,-12.6) {$a_{\ell+2,1}$};
\node[cell,s3] at (1,-12.6) {$a_{\ell+2,2}$};
\node[cell] at (0,-13.6) {$a_{\ell+3,1}$};
\node[cell] at (0,-14.6) {$a_{\ell+4,1}$};
\node at (0.5,-16.05) {$\vdots$};
\node[lab,anchor=east] at (-0.12,-12.05) {row $\ell{+}1$};
\node[lab,anchor=east] at (-0.12,-13.05) {row $\ell{+}2$};
\node[note] (n10096) at (1.12,-13.65) {the search for $s$ starts and stops at row $\ell+2$ since $a_{\ell+2,1}\nprec a_{\ell+1,2}$};
\node at (5.35,-14) {$\xrightarrow{\ \ g\ \ }$};
\node[lab,anchor=south,align=center] at (7.7,-10.15) {$g(T_2)\in D_3$};
\node at (6.7,-10.48) {$\vdots$};
\node[ghost] at (6.2,-10.6) {$a_{\ell,1}$};
\node[ghost] at (7.2,-10.6) {$a_{\ell,2}$};
\node[ghost] at (8.2,-10.6) {$a_{\ell,3}$};
\node[cell,s1] at (6.2,-11.6) {$a_{\ell+1,1}$};
\node[cell,s3] at (7.2,-11.6) {$a_{\ell+1,3}$};
\node[cell,s2] at (6.2,-12.6) {$a_{\ell+1,2}$};
\node[cell,s3] at (7.2,-12.6) {$a_{\ell+2,2}$};
\node[cell,s12] at (6.2,-13.6) {$a_{\ell+2,1}$};
\node[cell] at (6.2,-14.6) {$a_{\ell+3,1}$};
\node[cell] at (6.2,-15.6) {$a_{\ell+4,1}$};
\node at (6.7,-17.05) {$\vdots$};
\node[note,text=blue!70!black] (n13360) at (8.32,-11.65) {only move comparable pairs, $\iv$ unchanged};

\begin{scope}[yshift=-13.832cm,xshift=1.2cm]
\node[cell,s1,minimum width=0.45cm,minimum height=0.35cm] at (0,0) {}; \node[lab,anchor=west] at (0.45,-0.25) {$S_1$};
\node[cell,s2,minimum width=0.45cm,minimum height=0.35cm] at (1.0,0) {}; \node[lab,anchor=west] at (1.45,-0.25) {$S_2$};
\node[cell,s3,minimum width=0.45cm,minimum height=0.35cm] at (2.0,0) {}; \node[lab,anchor=west] at (2.45,-0.25) {$S_3$};
\node[cell,s12,minimum width=0.45cm,minimum height=0.35cm] at (3.0,0) {}; \node[lab,anchor=west] at (3.45,-0.25) {$S_1\cup S_2$};
\node[cell,s23,minimum width=0.45cm,minimum height=0.35cm] at (4.8,0) {}; \node[lab,anchor=west] at (5.25,-0.25) {$S_2\cup S_3$};
\end{scope}
\end{tikzpicture}
\caption{The map $g$ for Case $2(b)$ and Case $2(c)$ when $s=\ell+2$.}
\label{g2}
\end{figure}

(Step 3): Let $\mathcal{A}=D_1\setminus\im(f),$ and $\mathcal{B}=D_3\setminus\im(g).$ We will explicitly characterize the two sets in the following lemmas.

\begin{lemma}\label{charA}
A tableau $T$ of shape $3^\ell2^11^{n-3\ell-2}$ lies in $\im(f)$ if and only if $\exists\, s\,\in[\ell+2,n-2\ell-1]$ satisfies Equation (\ref{conditioneqalpha}), so that $f^{-1}=\alpha$. Consequently, $\mathcal{A}$ is the set of $P$-tableaux of shape $3^\ell2^11^{n-3\ell-2}$, for which there is no $s\in[\ell+2,n-2\ell-1]$ satisfies Equation (\ref{conditioneqalpha}).
\end{lemma}

\begin{proof}
Let $T_1\in D_4$ and $T=f(T_1)$. Write $b=a_{\ell+1,2}(T_1)$ for the entry that $f$ inserts into the first column, and let $r$ be the row it occupies in $T$, so that $a_{r,1}(T)=b$. By the choice of the insertion position, it is clear that in $T,$ $a_{\ell+1,1},a_{\ell+2,1},\cdots,a_{r-1,1}\prec b\prec a_{\ell+1,3}(T_1)=a_{\ell+1,2}(T) $. Hence $s=r$ satisfies Equation (\ref{conditioneqalpha}).

For the other direction, suppose $\exists\,s$ satisfies Equation (\ref{conditioneqalpha}) in a $P$-tableau $T$ of shape $3^\ell2^11^{n-3\ell-2}$. Then $\alpha$ is in its Case (1) at $T$, and $\alpha(T)$ is obtained by moving $a_{s,1}$ into the position $(\ell+1,2)$ and moving $a_{\ell+1,2}$ to position $(\ell+1,3)$. Now $\alpha(T)\in D_4$. Comparing Figure \ref{coeff1} (top) and Figure \ref{f}, $f$ inserts $a_{s,1}$ back into the first column of $\alpha(T)$, so that $a_{s,1}$ returns to its original row $s,$ and $a_{\ell+1,2}$ returns to position $(\ell+1,2)$. Hence, $f(\alpha(T))=T$, so $T\in\im(f)$ and $\alpha=f^{-1}$. The description of $\mathcal{A}$ follows by taking complements.
\end{proof}

\begin{definition}
    Let $T$ be a $P$-tableau with entries $a_{i,j}$ and let $r\geq r_1$, we say $a_{r,1}$ is reachable from $a_{r_1,1}$ if $a_{r_1,1}, a_{r_1+1,1},\cdots, a_{r-1,1}\prec a_{r,1}.$ When $r=r_1$, we have $a_{r,1}=a_{r_1,1}$ and the condition is empty.
\end{definition}

\begin{lemma} \label{CharB}
    Let $T$ be a $P$-tableau of shape $3^\ell2^21^{n-3\ell-4}$. $T\in \im(g)$ if and only if $T$ satisfies one of the following conditions. Consequently, $\mathcal{B}$ consists of the tableaux of shape $3^\ell2^21^{n-3\ell-4}$ that satisfy none of the following.

\begin{enumerate}
    \item (Case 1) $a_{\ell+1,2}\prec a_{\ell+2,2}$. There exists some $r>\ell+2$ such that $a_{r,1}$ is reachable from $a_{\ell+2,1}$, and $a_{r,1}\prec a_{\ell+2,2}$.
    \item (Case 2a) $a_{\ell+1,2}\nprec a_{\ell+2,2}$. There exists some $r\geq \ell+3$ such that $a_{r,1}$ is reachable from $a_{\ell+3,1}$, and $a_{\ell+1,1}\prec a_{r,1}\prec a_{\ell+1,2}$, $a_{r,1}\nprec a_{\ell+2,2}$ and $a_{\ell+2,1}\nprec a_{r,1}$.
    \item (Case 2b) $a_{\ell+1,2}\nprec a_{\ell+2,2}$ and $a_{\ell+2,1}\prec a_{\ell+1,2}$. There exists some $r\geq \ell+3$ such that $a_{r,1}$ is reachable from $a_{\ell+1,1}$. $a_{r,1}\prec a_{\ell+2,2}$, $a_{r,1}\nprec a_{\ell+1,2}$.
    \item (Case 2c) $a_{\ell+1,2}\nprec a_{\ell+2,2}$. There exists some $r\geq \ell+3$ such that $a_{r,1}$ is reachable from $a_{\ell+1,1}$. $a_{r,1}\prec a_{\ell+1,2}$ and $a_{r,1}\prec a_{\ell+2,2}$.
    \item (Case 2c with $s=\ell+2$)
    $a_{\ell+1,2}\nprec a_{\ell+2,2},$ and $a_{\ell+1,1}\prec a_{\ell+2,1}\prec a_{\ell+1,2}$, and $a_{\ell+3,1}\prec a_{\ell+2,2}$.
\end{enumerate}
\end{lemma}

\begin{proof}
    Let $T=g(T_2)$ for some $T_2\in D_2$. It is straightforward to check that $T$ satisfies one of the entry conditions, depending on which case $T_2$ falls into.

    For the other direction, suppose $T$ satisfies one of the entry conditions and define $T_2$ by undoing the corresponding case of the map $g$ (Figure \ref{g}, \ref{g2}). Specifically, for $(1)$, move $a_{r,1}, a_{\ell+2,2}$ to positions $(\ell+2,2), (\ell+1,3)$ respectively. For $(2)$ and $(4)$, move $a_{r,1}, a_{\ell+1,2}$ to positions $(\ell+1,2), (\ell+1,3)$ respectively. For $(3)$, move $a_{r,1}, a_{\ell+2,2}, a_{\ell+1,2}$ to positions $(\ell+1,2), (\ell+1,3), (\ell+2,2)$ respectively. For $(5)$, move $a_{\ell+2,1}, a_{\ell+1,2}, a_{\ell+3,1}$ to positions $(\ell+1,2), (\ell+1,3), (\ell+2,1)$ respectively. One can verify that $T_2$ is a $P$-tableau using the relations in each of the five cases. Further, a direct comparison shows $T_2$ falls exactly into one of Cases 1, 2a, 2b, or 2c. Thus $g(T_2)=T$. The description of $\mathcal{B}$ follows by taking complements.
\end{proof}

(Step 4): We now construct the injection $\phi:\mathcal{B}\to \mathcal{A}$. For convenience, in $T\in \mathcal{B}$ we write 
\[p=a_{\ell+1,1},\quad q=a_{\ell+1,2},\quad u=a_{\ell+2,1},\quad z=a_{\ell+2,2},\quad c_i=a_{\ell+2+i,1}\ \ (1\leq i\leq k),\]
where $k=n-3\ell-4$. We consider $\phi$ for a few cases, and its basic move is the map $\phi_1$ of Figure \ref{phi1fig}. Find the smallest index $s\in[\ell+3,n-2\ell-2]$ such that $a_{s,1}\nprec z$ and insert $z$ above $a_{s,1}$. If no such $s$ exists, move $z$ to the bottom of the first column. Rows $1,\dots,\ell+1$ all stay. The following lemma describes a certain case when the map $\phi_1$ has to be adjusted.

\begin{figure}[h]
\centering
\begin{tikzpicture}[x=0.95cm,y=0.72cm,node distance=3pt,
  cell/.style={draw,thick,minimum width=0.95cm,minimum height=0.72cm,inner sep=0pt,anchor=north west,font=\scriptsize},
  s23/.style={fill=gray!20},
  s12/.style={fill=yellow!35},
  ghost/.style={draw,thick,gray!60,minimum width=0.95cm,minimum height=0.72cm,inner sep=0pt,anchor=north west,font=\scriptsize,text=gray},
  lab/.style={font=\scriptsize}]
\begin{scope}
\node[lab,anchor=south] at (1.5,0.55) {$T\in\mathcal{B}$};
\node at (0.5,0.15) {$\vdots$};
\node[ghost] at (0,0) {$a_{\ell,1}$}; \node[ghost] at (1,0) {$a_{\ell,2}$}; \node[ghost] at (2,0) {$a_{\ell,3}$};
\node[cell,s12] at (0,-1) {$p$}; \node[cell,s23] at (1,-1) {$q$};
\node[cell,s12] at (0,-2) {$u$}; \node[cell,s23] at (1,-2) {$z$};
\node[cell] at (0,-3) {$c_1$};
\node at (0.5,-4.45) {$\vdots$};
\node[cell] at (0,-5) {$c_{a-1}$};
\node[cell] at (0,-6) {$a_{s,1}$};
\node at (0.5,-7.45) {$\vdots$};
\node[cell] at (0,-8) {$c_k$};
\node[lab,anchor=east] at (-0.15,-1.45) {row $\ell{+}1$};
\node[lab,anchor=east] at (-0.15,-2.45) {row $\ell{+}2$};
\draw[decorate,decoration={brace,amplitude=4pt},thick] (-0.15,-5.95) -- (-0.15,-3.05)
   node[midway,left=5pt,lab,align=right] {all $\prec z$};
\node[lab,anchor=west,align=left] at (1.1,-6.45) {first entry of the\\ singletons $\nprec z$};
\end{scope}
\node at (4.6,-1.4) {$\xrightarrow{\ \ \phi_1\ \ }$};
\begin{scope}[xshift=6.0cm]
\node[lab,anchor=south] at (1.5,0.55) {$\phi_1(T)$};
\node at (0.5,0.15) {$\vdots$};
\node[ghost] at (0,0) {$a_{\ell,1}$}; \node[ghost] at (1,0) {$a_{\ell,2}$}; \node[ghost] at (2,0) {$a_{\ell,3}$};
\node[cell,s12] at (0,-1) {$p$}; \node[cell,s23] at (1,-1) {$q$};
\node[cell,s12] at (0,-2) {$u$};
\node[cell] at (0,-3) {$c_1$};
\node at (0.5,-4.45) {$\vdots$};
\node[cell] at (0,-5) {$c_{a-1}$};
\node[cell,s23] at (0,-6) {$z$};
\node[cell] at (0,-7) {$a_{s,1}$};
\node at (0.5,-8.45) {$\vdots$};
\node[cell] at (0,-9) {$c_k$};
\draw[decorate,decoration={brace,amplitude=4pt},thick] (-0.15,-5.95) -- (-0.15,-1.95)
   node[midway,left=5pt,lab,align=right] {all $\prec z$};
\node[lab,anchor=west,align=left,text=blue!70!black] at (1.1,-3.45) {$z$ passes only $u,c_1,\dots,c_{a-1}$,\\ all comparable to $z$. $\iv$ unchanged};
\node[lab,anchor=west,align=left] at (1.1,-6.45) {$z$ sits above $a_{s,1}\nprec z$\\ (or at the bottom)};
\end{scope}
\begin{scope}[yshift=-7.78cm,xshift=1.2cm]
\node[cell,s12,minimum width=0.45cm,minimum height=0.35cm] at (0,0) {}; \node[lab,anchor=west] at (0.45,-0.25) {$S_1\cup S_2$};
\node[cell,s23,minimum width=0.45cm,minimum height=0.35cm] at (1.8,0) {}; \node[lab,anchor=west] at (2.25,-0.25) {$S_2\cup S_3$};
\end{scope}
\end{tikzpicture}
\caption{The basic move $\phi_1$ for $T\in\mathcal{B}$.}
\label{phi1fig}
\end{figure}

\begin{lemma}\label{phi1basic}
Let $T\in \mathcal{B}$. Then $\phi_1(T)$ is a $P$-tableau of shape $3^\ell2^11^{n-3\ell-2}$ and $\iv(\phi_1(T))=\iv(T)$. Further, \[\phi_1(T)\notin \mathcal{A}\qquad\Longleftrightarrow\qquad p\prec u\prec q.\]
\end{lemma}
 
\begin{proof}
The entries $c_1,\dots,c_{a-1}=c_{s-\ell-3}$ passed over by $z$ all satisfy $c_t\prec z$ by the choice of $s$. These comparable pairs contribute no inversion either before or after the move. All other entries are untouched, hence inversions are preserved. The result is a $P$-tableau since $c_{s-\ell-3}\prec z$ and $c_{s-\ell-2}\nprec z$. By Lemma \ref{charA}, $\phi_1(T)\notin \mathcal{A}$ exactly when
\begin{equation}\label{Echain}
p,\,u,\,c_1,\cdots,\,c_{r-1}\ \prec\ c_r\ \prec\ q
\end{equation}
for some $r\geq \ell+2$. For $r=\ell+2$ this becomes $p\prec u\prec q$, which gives the implication $\Leftarrow$.
 
For $\Rightarrow$, suppose Condition (\ref{Echain}) holds with $r>\ell+2$. $p\prec c_r\prec q$ forces $c_r\in S_2$ by (F2). The column condition of $T$ gives $z\nprec q$, so $c_r<z$ by (F1). $z$ cannot lie above $c_r$, otherwise $z\prec c_r\in S_2$ and hence $z\in S_1$, which is impossible since $z\in S_2\cup S_3.$ Therefore $z$ is inserted passing over $c_r,$ giving $c_r\prec z.$ Now reading these conditions back in $T$ gives
\[c_1,\cdots,c_{r-1}\prec c_r,\qquad p\prec c_r\prec q,\qquad u\prec c_r\prec z,\qquad q\nprec z\]
(The last since $q\in S_3$ is maximal.) This is exactly the case (4) of Lemma \ref{CharB}, contradicting $T\in \mathcal{B}$. Hence \eqref{Echain} holds only at $r=\ell+2$.
\end{proof}

Now when $p\prec u\prec q$, where $u\in S_2$ and $q,z\in S_3$ by (F2), we replace map $\phi_1$ by the adjusted map $\phi_2$ of Figure \ref{phi2fig}. In $\phi_2(T)$, row $\ell+1$ becomes $(p,u)$, and $q,z$ become the first two entries of the column of singletons followed by $c_1,\cdots,c_k$. We can examine $\phi_2(T)$: It is a $P$-tableau. Column conditions hold as in $T$. $u\nprec a_{\ell,2}$ since both lie in $S_2$. $c_1\nprec z$ since otherwise $T$ would fall into Condition (5) of Lemma \ref{CharB}. The pairs changing relative positions are $\{p,q\}, \{u,p\}, \{u,q\}, \{u,z\}$, which are comparable, so that $\iv(\phi_2(T))=\iv(T)$. Moreover, $\phi_2(T)\in \mathcal{A}$ since no entry $a_{s,1}$ exists such that $p\prec a_{s,1}\prec u$, as $p\in S_1$ and $u\in S_2$. Clearly $T$ is recovered from $\phi_2(T)$ so that $\phi_2$ is injective. In addition, it is a direct comparison that the images of $\phi_2$ are disjoint from those of $\phi_1$ and also the map $\phi_3$ we define below. For example, the first entry of the singletons (from the top) is $q\in S_3$ in $\phi_2(T)$, but $u\in S_1\cup S_2$ in $\phi_1(T)$ or $\phi_3(T)$.

\begin{figure}[h]
\centering
\begin{tikzpicture}[x=0.95cm,y=0.72cm,node distance=3pt,
  cell/.style={draw,thick,minimum width=0.95cm,minimum height=0.72cm,inner sep=0pt,anchor=north west,font=\scriptsize},
  s1/.style={fill=orange!30},
  s2/.style={fill=green!25},
  s3/.style={fill=cyan!25},
  s23/.style={fill=gray!20},
  s12/.style={fill=yellow!35},
  ghost/.style={draw,thick,gray!60,minimum width=0.95cm,minimum height=0.72cm,inner sep=0pt,anchor=north west,font=\scriptsize,text=gray},
  lab/.style={font=\scriptsize}]
\begin{scope}
\node[lab,anchor=south] at (1.5,0.55) {$T\in\mathcal{B}$ with $p\prec u\prec q$};
\node at (0.5,0.15) {$\vdots$};
\node[ghost] at (0,0) {$a_{\ell,1}$}; \node[ghost] at (1,0) {$a_{\ell,2}$}; \node[ghost] at (2,0) {$a_{\ell,3}$};
\node[cell,s1] at (0,-1) {$p$}; \node[cell,s3] at (1,-1) {$q$};
\node[cell,s2] at (0,-2) {$u$}; \node[cell,s3] at (1,-2) {$z$};
\node[cell] at (0,-3) {$c_1$};
\node[cell] at (0,-4) {$c_2$};
\node at (0.5,-5.45) {$\vdots$};
\node[cell] at (0,-6) {$c_k$};
\node[lab,anchor=east] at (-0.15,-1.45) {row $\ell{+}1$};
\node[lab,anchor=east] at (-0.15,-2.45) {row $\ell{+}2$};
\end{scope}
\node at (4.6,-1.4) {$\xrightarrow{\ \ \phi_2\ \ }$};
\begin{scope}[xshift=7.0cm]
\node[lab,anchor=south] at (1.5,0.55) {$\phi_2(T)\in\mathcal{A}$};
\node at (0.5,0.15) {$\vdots$};
\node[ghost] at (0,0) {$a_{\ell,1}$}; \node[ghost] at (1,0) {$a_{\ell,2}$}; \node[ghost] at (2,0) {$a_{\ell,3}$};
\node[cell,s1] at (0,-1) {$p$}; \node[cell,s2] at (1,-1) {$u$};
\node[cell,s3] at (0,-2) {$q$};
\node[cell,s3] at (0,-3) {$z$};
\node[cell] at (0,-4) {$c_1$};
\node[cell] at (0,-5) {$c_2$};
\node at (0.5,-6.45) {$\vdots$};
\node[cell] at (0,-7) {$c_k$};
\node[lab,anchor=west,align=left,text=blue!70!black] at (2.1,-1.45) {only comparable pairs move.\\ $\iv$ unchanged};

\end{scope}
\begin{scope}[yshift=-6.34cm,xshift=1.2cm]
\node[cell,s1,minimum width=0.45cm,minimum height=0.35cm] at (0,0) {}; \node[lab,anchor=west] at (0.45,-0.25) {$S_1$};
\node[cell,s2,minimum width=0.45cm,minimum height=0.35cm] at (1.0,0) {}; \node[lab,anchor=west] at (1.45,-0.25) {$S_2$};
\node[cell,s3,minimum width=0.45cm,minimum height=0.35cm] at (2.0,0) {}; \node[lab,anchor=west] at (2.45,-0.25) {$S_3$};
\end{scope}
\end{tikzpicture}
\caption{The map $\phi_2$ for $T\in \mathcal{B}$ with $p\prec u\prec q$.}
\label{phi2fig}
\end{figure}

Let $\mathcal{B}^-=\{T\in \mathcal{B}\,\, \text{with}\,\, p\prec u\prec q\}$ and $\mathcal{B}'=\mathcal{B}\setminus\mathcal{B^-}.$ We showed that $\phi_2$ is injective on $\mathcal{B^-}$, and $\phi_1$ on $\mathcal{B'}$ takes values in $\mathcal{A}$ by Lemma \ref{phi1basic}. We have not discussed the injectivity of $\phi_1$ yet. In fact, $\phi_1$ is not injective in general. We describe it in the following lemma.

\begin{lemma}\label{phifibre}
Let $U$ be in the image of $\phi_1$, and $\mathcal{S}=\{T\in \mathcal{B},\,\,\phi_1(T)=U\}$. Then \begin{enumerate}
    \item[(a)] $|\mathcal{S}|\leq2$. If $\mathcal{S}=\{T,T'\}$ with $T\neq T'$, then $a_{\ell+2,2}(T)$ and $a_{\ell+2,2}(T')$ are two entries $w\neq z$ of the first column of $U$ with $w\in S_2, z\in S_3, w\prec z$, and $w$ lies above $z$.
    \item[(b)] $w$ is reachable from $u$, so $u\in S_1$.
    \item[(c)] Pairs $(w,q), (q,z)$ are incomparable and $w<q<z$. $u\prec q$.
\end{enumerate}
\end{lemma}
 
\begin{proof}
(a) A tableau $T$ with $\phi_1(T)=U$ is recovered from $U$ by choosing an entry $z$ of the first column of $U$ below row $\ell+2$, pulling it back into the position $(\ell+2,2)$. Every entry between $a_{\ell+2,1}$ and $z$ is $\prec z$ so that inversion is unchanged. If $z,z'$ are two choices with $z$ above $z'$, then we have $a_{\ell+2,1},\cdots,\prec z\cdots\prec z'$. Thus $z\in S_2$ and $z'\in S_3$. A third choice of $z''$ would give a chain of length four $a_{\ell+2,1}\prec z\prec z'\prec z''$, contradicting Lemma \ref{remark 1}.

(b) $\phi_1$ inserts $w$ into the first column of $T$ and it passes over entries all $\prec w$, including $u$ to get $U$, so $w$ is reachable from $u$, so $u\in S_1$.

(c) Since $T'\in \mathcal{B}$ does not satisfy Case (1) of Lemma \ref{CharB}, we have $q\nprec z$. Now $w<q$ by (F1). $w\nprec q$ by the column condition in $T$ and $q<z$ by (F1). Then $u\prec q$ since $u\prec w$.
\end{proof}

Now we define the third move $\phi_3$ to deal with the failure of $\phi_1$ being injective. Following Lemma \ref{phifibre}, we set $U$ be the image of $T$ and define map $\phi_3$ on all $T'$ as described in the lemma. Specifically, for $T\in \mathcal{B'}$, let $w$ be the first singleton below $u$ such that $w\nprec q$ (if it exists). Then we say $T$ satisfies condition $(\star)$ if \[w\,\, \text{exists},\,\, u\prec w\prec z,\,\,q\nprec z.\]
Thus we have $w\in S_2, z\in S_3$ by (F2), and $w<q, q<z$ by (F1). See Figure \ref{phi3fig} for an example of such $T$. Define $\phi_3$ on all $T\in\mathcal{B'}$ satisfy $(\star)$ as shown in Figure \ref{phi3fig}. Insert $q$ into the first column above the first singleton below $w$ that is $\nprec q$ (at the bottom of the first column if there is none), and move $z$ to position $(\ell+1,2)$.
 
\begin{figure}[h]
\centering
\begin{tikzpicture}[x=0.95cm,y=0.72cm,node distance=3pt,
  cell/.style={draw,thick,minimum width=0.95cm,minimum height=0.72cm,inner sep=0pt,anchor=north west,font=\scriptsize},
  s1/.style={fill=orange!30},
  s2/.style={fill=green!25},
  s3/.style={fill=cyan!25},
  s23/.style={fill=gray!20},
  s12/.style={fill=yellow!35},
  ghost/.style={draw,thick,gray!60,minimum width=0.95cm,minimum height=0.72cm,inner sep=0pt,anchor=north west,font=\scriptsize,text=gray},
  lab/.style={font=\scriptsize}]
\begin{scope}
\node[lab,anchor=south] at (1.5,0.55) {$T\in\mathcal{B'}$ satisfying $(\star)$};
\node at (0.5,0.15) {$\vdots$};
\node[ghost] at (0,0) {$a_{\ell,1}$}; \node[ghost] at (1,0) {$a_{\ell,2}$}; \node[ghost] at (2,0) {$a_{\ell,3}$};
\node[cell,s12] at (0,-1) {$p$}; \node[cell,s23] at (1,-1) {$q$};
\node[cell,s1] at (0,-2) {$u$}; \node[cell,s3] at (1,-2) {$z$};
\node[cell] at (0,-3) {$c_1$};
\node at (0.5,-4.45) {$\vdots$};
\node[cell] at (0,-5) {$c_{a-1}$};
\node[cell,s2] at (0,-6) {$w$};
\node[cell] at (0,-7) {$c_{a+1}$};
\node at (0.5,-8.45) {$\vdots$};
\node[cell] at (0,-9) {$c_k$};
\node[lab,anchor=east] at (-0.15,-1.45) {row $\ell{+}1$};
\node[lab,anchor=east] at (-0.15,-2.45) {row $\ell{+}2$};
\draw[decorate,decoration={brace,amplitude=4pt},thick] (-0.15,-5.95) -- (-0.15,-2.05)
   node[midway,left=5pt,lab,align=right] {all $\prec q$};
\node[lab,anchor=west,align=left] at (1.1,-6.45) {first entry of the\\ singletons not $\prec q$};
\end{scope}
\node at (4.6,-1.4) {$\xrightarrow{\ \ \phi_3\ \ }$};
\begin{scope}[xshift=6.0cm]
\node[lab,anchor=south] at (1.5,0.55) {$\phi_3(T)\in\mathcal{A}$};
\node at (0.5,0.15) {$\vdots$};
\node[ghost] at (0,0) {$a_{\ell,1}$}; \node[ghost] at (1,0) {$a_{\ell,2}$}; \node[ghost] at (2,0) {$a_{\ell,3}$};
\node[cell,s12] at (0,-1) {$p$}; \node[cell,s3] at (1,-1) {$z$};
\node[cell,s1] at (0,-2) {$u$};
\node[cell] at (0,-3) {$c_1$};
\node at (0.5,-4.45) {$\vdots$};
\node[cell] at (0,-5) {$c_{a-1}$};
\node[cell,s2] at (0,-6) {$w$};
\node[cell] at (0,-7) {$c_{a+1}$};
\node at (0.5,-8.45) {$\vdots$};
\node[cell] at (0,-9) {$c_{b-1}$};
\node[cell,s23] at (0,-10) {$q$};
\node[cell] at (0,-11) {$c_b$};
\node at (0.5,-12.45) {$\vdots$};
\draw[decorate,decoration={brace,amplitude=4pt},thick] (-0.15,-9.95) -- (-0.15,-7)
   node[midway,left=5pt,lab,align=right] {all $\prec q$};
\draw[decorate,decoration={brace,mirror,amplitude=4pt},thick] (-0.15,-2.1) -- (-0.15,-5.95)
   node[midway,left=5pt,lab,align=right] {all $\prec q$};
\node[lab,anchor=west,align=left] at (1.1,-11.45) {$c_b\nprec q$, the second entry \\ of singletons not $\prec q$ \\ (or $q$ is last)};
\node[lab,anchor=west,align=left,text=blue!80!black] at (2.1,-1.45) {$\iv +1: (z,q)$};
\node[lab,anchor=west,align=left,text=blue!80!black] at (1.1,-6.45) {$\iv -1: (q,w)$ \\net change $0$};
\end{scope}
\begin{scope}[yshift=-10.26cm,xshift=1.2cm]
\node[cell,s1,minimum width=0.45cm,minimum height=0.35cm] at (0,0) {}; \node[lab,anchor=west] at (0.45,-0.25) {$S_1$};
\node[cell,s2,minimum width=0.45cm,minimum height=0.35cm] at (1.0,0) {}; \node[lab,anchor=west] at (1.45,-0.25) {$S_2$};
\node[cell,s3,minimum width=0.45cm,minimum height=0.35cm] at (2.0,0) {}; \node[lab,anchor=west] at (2.45,-0.25) {$S_3$};
\node[cell,s12,minimum width=0.45cm,minimum height=0.35cm] at (3.0,0) {}; \node[lab,anchor=west] at (3.45,-0.25) {$S_1\cup S_2$};
\node[cell,s23,minimum width=0.45cm,minimum height=0.35cm] at (4.8,0) {}; \node[lab,anchor=west] at (5.25,-0.25) {$S_2\cup S_3$};
\end{scope}
\end{tikzpicture}
\caption{The map $\phi_3$ on $T\in \mathcal{B'}$ satisfying condition $(\star)$.}
\label{phi3fig}
\end{figure}

First, $\phi_3(T)$ is a $P$-tableau. Row $\ell+1$ is a chain since $q<z$. $q$ passes over all entries comparable to it except $w$, and is inserted above the entry $c_b\nprec q$. Second, $\iv(\phi_3(T))=\iv (T)$. The only incomparable pairs with a change in relative position are $(q,z)$ and $(q,w)$. The new inversion $(z,q)$ is offset by the disappeared inversion $(q,w)$. Third, $\phi_3(T)\in \mathcal{A}$. Otherwise, by Lemma \ref{charA}, there exists a reachable $x$ of the first column of $\phi_3(T)$ from $p$ and $x\prec z$. Note first, $x\neq u$ since $p\nprec u$ (as $T\notin \mathcal{B}^-$). Second, $x$ lies above $q$ since otherwise $q\prec x$ for $x\in S_2$ and $q\in S_2\cup S_3$, a contradiction. Now if we read all entries back in $T$, we have $T$ satisfies either Case (3) or Case (4) of Lemma \ref{CharB}, contradicting $T\in \mathcal{B}$. Thus no such $x$ exists and $\phi_3(T)\in \mathcal{A}$.

Finally, we can define $\phi$ as follows. 

\begin{definition}
Let $\mathcal{C}\subseteq \mathcal{B'}$ be the set of $P$-tableaux satisfying $(\star)$. For $T\in \mathcal{B}$, define
\[\phi(T)=\begin{cases}\phi_1(T) & \text{if } T\in\mathcal{B}'\setminus\mathcal{C}, \\\phi_2(T) & \text{if } T\in\mathcal{B}^-,\\ \phi_3(T) & \text{if } T\in\mathcal{C}. \end{cases}\]
\end{definition}

By construction on $\phi_1, \phi_2, \phi_3$ and Lemma \ref{phi1basic}, $\phi(T)\in \mathcal{A}$ and $\iv(\phi(T))=\iv(T)$ for all $T\in \mathcal{B}$. It only remains to show that $\phi$ is injective.

We already showed the injectivity of $\phi_2$. Now let $T_1\neq T_2$ in $\mathcal{B'}$ with $\phi(T_1)=\phi(T_2)$. If $T_1, T_2\in\mathcal{B}'\setminus\mathcal{C}$, then $\phi_1(T_1)=\phi_1(T_2)$. Lemma \ref{phifibre} shows that one of $T_1$ and $T_2$ is in $\mathcal{C}$, a contradiction. If $T_2\in \mathcal{C}$ and $T_1\in \mathcal{B}'\setminus\mathcal{C}$, then $\phi_1(T_1)=\phi_3(T_2)=U$. We will show that $T_1\in \mathcal{C}$ to derive a contradiction. Using the notation in Figure \ref{phi3fig}, let $c=(c_1,\cdots,c_{b-1},q,c_b,\cdots,c_k)$ be the singletons below $u$ in first column of $U$. $T_1$ has the same rows $1,\cdots,\ell+1$ as $U$, and also $a_{\ell+2,1}(T_1)=u$. Let $a_{\ell+2,2}(T_1)=z_1$ be the entry inserted into $c$ under $\phi_1$. Thus the rows $\ell+1, \ell+2$ of $T_1$ are $p\prec z$ and $u\prec z_1$. Since in $U$, every entry above $q$ is $\prec z$ (including $w$) and each such entry is also $\prec q$ (except $w$),  but $z_1\nprec z$ by column condition in $T_1$. Then $z_1$ lies below $q$ ($z_1\neq q$ since $w\nprec q$). $z_1>z$ by (F1). Since $z_1$ is reachable from $u$ in $c$, so $q\prec z_1$, and $q\in S_2, z_1\in S_3$. Since $c_1,\cdots,c_{b-1}\prec z$ but $q\nprec z$, so there exists a $w_1=q$ as the first singleton below $u$ in $T_1$ that is $\nprec z$. Now we check that all conditions are satisfied for $T_1\in \mathcal{C}$ (i.e., satisfies $(\star)$): $w_1=q$ exists, $u\prec w_1\prec z_1$, and $z\nprec z_1$.

Finally we are left with the case $T_1, T_2\in \mathcal{C}$, so that $\phi_3(T_1)=\phi_3(T_2)$. Then $T_1$ and $T_2$ have the same entries $p,z,u$. For $i=1,2$, let $q_i=a_{\ell+1,2}(T_i)$. We can assume $q_1$ lies above $q_2$. Using $(\star)$, let $w_i=w(T_i)$, where $w_i$ is the unique singleton above $q_i$ that is $\nprec q_i$, and $w_i\in S_2$. Also $w_i\prec z$, $q_i<z$. If $q_1,q_2\in S_3$, then $q_1\nprec q_2$ and $q_1$ lies above $q_2$, so $q_1=w_2\in S_2$, a contradiction. If $q_1\in S_2$ and $q_2\in S_3$, then $q_1\neq w_2$ since $w_2\prec z$ but $q_1\nprec z$. Thus $q_1\prec q_2$, giving $q_2>z$ by (F1), a contradiction. If $q_1\in S_2\cup S_3$, and $q_2\in S_2$, then $w_1\in S_2$ lies above $q_1$ above $q_2$. Since $w_1\nprec q_2$ by (F2), so $w_1=w_2$, forcing $q_1\prec q_2$, a contradiction by (F2).

To conclude, $\phi$ is injective, and the proof for $m(\lambda)=1$ is complete.

Lastly, if $m(\lambda)=2$, in which case $\lambda_1=\lambda_2+1, \lambda_2=\lambda_3+1$ and $n=3\ell+3$ so that $\lambda=(\ell+2,\ell+1,\ell)$. Essentially the coefficient reduces to two terms $[e_\lambda]X_G(\mathbf{x},q)=B(\ell,1)-2B(\ell+1,0)$. More precisely, $B(\ell+1,0)\neq 0$ only if $\ell+1\leq \min\{n_1,n_2,n_3\}$, so that $n_1=n_2=n_3=\ell+1$. Otherwise, $B(\ell+1,0)=0$ and $[e_\lambda]X_G(\mathbf{x},q)$ is clearly positive. Now following previous notation, $D_1$ and $D_4$ are the sets of $P$-tableaux of shapes $3^\ell2^11^1$ and $3^{\ell+1}$ counted by $B(\ell,1)$ and $B(\ell+1,0)$. Note that $D_2=D_3=\emptyset$. We consider an extra inversion preserving injections besides $f: D_4\to D_1$ with disjoint images. 

Define the map $f': D_4\to D_1$ such that for $T\in D_4$, $f'(T)$ is obtained by moving $a_{\ell+1,3}$ to the bottom of the first column. Clearly $a_{\ell+1,3}\in S_3$, which is comparable to $a_{\ell+1,1}$ and $a_{\ell+1,2}$. Thus $f'(T)$ is a $P$-tableau and inversion is unchanged. $f'$ in injective since $T$ can be recovered by pulling the singleton entry back into position $(\ell+1,3)$.

Therefore, for any $T\in D_4,$ we take two copies of $T$, one is mapped under $f$ and the other is mapped under $f'$. Then we obtain two disjoint images in $D_1$ since they have distinct singletons $a_{\ell+2,1}$ ($a_{\ell+2,1}\in S_2$ under $f$ and $\in S_3$ under $f'$). Thus \[B(\ell,1)-2B(\ell+1,0)=\sum_{T\in D_1\setminus(\im(f)\cup\im(f'))}q^{\iv(T)}\ \in\ \bN[q].\]

This completes the proof of Theorem \ref{e_{(n-2l-1,l+1,l)}}.
\end{proof}

\begin{remark}\label{knowncases}
Our proofs in Theorem \ref{e_{(n-2l,l,l)}} and \ref{e_{(n-2l-1,l+1,l)}} treat all values of $\ell$ when $r=3$. For every natural unit interval order (an arbitrary $r$), certain special cases have been proved. The coefficient of $e_n (\ell=0)$ lies in $\bN[q]$ by \cite{SHARESHIAN2016497}. The hook shapes was proved by Hwang \cite[Theorem 4.34]{Hwang2022} and Blasiak, Eriksson, Pylyavskyy and Siegl \cite[Theorem 3.32]{BEPS2025}, which covers the case $\ell=1$ of Theorem \ref{e_{(n-2l,l,l)}} and the case $\ell=0$ of Theorem \ref{e_{(n-2l-1,l+1,l)}}. The shape $\lambda=(k,2^a)$ was proved by Siegl \cite[Theorem 1.9]{Siegl2025}, which covers the case $\ell=2$ of Theorem \ref{e_{(n-2l,l,l)}}.
\end{remark}

\section{$e$-positivity for posets with an antichain of size $n-3$}\label{section4}

Throughout this section $P=P(\mathbf{d})$ is a natural unit interval order on $[n]$ with $r=3$, $G=\inc(P)$. We prove Conjecture \ref{SWepositive} for Dyck paths of bounce number three whose corresponding posets contain an antichain with $n-3$ elements. Since the antichains of $P$ are exactly the cliques of $G$, these are also equivalent to incomparability graphs that contain a clique with $n-3$ vertices.
 
\begin{proposition}\label{equvcondition}
Let $P=P(\mathbf{d})$ be a natural unit interval order on $[n]$ with $r=3$ and $G=\inc(P)$. The following conditions are equivalent:
\begin{enumerate}
\item $P$ contains an antichain with $n-3$ elements.
\item $d_i\geq n-4+i$ for some $i\in\{1,2,3,4\}$.
\item the associated partition $\tau$ does not contain the staircase $(4,3,2,1)$.
\end{enumerate}
\end{proposition}

\begin{proof}
 $(1)\Leftrightarrow(2)$: Let $A$ be an antichain of $P$ and $i$ be the smallest element in $A$. Then for any $a\in A$, $i\nprec a$ and $a\leq d_i$. Therefore $A\subseteq\{i,i+1,\dots,d_i\}$. Then $P$ contains an antichain with $n-3$ elements if and only if $d_i-i+1\geq n-3$, which gives $d_i\geq n-4+i$. Since $d_i\leq n$ for all $i$, then we have $1\leq i\leq 4$.

 $(2)\Leftrightarrow(3)$: Since $\tau_i=\#\{\,j\in[n]\;:\;d_j<n+1-i\}$, $(4,3,2,1)\subseteq\tau$ if and only if $d_4\leq n-1$, $d_3\leq n-2$, $d_2\leq n-3$ and $d_1\leq n-4$, which is the negation of $(2)$.
\end{proof}

Figure \ref{exthm1} shows two examples for $d_3\geq n-1$ (left) with $\mathbf{d}=(d_1,d_2,n-1,\cdots,n-1,n,\cdots,n)$, and $d_4=n$ (right) with $\mathbf{d}=(d_1,d_2,d_3,n,\cdots,n)$. The first family contains the hook shapes $\tau=(\tau_1,1,\dots,1)$, for which Conjecture \ref{SWepositive} was proved in \cite{Cho2019OnEA}. We will see in Proposition \ref{transpose} that transposing $\tau$ exchanges the conditions $d_3\geq n-1$ and $d_2\geq n-2$, also the conditions $d_4=n$ and $d_1\geq n-3$.

\begin{figure}[h]
    \centering
\begin{tikzpicture}[every node/.style={minimum size=.3cm-\pgflinewidth, outer sep=0pt}]
\draw[step=0.3cm] (-1.5001,-1.5001) grid (1.5,1.5);
\draw (-1.5,-1.5) -- (1.5,1.5);
\draw [red,very thick] (-1.5,-1.5)--(-1.5,-.3)--(-1.2,-.3)--(-1.2,0)--(-.9,0)--(-.9,1.2)--(.6,1.2)--(.6,1.5)--(1.5,1.5);
\draw [green,ultra thick, dashed] (-1.5,-1.5)--(-1.5,-.3)--(-.3,-.3)--(-.3,1.2)--(1.2,1.2)--(1.2,1.5)--(1.5,1.5);
    \node[fill=yellow!30] at (-1.35,1.35) {};
    \node[fill=yellow!30] at (-1.05,1.35) {};
    \node[fill=yellow!30] at (-.75,1.35) {};
    \node[fill=yellow!30] at (-0.45,1.35) {};
    \node[fill=yellow!30] at (-0.15,1.35) {};
    \node[fill=yellow!30] at (0.15,1.35) {};
    \node[fill=yellow!30] at (0.45,1.35) {};
    \node[fill=yellow!30] at (-1.35,1.05) {};
    \node[fill=yellow!30] at (-1.05,1.05) {};
    \node[fill=yellow!30] at (-1.35,.75) {};
    \node[fill=yellow!30] at (-1.05,.75) {};
    \node[fill=yellow!30] at (-1.35,.45) {};
    \node[fill=yellow!30] at (-1.05,.45) {};
    \node[fill=yellow!30] at (-1.35,.15) {};
    \node[fill=yellow!30] at (-1.05,.15) {};
    \node[fill=yellow!30] at (-1.35,-.15) {};
\end{tikzpicture}
\qquad
\hspace{2cm}
\begin{tikzpicture}[every node/.style={minimum size=.3cm-\pgflinewidth, outer sep=0pt}]
\draw[step=0.3cm] (-1.5001,-1.5001) grid (1.5,1.5);
\draw (-1.5,-1.5) -- (1.5,1.5);
\draw [red,very thick] (-1.5,-1.5)--(-1.5,-.9)--(-1.2,-.9)--(-1.2,-.3)--(-.9,-.3)--(-.9,.6)--(-.6,.6)--(-.6,1.5)--(1.5,1.5);
\draw [green, ultra thick, dashed] (-1.5,-1.5)--(-1.5,-.9)--(-.9,-.9)--(-.9,.6)--(.6,.6)--(.6,1.5)--(1.5,1.5);
    \node[fill=yellow!30] at (-1.35,1.35) {};
    \node[fill=yellow!30] at (-1.35,1.05) {};
    \node[fill=yellow!30] at (-1.35,.75) {};
    \node[fill=yellow!30] at (-1.35,.45) {};
    \node[fill=yellow!30] at (-1.35,.15) {};
    \node[fill=yellow!30] at (-1.35,-.15) {};
    \node[fill=yellow!30] at (-1.35,-.45) {};
    \node[fill=yellow!30] at (-1.35,-.75) {};
    \node[fill=yellow!30] at (-1.05,1.35) {};
    \node[fill=yellow!30] at (-1.05,1.05) {};
    \node[fill=yellow!30] at (-1.05,.75) {};
    \node[fill=yellow!30] at (-1.05,.45) {};
    \node[fill=yellow!30] at (-1.05,.15) {};
    \node[fill=yellow!30] at (-1.05,-.15) {};
    \node[fill=yellow!30] at (-.75,1.35) {};
    \node[fill=yellow!30] at (-.75,1.05) {};
    \node[fill=yellow!30] at (-.75,.75) {};
\qquad
\end{tikzpicture}
\caption{Two natural unit interval orders satisfying the conditions of Proposition \ref{equvcondition}. Left: $\mathbf{d}=(4,5,9,9,9,9,9,10,10,10)$ with $d_3\geq n-1$, where $\{3,\dots,9\}$ is an antichain and $\tau=(7,2,2,2,2,1)$. Right: $\mathbf{d}=(2,4,7,10,10,10,10,10,10,10)$ with $d_4=n$, where $\{4,\dots,10\}$ is an antichain and $\tau=(3,3,3,2,2,2,1,1)$.}
\label{exthm1}
\end{figure}
 
\begin{theorem}\label{epos1}
If the equivalent conditions in Proposition \ref{equvcondition} hold, then $X_G(\mathbf{x},q)$ is $e$-positive.    
\end{theorem}

The rest of this section focuses on the proof of Theorem \ref{epos1}. In Section \ref{subsecgeneral} we show that the expansion of Proposition \ref{ecoeff} has only six terms, and we collect some general tools to help us simplify the case discussions later. Sections \ref{subsecfamI} and \ref{subsecfamIII} prove for the two subfamilies $d_3\geq n-1$ and $d_4=n$. The remaining two cases follow by transposing $\tau$.
 
\subsection{The expansion of $X_G(\mathbf{x},q)$ and some general tools}\label{subsecgeneral} The following lemma refines Lemma \ref{P exist}.
 
\begin{lemma}\label{rooklemma}
Let $P=P(\mathbf{d})$ be a natural unit interval order on $[n]$ (of any bounce number) with associated partition $\tau$. If a $P$-tableau of shape $3^\ell2^j1^{n-3\ell-2j}$ exists, then $\tau$ contains the staircase $(2\ell+j,2\ell+j-1,\dots,1)$. In particular, if $\tau$ does not contain $(4,3,2,1)$, then $B(\ell,j)=0$ unless $2\ell+j\leq3$.
\end{lemma}
 
\begin{proof}
Let $T$ be a $P$-tableau of shape $3^\ell2^j1^{n-3\ell-2j}$. Each horizontally adjacent pair $x\prec y$ of $T$ gives a cell with coordinate $(x,y)$ of $\tau$. There are $2\ell+j$ such pairs in $T$. Since every entry of $T$ is distinct, we have $2\ell+j$ cells of $\tau$ in distinct columns and distinct rows. It is straightforward to see that $(2\ell+j, 2\ell+j-1,\cdots,1)$ is the smallest partition containing such $2\ell+j$ cells in distinct columns and rows. Thus $(2\ell+j, 2\ell+j-1,\cdots,1)\subseteq \tau$. In particular, if $\tau$ does not contain $(4,3,2,1)$, then $B(\ell,j)=0$ unless $2\ell+j\leq3$.
\end{proof}

\begin{remark}
    To prove Theorem \ref{epos1}, we can assumne $n\geq 6$. If $n\leq 5$, then every partition $\lambda$ has $\lambda_1\leq 1$, and $\lambda_2-\lambda_3\leq 1$ unless $\lambda=(2,2)$ or $\lambda=(3,2)$. Hence the coefficients of all $e_\lambda$ other than $e_{(2,2)}$ and $e_{(3,2)}$ lie in $\bN[q]$ by Proposition \ref{ecoeff}, Theorems \ref{e_{(n-2l,l,l)}} and \ref{e_{(n-2l-1,l+1,l)}}. It is easy to check these two particular coefficients when $n\leq 5$, and they are also covered in \cite{Siegl2025}.
\end{remark}

Throughout this section, we now can assume $n\geq 6$. Using Lemma \ref{rooklemma}, Proposition \ref{ecoeff}, Equation (\ref{expansion}) reduces to
\begin{align*}\label{expansionepos1}
\numberthis
X_G(\mathbf{x},q)={}&\big(B(0,0)+B(1,0)-B(0,1)\big)\,e_{(n)}+\big(B(0,1)+B(1,1)-B(0,2)-B(1,0)\big)\,e_{(n-1,1)}\\[-1pt]
&+\big(B(0,2)-B(0,3)-B(1,0)\big)\,e_{(n-2,2)}+\big(B(0,3)-B(1,1)\big)\,e_{(n-3,3)}\\
&+\big(B(1,0)-B(1,1)\big)\,e_{(n-2,1,1)}+B(1,1)\,e_{(n-3,2,1)}.
\end{align*}

The coefficients of $e_{(n)}$ and $e_{(n-2,1,1)}$ are proved to be in $\bN[q]$ by Theorem \ref{e_{(n-2l,l,l)}} for $\ell=0$ and $\ell=1$. The coefficient of $e_{(n-1,1)}$ lies in $\bN[q]$ is due to Theorem \ref{e_{(n-2l-1,l+1,l)}} for $\ell=0$. The coefficient of $e_{(n-3,2,1)}$ is trivially in $\bN[q]$. Therefore, to prove Theorem \ref{epos1}, it suffices to show
\begin{equation}\label{L1L2}
B(0,2)-B(0,3)-B(1,0)\in\bN[q]\qquad\text{and}\qquad B(0,3)-B(1,1)\in\bN[q].
\end{equation}

\begin{remark}\label{knowncases4}
By Remark \ref{knowncases}, the coefficient of $e_{(n-2,2)}$ was known to lie in $\bN[q]$ by Siegl \cite{Siegl2025}. Nevertheless, we still prove it (Lemmas \ref{lemman-2} and \ref{lemman-2III}) so that the proof of Theorem \ref{epos1} is self-contained. Moreover, the injections $\sigma_1,\sigma_2$ describe $[e_{(n-2,2)}]X_G(\mathbf{x},q)$ as the inversion generating function of an explicit set of $P$-tableaux of shape $2^21^{n-4}$, which can be compared with the family $K_P(\lambda)$ of \cite{Siegl2025}. The result of $[e_{(n-3,3)}]X_G(\mathbf{x},q)$ lies in $\bN[q]$ (Lemmas \ref{lemman-3} and \ref{lemman-3III}) is new.
\end{remark}

We first have the following observations on the transpose of $\tau$.
 
\begin{proposition}\label{transpose}
Let $\mathbf{d}^{\rm tr}$ be the Dyck path whose partition is the conjugate $\tau'$ of $\tau$, and let $P^{\rm tr}=P(\mathbf{d}^{\rm tr})$ and $G^{\rm tr}=\inc(P^{\rm tr})$. Then $x\prec y$ in $P$ if and only if $n+1-y\prec n+1-x$ in $P^{\rm tr}$. Consequently, 
\begin{enumerate}
    \item $P$ and $P^{\rm tr}$ have the same bounce number, and $P^{\rm tr}$ satisfies the conditions of Proposition \ref{equvcondition} if and only if $P$ does.
    \item $d_3\geq n-1$ holds for $\mathbf{d}$ if and only if $d_2\geq n-2$ holds for $\mathbf{d}^{\rm tr}$, and $d_4=n$ holds for $\mathbf{d}$ if and only if $d_1\geq n-3$ holds for $\mathbf{d}^{\rm tr}$.
    \item $X_{G^{\rm tr}}(\mathbf{x},q)=X_G(\mathbf{x},q)$.
\end{enumerate}
\end{proposition}

\begin{proof}
Counting rows from top to bottom and columns from left to right, a cell at row $r$ and column $c$ corresponds to $c\prec n+1-r$ in $P$. Its conjugate is the cell at row $c$ and column $r$, which corresponds to $r\prec n+1-c$ in $P^{\rm tr}$. Let $x=c$ and $y=n+1-r$, we have $x\prec y$ in $P$ if and only if $n+1-y\prec n+1-x$ in $P^{\rm tr}$. Thus $i\mapsto n+1-i$ is an order-reversing bijection from $P$ to $P^{\rm tr}$, and the size of chains and antichains are both preserved. Hence $P$ and $P^{\rm tr}$ have the same bounce number, and Condition (1) of Proposition \ref{equvcondition} holds for $P$ if and only if it holds for $P^{\rm tr}$, so as Condition (2) and (3). In addition, $d_3\geq n-1$ for $\mathbf{d}$ means $\tau_2\leq2$, giving $\tau'_3\leq1$, and $d_2\geq n-2$ for $\mathbf{d}^{\rm tr}$. Similarly, $d_4=n$ for $\mathbf{d}$ means $\tau_1\leq3$, giving $\tau'_4=0$, and $d_1\geq n-3$ for $\mathbf{d}^{\rm tr}$.

Now for vertex $i$ in $G$, the map $i\mapsto n+1-i$ gives $G\cong G^{\rm tr}$. However, it reverses the natural order of pairs of vertices and turns the ascents of a proper coloring into descents. Thus \[X_{G^{\rm tr}}(\mathbf{x},q)=\sum_{\kappa}q^{\operatorname{des}(\kappa)}\mathbf{x}_\kappa,\]
where the sum ranges over the proper colorings $\kappa$ of $G$ and $\operatorname{des}(\kappa)$ is the number of edges $\{i,j\}$ of $G$ with $i<j$ and $\kappa(i)>\kappa(j)$. By \cite{SHARESHIAN2016497}, $X_{G^{\rm tr}}(\mathbf{x},q)=q^{|E(G)|}X_G(\mathbf{x},q^{-1})$ and it is palindromic, hence $X_{G^{\rm tr}}(\mathbf{x},q)=X_G(\mathbf{x},q)$.
\end{proof}

The next proposition and its corollary allow us to focus on connected incomparability graphs.

\begin{proposition}\label{reductions}
Let $P=P(\mathbf{d})$ be a natural unit interval order on $[n]$ and $G=\inc(P)$, and suppose that $d_k=k$ for some $1\leq k<n$. Then $G$ is the disjoint union of its induced subgraphs $G_1$ on $\{1,\dots,k\}$ and $G_2$ on $\{k+1,\dots,n\}$, where $G_1, G_2$ are the incomparability graphs of the natural unit interval orders $P(\mathbf{d}_1)$ and $P(\mathbf{d}_2)$ with
\[\mathbf{d}_1=(d_1,\dots,d_k),\qquad\qquad \mathbf{d}_2=(d_{k+1}-k,\dots,d_n-k).\]
Moreover, $X_G(\mathbf{x},q)=X_{G_1}(\mathbf{x},q)\,X_{G_2}(\mathbf{x},q)$, and the bounce number $r(\mathbf{d})=r(\mathbf{d}_1)+r(\mathbf{d}_2)$. In particular, if $X_{G_1}(\mathbf{x},q)$ and $X_{G_2}(\mathbf{x},q)$ are $e$-positive then so is $X_G(\mathbf{x},q)$.
\end{proposition}
 
\begin{proof}
Since $i\leq k<j$ and so $d_i\leq d_k=k<j$ give $i\prec j$, $\{1,\cdots,k\}$ and $\{k+1,\cdots,n\}$ are two disconnected components of $G$. The relations $i\prec j$ for $i,j
\leq k$ give relations of $P(\mathbf{d_1})$. The relations $i\prec j$ for $i,j>k$ give $i-k\prec j-k$ as the relations of $P(\mathbf{d_2})$. A proper coloring of $G$ is preserved as the union of a proper coloring of $G_1$ and $G_2$, and the number of ascents is the sum of the numbers of ascents of the two components. Thus $X_G(\mathbf{x},q)=X_{G_1}(\mathbf{x},q)\,X_{G_2}(\mathbf{x},q)$. Clearly $r(\mathbf{d})=r(\mathbf{d}_1)+r(\mathbf{d}_2)$ since $\mathbf{d}_1\cup \mathbf{d}_2$ is exactly $\mathbf{d}$.
\end{proof}
 
\begin{corollary}\label{connected}
Let $P=P(\mathbf{d})$ be a natural unit interval order on $[n]$ with $r=3$ and $G=\inc(P)$. If $d_k=k$ for some $1\leq k<n$, then $X_G(\mathbf{x},q)$ is $e$-positive. 
\end{corollary}
 
\begin{proof}
By Proposition \ref{reductions}, $X_G=X_{G_1}X_{G_2}$ with $r(\mathbf{d}_1)+r(\mathbf{d}_2)=3$. So one of the two sub-paths $\mathbf{d}_1$ and $\mathbf{d}_2$ has bounce number one. Say $\mathbf{d}_1=(m,\cdots,m)$ has $r(\mathbf{d}_1)=1$, and its incomparability graph is the complete graph $K_m$. Hence $X_{K_m}(\mathbf{x},q)=[m]_q!\,e_m$ with $[m]_q!=\prod_{i=1}^m(1+q+\dots+q^{i-1})$(See \cite{SHARESHIAN2016497}). A Dyck path with bounce number two is proved by \cite{Cho2019OnEA}, \cite{ALCO_2019__2_6_1059_0}. Hence both factors are $e$-positive, and so is $X_G$.
\end{proof}

Consequently, to prove Conjecture \ref{SWepositive} for Dyck paths of bounce number three, one may assume that $G$ is connected, which is equivalent to $d_i>i$ for all $i<n$. We will make this assumption for the rest of this section. In particular, $d_1\geq 2$ so that $1\nprec 2$.

\subsection{The case $d_3\geq n-1$}\label{subsecfamI}
 
In this subsection $\mathbf{d}=(d_1,d_2,n-1,\dots,n-1,n,\dots,n)$. We have the following fact, which will be used repeatedly later.
\begin{enumerate}
    \item[(F3)] Since $d_1\geq2$, $S_1=\{1,2,\cdots,d_1\}$, $S_2=\{d_1+1,\cdots,n-1\}$ and $S_3=\{n\}$. Every relation $x\prec y$ with $y\neq n$ has $x\in\{1,2\}$, and two elements $x,y\geq3$ except $n$ are incomparable. In particular, let $T$ be a $P$-tableau of shape $3^12^11^{n-5}$ with entries $a_{i,j}$, we have $a_{1,3}=n$, and $\{a_{1,1},a_{2,1}\}=\{1,2\}$.
\end{enumerate}

\begin{lemma}\label{lemman-3}
Let $\mathbf{d}$ be as above with $r=3, d_1\geq 2, n\geq6$. Then $B(0,3)-B(1,1)\in\bN[q]$.
\end{lemma}
 
\begin{proof}
Let $D$ be the set of $P$-tableaux of shape $3^12^11^{n-5}$ and $E$ the set of $P$-tableaux of shape $2^31^{n-6}$. We construct an injection $\psi:D\to E$ with $\iv(\psi(T))=\iv(T)$, so that $B(0,3)-B(1,1)=\sum_{U\in E\setminus\psi(D)}q^{\iv(U)}$. Let $T\in D$. In every case below, $\psi(T)$ is obtained by rearranging the entries in the first three rows. The rows of index at least four are unchanged. Using (F3) we discuss four cases according to whether $a_{2,2}\prec n$ and whether $a_{3,1}\prec n$.

\begin{itemize}
    \item Case 1: $a_{2,2}\prec n$ and $a_{3,1}\prec n$.

      $a_{2,1}\in\{1,2\}$ and $a_{2,1}\prec n$ by (F3). Define $\psi(T)$ by moving $n$ to the position $(3,2)$, see Figure \ref{psicase1and2a} (left). The entries passed by $n$ are both $\prec n$, so inversion is unchanged.

    \item Case 2: $a_{2,2}\nprec n$ and $a_{3,1}\prec n$.

    Since $a_{1,2}\prec n$ and $a_{3,1}\prec n$ while $a_{2,2}\nprec n$, (F1) gives $a_{1,2}<a_{2,2}$ and $a_{3,1}<a_{2,2}$. The entries $a_{1,2},a_{2,2},a_{3,1}$ are $\geq 3$ so that they are pairwise incomparable. We have three subcases.

\textbf{Case 2a:} $a_{2,1}\nprec a_{3,1}$. Define $\psi(T)$ as in Figure \ref{psicase1and2a} (right), where the second and third rows become $a_{3,1}\prec n$ and $a_{2,1}\prec a_{2,2}$. The column conditions all hold so that $\psi(T)$ is a $P$-tableau. The inversion $(a_{3,1},a_{2,1})$ is increased, while the inversion $(a_{2,2},a_{3,1})$ disappears. Hence $\iv(\psi(T))=\iv(T)$.

      \begin{figure}[h]
   \centering
   \ytableausetup{mathmode, boxsize=2em}
\begin{ytableau}
\scriptstyle a_{1,1} & \scriptstyle a_{1,2} & n \\
\scriptstyle a_{2,1} & \scriptstyle a_{2,2} \\
\scriptstyle a_{3,1} \\
\none[\vdots]  \\
\end{ytableau}\,\,\,
$\longrightarrow\,\,\,$
\begin{ytableau}
\scriptstyle a_{1,1} & \scriptstyle a_{1,2} \\
\scriptstyle a_{2,1} & \scriptstyle a_{2,2} \\
\scriptstyle a_{3,1} & n\\
\none[\vdots]  \\
\end{ytableau}
\qquad
\hspace{.8cm}
\begin{ytableau}
\scriptstyle a_{1,1} & \scriptstyle a_{1,2} & n \\
\scriptstyle a_{2,1} & \scriptstyle a_{2,2} \\
\scriptstyle a_{3,1} \\
\none[\vdots]  \\
\end{ytableau}\,\,\,
$\longrightarrow\,\,\,$
\begin{ytableau}
\scriptstyle a_{1,1} & \scriptstyle a_{1,2} \\
\scriptstyle a_{3,1} & n \\
\scriptstyle a_{2,1} & \scriptstyle a_{2,2}\\
\none[\vdots]  \\
\end{ytableau}
\qquad
\caption{The injective map $\psi$: Case 1 (left) and Case 2a (right).}
\label{psicase1and2a}  
\end{figure}

\textbf{Case 2b:} $a_{2,1}\prec a_{3,1}$ and $a_{2,1}\nprec a_{1,2}$. Define $\psi(T)$ as in Figure \ref{psicase2bc} (left). Specifically, we have $a_{1,1}=1, a_{2,1}=2$ otherwise $a_{1,1}=2\prec a_{1,2}$ contradicts $a_{2,1}=1\nprec a_{1,2}$. Thus $2\prec a_{3,1}$ and $2\nprec a_{1,2}$ give $a_{1,2}<a_{3,1}$ by (F1). Now the rows 1,2,3 become $1\prec a_{2,2}$, $2\prec a_{3,1}$, $a_{1,2}\prec n$. Note that $1<2\prec a_{2,2}$, and the column conditions $2\nprec 1$, $a_{3,1}\nprec a_{2,2}$, $a_{1,2}\nprec 2$, $n\nprec a_{3,1}$ hold. In $\psi(T)$, the inversions $(a_{1,2},2)$ and $(n,a_{2,2})$ disappear while the inversions $(a_{2,2},a_{1,2})$ and $(a_{3,1},a_{1,2})$ appear. All other moved pairs are comparable so $\iv(\psi(T))=\iv(T)$.

\begin{figure}[h]
   \centering
   \ytableausetup{mathmode, boxsize=2em}
\begin{ytableau}
1 & \scriptstyle a_{1,2} & n \\
2 & \scriptstyle a_{2,2} \\
\scriptstyle a_{3,1} \\
\none[\vdots]  \\
\end{ytableau}\,\,\,
$\longrightarrow\,\,\,$
\begin{ytableau}
1 & \scriptstyle a_{2,2} \\
2 & \scriptstyle a_{3,1} \\
\scriptstyle a_{1,2} & n\\
\none[\vdots]  \\
\end{ytableau}
\qquad
\hspace{10mm}
\begin{ytableau}
\scriptstyle a_{1,1} & \scriptstyle a_{1,2} & n \\
\scriptstyle a_{2,1} & \scriptstyle a_{2,2} \\
\scriptstyle a_{3,1} \\
\none[\vdots]  \\
\end{ytableau}\,\,\,
$\longrightarrow\,\,\,$
\begin{ytableau}
\scriptstyle a_{1,1} & \scriptstyle a_{2,2} \\
\scriptstyle a_{2,1} & \scriptstyle a_{1,2} \\
\scriptstyle a_{3,1} & n\\
\none[\vdots]  \\
\end{ytableau}
\qquad
\caption{The injective map $\psi$: Case 2b (left) and Case 2c (right).}
\label{psicase2bc}  
\end{figure}

\textbf{Case 2c:} $a_{2,1}\prec a_{3,1}$ and $a_{2,1}\prec a_{1,2}$. Define $\psi(T)$ as in Figure \ref{psicase2bc} (right). The first row becomes $a_{1,1}\prec a_{2,2}$ since if $a_{1,1}=1$, then $1<2=a_{2,1}\prec a_{2,2}$, and if $a_{1,1}=2$, then $2\prec a_{1,2}<a_{2,2}$. Row 2 and 3 become $a_{2,1}\prec a_{1,2}$ and $a_{3,1}\prec n$. All column conditions hold and the only pairs of incomparable entries whose relative position change are $(n,a_{2,2})$ (inversion -1) and $(a_{2,2}, a_{1,2})$ (inversion +1). Again $\iv(\psi(T))=\iv(T)$.

\item Case 3: $a_{2,2}\prec n$ and $a_{3,1}\nprec n$. 

By (F1), $a_{2,2}<a_{3,1}$ and so $a_{2,1}\prec a_{2,2}<a_{3,1}$. Define $\psi(T)$ as shown in Figure \ref{psicase3}. The three rows become $a_{1,1}\prec a_{1,2}$, $a_{2,1}\prec a_{3,1}$, $a_{2,2}\prec n$. The inversion $(n,a_{3,1})$ disappears while the inversion $(a_{3,1},a_{2,2})$ appears so that $\iv(\psi(T))=\iv(T)$.
 
\begin{figure}[h]
   \centering
   \ytableausetup{mathmode, boxsize=2em}
\begin{ytableau}
\scriptstyle a_{1,1} & \scriptstyle a_{1,2} & n \\
\scriptstyle a_{2,1} & \scriptstyle a_{2,2} \\
\scriptstyle a_{3,1} \\
\none[\vdots]  \\
\end{ytableau}\,\,\,
$\longrightarrow\,\,\,$
\begin{ytableau}
\scriptstyle a_{1,1} & \scriptstyle a_{1,2} \\
\scriptstyle a_{2,1} & \scriptstyle a_{3,1} \\
\scriptstyle a_{2,2} & n\\
\none[\vdots]  \\
\end{ytableau}
\qquad
\caption{The injective map $\psi$: Case 3.}
\label{psicase3}  
\end{figure}

\item Case 4: $a_{2,2}\nprec n$ and $a_{3,1}\nprec n$.

We discuss three subcases below. In Case 4a and 4b, we assume $a_{2,1}\prec a_{3,1}$. Since $a_{1,2}\prec n$, (F1) gives $a_{1,2}<a_{2,2}, a_{3,1}$.

\textbf{Case 4a:} $a_{2,1}\prec a_{1,2}$. Define $\psi(T)$ as shown in Figure \ref{psicase4ab} (left). One can verify all row and column conditions hold. Two inversions $(n,a_{2,2})$ and $(n,a_{3,1})$ disappear while two inversions $(a_{2,2},a_{1,2})$ and $(a_{3,1},a_{1,2})$ appear, preserving $\iv(T)$.

\begin{figure}[h]
   \centering
   \ytableausetup{mathmode, boxsize=2em}
\begin{ytableau}
\scriptstyle a_{1,1} & \scriptstyle a_{1,2} & n \\
\scriptstyle a_{2,1} & \scriptstyle a_{2,2} \\
\scriptstyle a_{3,1} \\
\none[\vdots]  \\
\end{ytableau}\,\,\,
$\longrightarrow\,\,\,$
\begin{ytableau}
\scriptstyle a_{1,1} & \scriptstyle a_{2,2} \\
\scriptstyle a_{2,1} & \scriptstyle a_{3,1} \\
\scriptstyle a_{1,2} & n\\
\none[\vdots]  \\
\end{ytableau}
\qquad
\hspace{10mm}
\begin{ytableau}
1 & \scriptstyle a_{1,2} & n \\
2 & \scriptstyle a_{2,2} \\
\scriptstyle a_{3,1} \\
\none[\vdots]  \\
\end{ytableau}\,\,\,
$\longrightarrow\,\,\,$
\begin{ytableau}
1 & \scriptstyle a_{2,2} \\
\scriptstyle a_{1,2} & n \\
2 & \scriptstyle a_{3,1}\\
\none[\vdots]  \\
\end{ytableau}
\caption{The injective map $\psi$: Case 4a (left) and Case 4b (right).}
\label{psicase4ab}  
\end{figure}

\textbf{Case 4b:} $a_{2,1}\nprec a_{1,2}$. Similar to Case 2b, we have $a_{1,1}=1$ and $a_{2,1}=2$. Define $\psi(T)$ as shown in Figure \ref{psicase4ab} (right). Now only the inversion $(n, a_{2,2})$ disappears and the inversion $(a_{2,2}, a_{1,2})$ appears, preserving $\iv(T)$.

\textbf{Case 4c:} $a_{2,1}\nprec a_{3,1}$. Then $a_{1,2}<a_{2,2}, a_{3,1}$ by (F1). Also (F1) gives $a_{3,1}<a_{2,2}$ since $a_{2,1}\nprec a_{3,1}$ and $a_{2,1}\prec a_{2,2}$. $a_{2,1}\nprec a_{1,2}$ since $a_{2,1}\prec a_{1,2}<a_{3,1}$ would give $a_{2,1}\prec a_{3,1}$. Further, $a_{1,1}=1$ and $a_{2,1}=2$. Otherwise $a_{1,1}=2\prec a_{1,2}$ and $a_{2,1}=1\nprec a_{3,1}$, contradicting $a_{1,2}<a_{3,1}$. Now define $\psi(T)$ as shown in Figure \ref{psicase4c}. The inversions $(n,a_{3,1})$ and $(a_{2,2},a_{3,1})$ disappear while the inversions $(a_{3,1},a_{1,2})$ and $(a_{3,1},2)$ appear, preserving $\iv(T)$.

\begin{figure}[h]
   \centering
   \ytableausetup{mathmode, boxsize=2em}
\begin{ytableau}
1 & \scriptstyle a_{1,2} & n \\
2 & \scriptstyle a_{2,2} \\
\scriptstyle a_{3,1} \\
\none[\vdots]  \\
\end{ytableau}\,\,\,
$\longrightarrow\,\,\,$
\begin{ytableau}
1 & \scriptstyle a_{3,1} \\
\scriptstyle a_{1,2} & n \\
2 & \scriptstyle a_{2,2}\\
\none[\vdots]  \\
\end{ytableau}
\caption{The injective map $\psi$: Case 4c.}
\label{psicase4c}  
\end{figure}

\end{itemize}

Let $U=\psi(T)$ with rows $u_{1,1}\prec u_{1,2}$, $u_{2,1}\prec u_{2,2}$, $u_{3,1}\prec u_{3,2}$. In every case $\psi$ moves the six entries of $T$ to fixed positions, so $T$ is determined by $U$ once the case is known, and it suffices to read off the case from $U$ to verify the injectivity of $\psi$. In particular, $u_{2,2}=n$ in Cases 2a, 4b, 4c and $u_{3,2}=n$ in all other cases. If $u_{2,2}=n$, then in Case 2a, $u_{1,2}=a_{1,2}\prec n$, but $u_{1,2}=a_{2,2}\nprec n$ in Case 4b and $u_{1,2}=a_{3,1}\nprec n$ in Case 4c. Additionally, Cases 4b and 4c are distinguished since $u_{3,1}\prec u_{1,2}$ in Case 4b and $u_{3,1}\nprec u_{1,2}$ in Case 4c. A similar brute force check if $u_{3,2}=n$ would distinguish the images in all other cases and we leave it to the readers.

\end{proof}
 
\begin{lemma}\label{lemman-2}
Let $\mathbf{d}$ be as above with $r=3, d_1\geq 2, n\geq6$. Then $B(0,2)-B(0,3)-B(1,0)\in\bN[q]$.
\end{lemma}
 
\begin{proof}
Let $E$, $D_1$ and $D_2$ be the sets of $P$-tableaux of shapes $2^21^{n-4}$, $2^31^{n-6}$ and $3^11^{n-3}$, respectively. Our strategy is to construct inversion-preserving injections $\sigma_1:D_1\to E$ and $\sigma_2:D_2\to E$ with $\im(\sigma_1)\cap\im(\sigma_2)=\emptyset$. Thus $B(0,2)-B(0,3)-B(1,0)=\sum_{U}q^{\iv(U)}$, where the sum ranges over $E\setminus(\im(\sigma_1)\cup\im(\sigma_2))$.
 
\textbf{The map $\sigma_1$.} Throughout Section \ref{subsecfamI}, we have $d_3\geq n-1$. So for any chain of length two, it must follow one of the forms $1\prec x$ with $x>d_1,$ $2\prec x$ with $x>d_2,$ or $x\prec n$ with $x\geq 3$. Now let $T\in D_1$ with rows 1,2,3 as $a_{1,1}\prec a_{1,2}$, $a_{2,1}\prec a_{2,2}$, $a_{3,1}\prec a_{3,2}$, and singletons $a_{4,1},\cdots, a_{n-3,1}$. Define $\sigma_1$ as follows: Find the smallest index $s$ for $s\in [4, n-3]$ such that $a_{s,1}\nprec a_{3,2}$. Then insert the entry $a_{3,2}$ of $T$ into the first column right above $a_{s,1}$ (at the bottom if no such $s$ exists). See Figure \ref{n-22sigma1}. Note that if $a_{3,2}\neq n$, then all singletons $\nprec a_{3,2}$, so that $s=4$. The entries passed by $a_{3,2}$ are all comparable to it and other column conditions are inherited from $T$, thus $\sigma_1(T)\in E$ and $\iv(\sigma_1(T))=\iv(T)$.

Let $U=\sigma_1(T)$ with first two rows $u_{1,1}\prec u_{1,2}$, $u_{2,1}\prec u_{2,2}$ and singletons $u_3,\cdots,u_{n-2}$. Let $u_w=a_{3,2}$ be the inserted entry so that $u_w$ is reachable from $u_3$. We claim that $u_w$ is unique. Otherwise if $u_{w'}$ is another one, then $u_3\prec u_w\prec u_{w'}$, which forces $u_{w'}=n$, a contradiction since pulling back $u_w$ would leave $n$ in the first column. Consequently $T$ is recovered from $U$ uniquely so that $\sigma_1$ is injective.

\begin{figure}[h]
   \centering
   \ytableausetup{mathmode, boxsize=2em}
\begin{ytableau}
\scriptstyle a_{1,1} & \scriptstyle a_{1,2} \\
\scriptstyle a_{2,1} & \scriptstyle a_{2,2} \\
\scriptstyle a_{3,1} & \scriptstyle a_{3,2}\\
\scriptstyle a_{4,1} \\
\none[\vdots]  \\
\end{ytableau}\,\,\,
$\longrightarrow\,\,\,$
\begin{ytableau}
\scriptstyle a_{1,1} & \scriptstyle a_{1,2} \\
\scriptstyle a_{2,1} & \scriptstyle a_{2,2} \\
\scriptstyle a_{3,1} \\
\none[\vdots]  \\
\scriptstyle a_{s-1,1} \\
\scriptstyle a_{3,2} \\
\scriptstyle a_{s,1} \\
\none[\vdots]  \\
\end{ytableau}
\qquad
\caption{The injective map $\sigma_1$.}
\label{n-22sigma1}  
\end{figure}

\textbf{The map $\sigma_2$.} Let $T\in D_2$ with first row $a_{1,1}\prec a_{1,2}\prec a_{1,3}=n$ with $a_{1,1}\in\{1,2\}, a_{1,2}\geq 3$, and singletons $a_{2,1},a_{3,1},\dots,a_{n-2,1}$. We will search for an index $s\in [4,n-2]$ in $T$ such that \[a_{3,1},a_{4,1},\dots,a_{s-1,1}\prec a_{s,1}\nprec n,\] and we discuss two main cases below depending on the existence of the index $s$.

\begin{itemize}
    \item Case 1: There does not exist an index $s$.

    \textbf{1a:} If $a_{2,1}\prec n$, then define $\sigma_2(T)$ by moving $n$ to the position $(2,2)$. See Figure \ref{sigma2case1ab} (left). Clearly $\iv(\sigma_2(T))=\iv(T)$.

    \textbf{1b:} If $a_{2,1}\nprec n$, then $a_{1,2}<a_{2,1}$ by (F1), and thus $a_{1,1}\prec a_{2,1}$. Also, we have $a_{2,1}\geq 3$ ($a_{2,1}\notin \{1,2\}$ since $a_{2,1}\nprec n$). So $a_{1,2}\nprec a_{2,1}$ by (F3). Define $\sigma_2(T)$ as shown in Figure \ref{sigma2case1ab} (right). All column conditions hold so that $\sigma_2(T)\in E$, and $\iv(T)$ is preserved by removing the inversion $(n, a_{2,1})$ and adding the inversion $(a_{2,1}, a_{1,2})$.


\begin{figure}[h]
   \centering
   \ytableausetup{mathmode, boxsize=2em}
\begin{ytableau}
\scriptstyle a_{1,1} & \scriptstyle a_{1,2}& n \\
\scriptstyle a_{2,1}\\
\scriptstyle a_{3,1} \\
\scriptstyle a_{4,1} \\
\none[\vdots]  \\
\end{ytableau}\,\,\,
$\longrightarrow\,\,\,$
\begin{ytableau}
\scriptstyle a_{1,1} & \scriptstyle a_{1,2} \\
\scriptstyle a_{2,1} & n \\
\scriptstyle a_{3,1} \\
\scriptstyle a_{4,1} \\
\none[\vdots]  \\
\end{ytableau}
\qquad
\hspace{10mm}
\begin{ytableau}
\scriptstyle a_{1,1} & \scriptstyle a_{1,2}& n \\
\scriptstyle a_{2,1}\\
\scriptstyle a_{3,1} \\
\scriptstyle a_{4,1} \\
\none[\vdots]  \\
\end{ytableau}\,\,\,
$\longrightarrow\,\,\,$
\begin{ytableau}
\scriptstyle a_{1,1} & \scriptstyle a_{2,1} \\
\scriptstyle a_{1,2} & n \\
\scriptstyle a_{3,1} \\
\scriptstyle a_{4,1} \\
\none[\vdots]  \\
\end{ytableau}
\qquad
\caption{The injective map $\sigma_2:$ Case 1a (left) and Case 1b (right).}
\label{sigma2case1ab}  
\end{figure}

\item Case 2: There exists an index $s$.

In this case $a_{3,1},\cdots,a_{s-1,1}\prec a_{s,1}\nprec n$ forces $a_{3,1},\cdots,a_{s-1,1}\in\{1,2\}$. Since $a_{1,1}\in \{1,2\}$, we have $s=4$, and $(a_{2,1}, a_{3,1})$ is an inversion of $T$. Thus $a_{1,2}, a_{2,1}<a_{4,1}$ by (F1).

\textbf{2a:} If $a_{3,1}\prec a_{1,2}$, then $a_{2,1}<a_{1,2}$ by (F1). Define $\sigma_2(T)$
as shown in Figure \ref{sigma2case2a} (left). The inversions $(a_{2,1},a_{3,1})$ and $(a_{1,2},a_{2,1})$ of $T$ disappear, while the inversions $(a_{4,1},a_{2,1})$ and $(a_{4,1},a_{1,2})$ appear in $\sigma_2(T)$, hence $\iv(\sigma_2(T))=\iv(T)$. 

\textbf{2b:} If $a_{3,1}\nprec a_{1,2}$, then define $\sigma_2(T)$ as shown in Figure \ref{sigma2case2a} (right). The inversions $(n,a_{4,1})$ and $(a_{2,1},a_{3,1})$ of $T$ disappear while the inversions $(a_{4,1},a_{1,2})$ and $(a_{4,1},a_{2,1})$ appear in $\sigma_2(T)$. Again $\iv(\sigma_2(T))=\iv(T)$.

\begin{figure}[h]
   \centering
   \ytableausetup{mathmode, boxsize=2em}
\begin{ytableau}
\scriptstyle a_{1,1} & \scriptstyle a_{1,2} & n \\
\scriptstyle a_{2,1}\\
\scriptstyle a_{3,1} \\
\scriptstyle a_{4,1} \\
\none[\vdots]  \\
\end{ytableau}\,\,\,
$\longrightarrow\,\,\,$
\begin{ytableau}
\scriptstyle a_{1,1} & n \\
\scriptstyle a_{3,1} & \scriptstyle a_{4,1} \\
\scriptstyle a_{2,1} \\
\scriptstyle a_{1,2} \\
\none[\vdots]  \\
\end{ytableau}
\qquad
\hspace{10mm}
\begin{ytableau}
\scriptstyle a_{1,1} & \scriptstyle a_{1,2}& n \\
\scriptstyle a_{2,1}\\
\scriptstyle a_{3,1} \\
\scriptstyle a_{4,1} \\
\none[\vdots]  \\
\end{ytableau}\,\,\,
$\longrightarrow\,\,\,$
\begin{ytableau}
\scriptstyle a_{1,1} & \scriptstyle a_{4,1} \\
\scriptstyle a_{1,2} & n \\
\scriptstyle a_{3,1} \\
\scriptstyle a_{2,1} \\
\none[\vdots]  \\
\end{ytableau}
\qquad
\caption{The injective map $\sigma_2:$ Case $2a$ (left) and Case 2b (right).}
\label{sigma2case2a}  
\end{figure}

\end{itemize}

Let $U=\sigma_2(T)$ with rows $u_{1,1}\prec u_{1,2}$, $u_{2,1}\prec u_{2,2}$ and singletons $u_3,\cdots, u_{n-2}$. We read off the case from $U$: $u_{1,2}\prec u_{2,2}$ holds in Case 1a and fails in all other cases. Otherwise, $u_{2,1}\nprec u_{1,2}$ holds exactly in Cases 1b, and 2b, and these two can be distinguished by whether $u_3\prec u_{1,2}$. Once the case is identified, $T$ is recovered from $U$ uniquely by reversing all movements applied on $T$. Therefore $\sigma_2$ is injective.

Finally, we need to check the disjointness of $\sigma_1$ and $\sigma_2$. Suppose that $U=\sigma_2(T)$ also lies in $\im(\sigma_1)$. Let $t\geq4$ be an index such that $u_t$ is reachable from $u_3$ and $u_t\nprec u_{2,2}$. i.e., $U\in \im (\sigma_1)$. In map $\sigma_2$ Case 1a and 1b, $t$ provides the index $s$ for $\sigma_2$ since $u_{2,2}=n$, a contradiction. In Case 2a, $u_3=a_{2,1}\geq3$, so $u_3\prec u_t$ forces $u_t=n$ but $n=u_{1,2}$ is not in the column. In Case 2b, $u_3=a_{3,1}, u_4=a_{2,1}$, where $t=4$ would give $a_{3,1}\prec a_{2,1}$ contradicting to the column condition, and $t\geq5$ would give $a_{2,1}\prec u_t=n$ which is not in the column again. Hence $\im(\sigma_1)\cap\im(\sigma_2)=\emptyset$, and the lemma follows.
\end{proof}

\subsection{The case $d_4=n$}\label{subsecfamIII}

In this subsection $\mathbf{d}=(d_1,d_2,d_3,n,\dots,n)$ and the set $\{4,\cdots,n\}$ is an antichain. Moreover, since $r=3$, $d_1\leq 2$. We assume $G$ is connected so $d_1=2$ and $d_2\geq 3$, and then $S_1=\{1,2\}$. We have the following fact discussing the basics for $P$-tableaux in this context.

(F4) If $\sh(T)=3^12^11^{n-5}$, then the first two rows of $T$ are $1\prec 3\prec z$ and $2\prec v$ with $z\in S_3, v\in S_2\cup S_3$. If $\sh(T)=2^31^{n-6}$, then $a_{1,1}, a_{2,1}, a_{3,1}\in \{1,2,3\}$. If $\sh(T)=3^11^{n-3}$, then the first row of $T$ is $1\prec 3\prec z$ with $z\in S_3$.

\begin{lemma}\label{lemman-3III}
   Let $\mathbf{d}$ be as above with $r=3, d_1=2, d_2\geq 3, n\geq6$. Then $B(0,3)-B(1,1)\in\bN[q]$. 
\end{lemma}

\begin{proof}
  Let $D$ and $E$ be the sets of $P$-tableaux of shapes $3^12^11^{n-5}$ and $2^31^{n-6}$. By (F4), we have rows 1 and 2 of $T\in D$ are $1\prec 3\prec z, 2\prec v$. Let $c=a_{3,1}$. Now we define $\psi(T)\in E$ in three cases (See Figure \ref{figpsiIII}).

  \begin{itemize}
      \item Case 1: If $2\nprec c$, the first three rows of $\psi(T)$ are $1\prec c$, $3\prec z$, $2\prec v$.
      \item Case 2: If $2\prec c$ and $3\nprec v$, the first three rows of $\psi(T)$ are $1\prec v$, $3\prec z$, $2\prec c$.
      \item Case 3: If $2\prec c$ and $3\prec v$, the first three rows of $\psi(T)$ are $1\prec z$, $3\prec v$, $2\prec c$.
  \end{itemize}

It is straightforward to verify all row and column conditions for $\psi(T)\in E$. Also $\iv(\psi(T))=\iv(T)$. Specifically, inversions $(z,c), (v,c)$ are removed and inversions $(c,2), (c,3)$ are added in Case 1. Inversion $(z,v)$ is removed and inversion $(v,3)$ is added in Case 2. Incomparable pairs do not change their relative positions in Case 3. Further, $\psi$ is injective. The rows of $U=\psi(T)$ are $1\prec u_1, 3\prec u_2, 2\prec u_3$, but $2\nprec u_1$ in Case 1, $2\prec u_1, 3\nprec u_1$ in Case 2, $3\prec u_1$ in Case 3. Thus $T$ can be uniquely recovered from $U$, and hence $B(0,3)-B(1,1)=\sum_{U\in E\setminus\psi(D)}q^{\iv(U)}\in\bN[q]$.
\end{proof}

\begin{figure}[h]
   \centering
   \ytableausetup{mathmode, boxsize=2em}
\begin{ytableau}
1 & 3 & z \\
2 & v \\
c \\
\none[\vdots]  \\
\end{ytableau}\,\,\,
$\longrightarrow\,\,\,$
\begin{ytableau}
1 & c \\
3 & z \\
2 & v \\
\none[\vdots]  \\
\end{ytableau}
\quad
\begin{ytableau}
1 & v \\
3 & z \\
2 & c \\
\none[\vdots]  \\
\end{ytableau}
\quad
\begin{ytableau}
1 & z \\
3 & v \\
2 & c \\
\none[\vdots]  \\
\end{ytableau}
\caption{The injection $\psi$ of Lemma \ref{lemman-3III}. The images from left to right are Case 1: $2\nprec c$, Case 2: $2\prec c$ and $3\nprec v$, Case 3: $2\prec c$ and $3\prec v$.}
\label{figpsiIII}
\end{figure}

\begin{lemma}\label{lemman-2III}
Let $\mathbf{d}$ be as above with $r=3, d_1=2, d_2\geq 3, n\geq6$. Then $B(0,2)-B(0,3)-B(1,0)\in\bN[q]$.
\end{lemma}

\begin{proof}
 Let $E, D_1, D_2$ be the sets of $P$-tableaux of shapes $2^21^{n-4}$, $2^31^{n-6}$ and $3^11^{n-3}$ respectively. As in Lemma \ref{lemman-2}, we construct inversion-preserving injections $\sigma_1:D_1\to E$ and $\sigma_2:D_2\to E$ with disjoint images.

\textbf{The map $\sigma_1.$}
Let $T\in D_1$ with rows $a_{1,1}\prec a_{1,2}, a_{2,1}\prec a_{2,2}, a_{3,1}\prec a_{3,2}$. Define $\sigma_1(T)$ as follows.

\begin{itemize}
    \item Case 1: If $a_{1,1}=1, a_{2,1}=3, a_{3,1}=2, 2\nprec a_{1,2}$ and $3\nprec a_{3,2}$, the rows of $\sigma_1(T)$ are $1\prec a_{2,2}, 2\prec a_{3,2}, 3, a_{1,2},\cdots$. See Figure \ref{figsigma1III} (left).

    \item Case 2: If $a_{1,1}=1, a_{2,1}=3, a_{3,1}=2, 2\nprec a_{1,2}$ and $3\prec a_{3,2}$, the rows of $\sigma_1(T)$ are $1\prec a_{2,2}, 2\prec a_{3,2}, a_{1,2}, 3,\cdots$. See Figure \ref{figsigma1III} (left).

    \item Case 3: Otherwise, the rows of $\sigma_1(T)$ are $a_{1,1}\prec a_{1,2}, a_{2,1}\prec a_{2,2}$ and insert $a_{3,2}$ above $a_{4,1}$. See Figure \ref{figsigma1III} (right).
\end{itemize}

\begin{figure}[h]
   \centering
   \ytableausetup{mathmode, boxsize=2em}
\begin{ytableau}
1 & \scriptstyle a_{1,2} \\
3 & \scriptstyle a_{2,2} \\
2 & \scriptstyle a_{3,2} \\
\scriptstyle a_{4,1} \\
\none[\vdots]  \\
\end{ytableau}\,\,\,
$\longrightarrow\,\,\,$
\begin{ytableau}
1 & \scriptstyle a_{2,2} \\
2 & \scriptstyle a_{3,2} \\
3 \\
\scriptstyle a_{1,2} \\
\scriptstyle a_{4,1} \\
\none[\vdots]  \\
\end{ytableau}
\begin{ytableau}
1 & \scriptstyle a_{2,2} \\
2 & \scriptstyle a_{3,2} \\
\scriptstyle a_{1,2} \\
3 \\
\scriptstyle a_{4,1} \\
\none[\vdots]  \\
\end{ytableau}
\qquad
\hspace{8mm}
\begin{ytableau}
\scriptstyle a_{1,1} & \scriptstyle a_{1,2} \\
\scriptstyle a_{2,1} & \scriptstyle a_{2,2} \\
\scriptstyle a_{3,1} & \scriptstyle a_{3,2} \\
\scriptstyle a_{4,1} \\
\none[\vdots]  \\
\end{ytableau}\,\,\,
$\longrightarrow\,\,\,$
\begin{ytableau}
\scriptstyle a_{1,1} & \scriptstyle a_{1,2} \\
\scriptstyle a_{2,1} & \scriptstyle a_{2,2} \\
\scriptstyle a_{3,1} \\
\scriptstyle a_{3,2} \\
\scriptstyle a_{4,1} \\
\none[\vdots]  \\
\end{ytableau}
\quad
\caption{The injection $\sigma_1$ of Lemma \ref{lemman-2III}. Left: Case 1 and Case 2 with images for $3\nprec a_{3,2}$ and $3\prec a_{3,2}$. Right: Case 3.}
\label{figsigma1III}
\end{figure}

To verify that $\iv(\sigma_1(T))=\iv(T)$, in Case 1, the inversions $(3,2), (a_{1,2},2), (a_{1,2},3)$ are removed but inversions $(a_{2,2}, a_{1,2}), (a_{3,2}, a_{1,2}), (a_{3,2}, 3)$ are added. In Case 2, the inversions $(3,2), (a_{1,2},2)$ are removed but inversions $(a_{2,2}, a_{1,2}), (a_{3,2}, a_{1,2})$ are added. In Case 3, no incomparable pairs change.

Let $U\in E$ with rows $u_{1,1}\prec u_{1,2},u_{2,1}\prec u_{2,2}$ and singletons $u_3,u_4,\cdots$. Clearly, the images in Case 3 have unique preimages if we move $u_4$ back to position $(3,2)$. The images of Case 1 and 2 can also be uniquely recovered and they are distinguished from images of Case 3 since $u_3\nprec u_4$. Thus $\sigma_1$ is injective.

\textbf{The map $\sigma_2$.}
Let $T\in D_2$ with first row $1\prec 3\prec z$ for $z\in S_3$, and all singletons $c_1,c_2,\dots,c_{n-3}$. We define $\sigma_2(T)$ by rearranging the entries $1,3,z,c_1$ of the first two rows, and keeping the rows of index at least three unchanged (see Figure \ref{figsigma2III}).
\begin{itemize}
    \item Case 1: If $c_1=2$, then simply move $z$ to position $(2,2)$.
    \item Case 2: If $c_1\geq4$ and $3\nprec c_1$, then first two rows of $\sigma_2(T)$ are $1\prec c_1$ and $3\prec z$.
\item Case 3: If $3\prec c_1$, the first two rows of $\sigma_2(T)$ are $1\prec z$ and $3\prec c_1$.
\end{itemize}

\begin{figure}[h]
   \centering
   \ytableausetup{mathmode, boxsize=2em}
\begin{ytableau}
1 & 3 & z \\
c_1 \\
c_2 \\
\none[\vdots]  \\
\end{ytableau}\,\,\,
$\longrightarrow\,\,\,$
\begin{ytableau}
1 & 3 \\
2 & z \\
c_2 \\
\none[\vdots]  \\
\end{ytableau}
\quad
\begin{ytableau}
1 & c_1 \\
3 & z \\
c_2 \\
\none[\vdots]  \\
\end{ytableau}
\quad
\begin{ytableau}
1 & z \\
3 & c_1 \\
c_2 \\
\none[\vdots]  \\
\end{ytableau}
\caption{The injection $\sigma_2$ of Lemma \ref{lemman-2III}. The images from left to right are Case 1: $c_1=2$, Case 2: $c_1\geq 4$ and $3\nprec c_1$, Case 3: $3\prec c_1$.}
\label{figsigma2III}
\end{figure}

Since $c_1$ and $z$ are incomparable in Case 2 and 3 ($c_1=2\prec z$ in Case 1), all row and column conditions hold so that $\sigma_2(T)\in E$ in each case. In Case 1, no inversion changes at all, and in Case 2, $3<c_1<z$ by (F1) so that the inversion $(z,c_1)$ disappears while $(c_1,3)$ appears. In Case 3, $1,3\prec c_1,z$ so that the inversion does not change at all. Thus $\iv(\sigma_2(T))=\iv(T)$. To see that $\sigma_2$ is injective, we can again read off from the images and recover a unique $T$. The second entry of the first row is $3$ only in Case 1. To distinguish Case 2 and 3, note that $3$ is comparable to the second entry of the first row in only Case 3.

Finally, we compare $\im(\sigma_1)$ and $\im(\sigma_2)$. Case 1 and 2 of $\im(\sigma_1)$ have entries $1,2,3$ increasingly ordered in the first column, which are obviously disjoint from $\im(\sigma_2)$. Case 3 of $\im(\sigma_1)$ has $1,2,3$ (in any order) as the first three entries of the first column., and $2\prec a_{1,2}$ whenever $(a_{1,1},a_{2,1},a_{3,1})=(1,3,2)$. This is disjoint from all cases of $\im(\sigma_2)$. In Case 1, entry $3$ is in the first row. In Case 2, $c_2=2\nprec c_1$ because of the column condition of $T$. In Case 3, $c_2=2<3\prec c_1$ again contradicts to the column condition of $T$. Hence $\im(\sigma_1)\cap\im(\sigma_2)=\emptyset$, and $B(0,2)-B(0,3)-B(1,0)=\sum_Uq^{\iv(U)}\in \bN[q]$ over $U\in E\setminus(\im(\sigma_1)\cup\im(\sigma_2))$.
\end{proof}

\subsection{Proof of Theorem \ref{epos1}}\label{subsecproof}
\begin{proof}[Proof of Theorem \ref{epos1}]
Now we can complete our proof of Theorem \ref{epos1} by collecting all results in this section. Due to Theorem \ref{e_{(n-2l,l,l)}} and Theorem \ref{e_{(n-2l-1,l+1,l)}}, it suffices to prove (\ref{L1L2}). We showed (\ref{L1L2}) using the equivalent conditions in Proposition \ref{equvcondition} ($d_i\geq n-4+i$ for $i\in\{1,2,3,4\}$). If $d_3\geq n-1$, (\ref{L1L2}) is due to Lemma \ref{lemman-3} and Lemma \ref{lemman-2}. If $d_4=n$, it is due to Lemma \ref{lemman-3III} and Lemma \ref{lemman-2III}. If $d_2\geq n-2$ or $d_1\geq n-3$, we obtain the transposed Dyck path of the other cases, which are proved by Proposition \ref{transpose}. 
\end{proof}

\section{Discussions}\label{section5}
 
We close with some discussions on what remains open for Dyck paths of bounce number three, and how the results of this paper relate to recent developments.

\subsection{The remaining coefficients for bounce number three}\label{open1}
 
By Proposition \ref{ecoeff}, every coefficient $[e_\lambda]X_G(\mathbf{x},q)$ with $\lambda=(n-2\ell-j,\ell+j,\ell)$ is an alternating sum of at most seven of the polynomials $B(\ell',j')$. Theorems \ref{e_{(n-2l,l,l)}} and \ref{e_{(n-2l-1,l+1,l)}} show that the coefficients with $\lambda_2-\lambda_3\leq1 (j\in\{0,1\})$ lie in $\bN[q]$. The special case when $\ell=0, j=2$ was covered by Siegl's result \cite{Siegl2025} on the shape $(k,2^a)$. In general, a sign reversing injection is needed to map the $P$-tableaux of negative shapes into the $P$-tableaux of positive shapes.

\subsection{Combinatorial interpretations of the coefficients}
 
After the first version of this paper was written, Hikita \cite{Hikita2024} proved Conjecture \ref{3+1epositive} by giving an explicit $e$-expansion of the chromatic quasisymmetric function of a unit interval graph with coefficients the sums of rational functions of $q$ which are nonnegative for every real $q>0$. Conjecture \ref{SWepositive} remains open and so does the question on finding a combinatorial interpretation of the coefficients $[e_\lambda]X_G(\mathbf{x},q)$. The sign reversing injections of this paper provide a combinatorial interpretation of the coefficients
\[[e_\lambda]X_G(\mathbf{x},q)=\sum_{T\in\mathcal{T}_\lambda}q^{\iv(T)},\]
where $\mathcal{T}_\lambda$ is a set of $P$-tableaux with explicit description. Recently Siegl \cite{Siegl2025} introduced the \emph{strong} and \emph{powerful} $P$-tableaux and conjectured that the strong $P$-tableaux of shape $\lambda$ give a lower bound for $[e_\lambda]X_G(\mathbf{x},q)$ for every natural unit interval order. The equality was proved in several cases. Hwang \cite{Hwang2022} gave a combinatorial interpretation for hook shapes and two column shapes using \emph{heaps}, which are the acyclic orientations of $G$ counted by ascents. Blasiak, Eriksson, Pylyavskyy and Siegl \cite{BEPS2025} introduced the \emph{key} $P$-tableaux for hook shapes and the \emph{left} $P$-tableaux for two column shapes. It would be interesting to compare these sets with the sets $\mathcal{T}_\lambda$ here.

\subsection{Coefficients of length two for an arbitrary bounce number}
 
The inverse Kostka expansion of Section \ref{section3} is not restricted to $r=3$. For a Dyck path of arbitrary bounce number $r$, Equation \eqref{inv kostka expand} expands $X_G(\mathbf{x},q)$ into the polynomials $B_\mu(q)$ with $\mu_1\leq r$ and the special rim hook tabloids of shape $\mu'$. For a partition $\lambda=(n-i,i)$, every special rim hook tabloid of type $\lambda$ consists of two special rim hooks with sizes $n-i$ and $i$ (see Figure \ref{2ribbons}). The shape $\mu'$ of a tabloid of type $(n-i,i)$ has $\mu'_3\leq 2$. Thus $K\inv_{\lambda,\mu'}\neq0$ only if $\mu_3\leq2$, and $[e_{(n-i,i)}]X_G(\mathbf{x},q)$ is an alternating sum of polynomials $B_\mu(q)$ with $\mu=(\mu_1,\mu_2,2^a,1^b)$ and $\mu_1\leq r$. By Remark \ref{knowncases}, the coefficients of $e_{(n-1,1)}$ (included in hook shapes) and $e_{(n-2,2)}$ (included in shapes of the form $(k,2^a)$) lie in $\bN[q]$ for every natural unit interval order. We are not aware of the results for $i\geq3$ for an arbitrary $r$. At $q=1$, the positivity of all coefficients $[e_{(n-i,i)}]X_G(\mathbf{x},1)$ was proved by Abreu--Nigro \cite{AbreuNigro2023}, and of course, it also follows from Hikita \cite{Hikita2024}. For a general $q$ and an arbitrary bounce number $r$, it is natural to ask for a sign reversing injection that absorbs all negative $P$-tableaux into positive ones. For $r\geq 4$, longer chains appear in the first column, so that the entry inserted as described in Section \ref{section3} is no longer unique, thus the maps have to be refined.

\begin{figure}[h]
    \centering
    \vspace{3mm}
\begin{tikzpicture}
\draw[black] (0,0) -- (.3,0) -- (.3,2) -- (2.5,2) -- (2.5,2.3)-- (0,2.3) -- cycle;
\draw[black] (0.3,0.6) -- (.6,.6) -- (.6,1.7) -- (1.9,1.7) -- (1.9,2);
\draw[blue, very thick] (0.15,0.15) -- (.15,.75) -- (.45,.75) -- (.45,1.85) -- (1.75,1.85)-- (1.75,2.15) -- (2.35,2.15);
\draw[blue, very thick] (0.15,0.9) -- (.15,2.15) -- (1.6,2.15);
\end{tikzpicture}
\qquad
\caption{A special rim hook tabloid of type $(n-i,i)$.}
\label{2ribbons}
\end{figure}

\section{Acknowledgements} 

The author thanks her advisor, Greta Panova, for her guidance and helpful discussions on this project.

\printbibliography
\end{document}